\documentclass[reqno]{amsart}
\usepackage{amsthm, amssymb, amsfonts}
\usepackage{lmodern}
\usepackage{color}
\usepackage{hyperref}
\usepackage[T5]{fontenc}
\usepackage{amscd,amssymb}
\usepackage[v2,cmtip]{xy}
\usepackage{mathrsfs}
\theoremstyle{plain}
\newtheorem{thm}{Theorem}[section]
\newtheorem{prop}[thm]{Proposition}
\newtheorem{lem}[thm]{Lemma}

\theoremstyle{definition}
\newtheorem{defn}[thm]{Definition}

\newtheorem{nota}[thm]{Notation}

\theoremstyle{plain}
\newtheorem{thms}{Theorem}[subsection]
\newtheorem{props}[thms]{Proposition}
\newtheorem{lems}[thms]{Lemma}
\newtheorem{corls}[thms]{Corollary}
\theoremstyle{definition}

\allowdisplaybreaks

\begin{document} 
\title[A minimal set of generators for the polynomial algebra]
{A minimal set of generators for the polynomial algebra of five variables in a generic degree}

\author{Nguy\~\ecircumflex n Sum, Ph\d am \DJ\~\ocircumflex\ T\`ai }
\address{Department of Mathematics and Applications, S\`ai G\`on University, 273 An D\uhorn \ohorn ng V\uhorn \ohorn ng, District 5, H\`\ocircumflex\ Ch\'i Minh city, Viet Nam}
 
\email{nguyensum@sgu.edu.vn; phamdotai9a13@gmail.com}


\bigskip
\begin{abstract}
Let $P_k$ be the graded polynomial algebra $\mathbb F_2[x_1,x_2,\ldots ,x_k]$ over the prime field with two elements, $\mathbb F_2$, with the degree of each $x_i$ being 1. We study the {\it hit problem}, set up by Frank Peterson, of finding a minimal set of generators for $P_k$ as a module over the mod-$2$ Steenrod algebra, $\mathcal{A}.$ It is an open problem in Algebraic Topology. In this paper, we explicitly determine a minimal set of $\mathcal{A}$-generators for $P_5$ in the case of the generic degree $m = 2^{d}$ for all $d \geqslant 8$.
\end{abstract}
\maketitle

\section{Introduction}\label{s1} 

Let $V_k$ be an elementary abelian 2-group of rank $k$ and let $BV_k$ be the classifying space of $V_k$.  Then, 
$$P_k:= H^*(BV_k) \cong \mathbb F_2[x_1,x_2,\ldots ,x_k],$$ a polynomial algebra in  $k$ variables $x_1, x_2, \ldots , x_k$, each of degree 1. Here, the cohomology is taken with coefficients in the prime field $\mathbb F_2$ with two elements. The algebra $P_k$ is a module over the mod-2 Steenrod algebra, $\mathcal A$.  The action of $\mathcal A$ on $P_k$ is determined by the formula
$$Sq^t(x_i) = \begin{cases} x_i, &t=0,\\ x_i^2, &t=1,\\ 0, &\text{otherwise,}
\end{cases}$$
and subject to the Cartan formula
$$Sq^m(uv) = \sum_{i=0}^mSq^i(u)Sq^{m-i}(v),$$
for $u,\, v \in P_k$ (see Steenrod and Epstein~\cite{st}).

A polynomial $g$ in $P_k$ is called {\it hit} if it can be written as a finite sum $g = \sum_{i\geqslant 0}Sq^{i}(\ell_i)$ for some polynomials $\ell_i \in P_k$. That means $g$ belongs to  $\mathcal{A}^+P_k$, where $\mathcal{A}^+$ is the augmentation ideal of $\mathcal A$.

We are interested in the {\it Peterson hit problem} of determining a minimal set of generators for the polynomial algebra $P_k$ as a module over the Steenrod algebra, $\mathcal A$. That means, we need to determine a basis of the $\mathbb F_2$-vector space 
$$QP_k := P_k/\mathcal A^+P_k = \mathbb F_2 \otimes_{\mathcal A} P_k.$$ 

This problem was first studied by Peterson~\cite{pe}, Wood~\cite{wo}, Singer~\cite {si1}, and Priddy~\cite{pr}, who showed its relation to several classical problems in Algebraic Topology. Then, this problem  was investigated by Carlisle and Wood~\cite{cw}, Crabb and Hubbuck~\cite{ch}, Janfada and Wood~\cite{jw1}, Ly and T\'in \cite{Tin20}, Kameko~\cite{ka}, Mothebe \cite{mo}, Nam~\cite{na}, Ph\'uc and Sum \cite{sp2}, Repka and Selick~\cite{res}, Silverman~\cite{sl,sl2}, Silverman and Singer~\cite{ss}, Singer~\cite{si2}, Sum and T\'in \cite{su5}, T\'in \cite{tin21}, Walker and Wood~\cite{wa3}, Wood~\cite{wo2}, the first named author \cite{su1,su2,su4} and others. 

The hit problem was explicitly computed by 
Peterson~\cite{pe} for $k=1, 2,$ by Kameko~\cite{ka} for $k=3$ and  by the first named author \cite{su2} for $k = 4$. This problem is open for any $k > 4$. Recently, the results on the hit problem and its applications to representations of general linear groups have been presented in the books of Walker and Wood \cite{wa1,wa2}.

For any nonnegative integer $m$, Denote by $(P_k)_m$ and  $(QP_k)_m$ the subspaces of degree $m$ homogeneous polynomials in the spaces $P_k$ and $QP_k$ respectively. Set $\mu(m) = \min\{u \in \mathbb Z : \alpha (m+u)\leqslant u\}$, where $\alpha (a)$ denotes the number of one in dyadic expansion of a positive integer $a$. 

By combining the results in Kameko \cite{ka} and Wood \cite{wo} the hit problem is reduced to the case of degree $m$ such that $\mu(m) < k$. For $\mu(m)=k-1$, the problem was studied by Crabb-Hubbuck~\cite{ch}, Nam~\cite{na}, Repka-Selick~\cite{res}, Walker-Wood \cite{wa3} and the first named author ~\cite{su1,su2}. For $\mu(m) = k-2$, it was studied in \cite{su4} which provides a new tool for studying the Peterson hit problem. 

In this paper, we compute the hit problem for the case $k=5$ and the degrees $m = 2^{d}$ with $d \geqslant 8$ in terms of the admissible monomials (see Section \ref{s2}). We prove the following. 

\begin{thm}\label{thm1}
For any integer $d \geqslant 8$, there exist exactly $1984$	
admissible monomials of degree $2^{d}$ in $P_5$. Consequently $\dim (QP_5)_{2^{d}} = 1984.$
\end{thm}
For the cases $d = 1,\, 2,\, 3$, the computation is easy and we have $\dim (QP_5)_{2} = 10$, $\dim (QP_5)_{4} = 45$ and $\dim (QP_5)_{8} = 174$. The cases $d = 4,\, 5$ are respectively determined in our work \cite{su3} and Ph\'uc~\cite{ph21}, we showed that $\dim (QP_5)_{16} = 443$ and $\dim (QP_5)_{32} = 1004.$ We also study the problem for $d = 6,\, 7$ and we have $1679 \leqslant \dim (QP_5)_{64} \leqslant 1694$ and $1984 \leqslant \dim (QP_5)_{128} \leqslant 1990$ (see Section \ref{sect6}).  

\medskip
The paper is organized as follows. In Section \ref{s2}, we recall some needed information on the admissible monomials in $P_k$, the criteria of Singer and Silverman on the hit monomials. The detailed proof of Theorem \ref{thm1} is presented in Section ~ \ref{s3}. In Section \ref{sect5} we list the needed admissible monomials in $P_5$. Finally, in Section \ref{sect6} we present the computations of $(QP_5)_{2^d}$ for $5 \leqslant d \leqslant 7$.
 
\section{Preliminaries}\label{s2}
\setcounter{equation}{0}

In this section, we recall some needed information on the weight vector of a monomial and the admissible monomials from Kameko~\cite{ka}, Singer \cite{si2} and the first named author \cite{su2} which will be used in the next section. We refer the reader to \cite{su2} and \cite{su4} for the details of notions and needed results. 

Let $x=x_1^{a_1}x_2^{a_2}\ldots x_k^{a_k} \in P_k$. We denote $\nu_j(x) = a_j, 1 \leqslant j \leqslant k$ and $\nu(x) = \max\{\nu_j(x):1 \leqslant j \leqslant k\}$. We define two sequences associated with $x$ by
\begin{align*} 
	\omega(x)&=(\omega_1(x),\omega_2(x),\ldots , \omega_i(x), \ldots),\ \
	\sigma(x) = (\nu_1(x),\nu_2(x),\ldots ,\nu_k(x)),
\end{align*}
where
$\omega_i(x) = \sum_{1\leqslant j \leqslant k} \alpha_{i-1}(\nu_j(x)),\ i \geqslant 1.$
The sequences $\omega(x)$ and $\sigma(x)$ are respectively called the weight vector and the exponent vector of $x$. 

The sets of weight vectors and exponent vectors respectively are given the left lexicographical order.   

\begin{defn}\label{dfn2} Let $\omega$ be a weight vector and $g,\, h$ two polynomials  of the same degree in $P_k$. 
	
i) $g \equiv h$ if and only if $g+h \in \mathcal A^+P_k$. If $g$ is hit, then $g \equiv 0$.
	
ii) $g \equiv_{\omega} h$ if and only if $g + h \in \mathcal A^+P_k+P_k^-(\omega)$. 
\end{defn}

It is easy to see that, the relations $\equiv$ and $\equiv_{\omega}$ are equivalence ones. Denote   
$$QP_k(\omega)= P_k(\omega)/ ((\mathcal A^+P_k\cap P_k(\omega))+P_k^-(\omega)).$$  

For a  polynomial $g \in  P_k$, denote by $[g]$ the class in $QP_k$ represented by $g$. If  $\omega$ is a weight vector and $g \in  P_k(\omega)$, then we denote $[g]_\omega$ the class in $QP_k(\omega)$ represented by $g$. Denote by $|D|$ the cardinal of a set $D$.

From \cite{su3}, for any degree $m$, we have
\begin{equation}\label{ct2.1}
(QP_k)_m \cong \bigoplus_{\deg \omega = m}QP_k(\omega).
\end{equation}

\begin{defn}\label{defn3} 
Let $u, v$ be monomials in $P_k$ with $\deg u = \deg v$. We define $u < v$ if and only if one of the following conditions holds:  
	
i) $\omega (u) < \omega(v)$;
	
ii) $\omega (u) = \omega(v)$ and $\sigma(u) < \sigma(v).$
\end{defn}

\begin{defn}
A monomial $u$ in $P_k$ is said to be inadmissible if there are monomials $v_1,v_2,\ldots, v_r$ such that $v_j<u$ for $j=1,2,\ldots , r$ and $u+ \sum_{j=1}^rv_j \in \mathcal A^+P_k.$ 
A monomial $u$ is said to be admissible if it is not inadmissible.
\end{defn} 

Obviously, the set of all the admissible monomials of degree $m$ in $P_k$ is a minimal set of $\mathcal{A}$-generators for $P_k$ in degree $m$.

Denote by $\mathcal A_s$ the sub-Hopf algebra of $\mathcal A$ generated by  $Sq^i$ with $0\leqslant i \leqslant 2^s$, and $\mathcal A_s^+ = \mathcal A^+\cap\mathcal A_s$. 
 
\begin{defn} 
A monomial $u$ in $P_k$ is said to be strictly inadmissible if and only if there exist monomials $v_1,v_2,\ldots, v_r$ such that $v_j<u,$ for $j=1,2,\ldots , r$ and $u + \sum_{j=1}^r v_j\in \mathcal A_{s-1}^+P_k$ with $s = \max\{i : \omega_i(x) > 0\}$.
\end{defn}

\begin{thm}[See Kameko \cite{ka}, Sum \cite{su1}]\label{dlcb1}  
Let $u, v, w$ be monomials in $P_k$ such that $\omega_i(u) = 0$ for $i > r>0$, $\omega_s(w) \ne 0$ and   $\omega_i(w) = 0$ for $i > s > 0$.
	
{\rm i)}  If  $w$ is inadmissible, then so is $uw^{2^r}$.
	
{\rm ii)}  If $w$ is strictly inadmissible, then so is $wv^{2^{s}}$.
\end{thm} 

\begin{prop}[See {\cite[Proposition 4.3]{su2}}]\label{mdcb3} Let $u$ be an admissible monomial in $P_k$ and let $t_0$ be a positive integer. Then we have
	
{\rm i)} If $\omega_{t_0}(u)=0$, then $\omega_{t}(u)=0$ for all $t > t_0$.
	
{\rm ii)} If $\omega_{t_0}(u)<k$, then $\omega_{t}(u)<k$ for all $t > t_0$.
\end{prop}

Now, we recall a result of Singer \cite{si2} on the hit monomials in $P_k$. 

\begin{defn}\label{spi}  A monomial $z$ in $P_k$ is called a spike if $\nu_j(z)=2^{d_j}-1$ for $d_j$ a non-negative integer and $j=1,2, \ldots , k$. If $d_1>d_2>\ldots >d_{r-1}\geqslant d_r>0$ and $d_j=0$ for $j>r,$ then $z$ is called the minimal spike.
\end{defn}

In \cite{si2}, Singer showed that if $\mu(m) \leqslant k$, then there exists uniquely a minimal spike of degree $m$ in $P_k$. 

\begin{thm}[See Singer~\cite{si2}]\label{dlsig} Suppose $u \in P_k$ is a monomial of degree $m$, where $\mu(m) \leqslant k$. Let $z$ be the minimal spike of degree $m$. If $\omega(u) < \omega(z)$, then $u \equiv 0$.
\end{thm}

We set 
\begin{align*} 
P_k^0 &=\langle\{x\in P_k \ : \ \prod_{j=1}^k\nu_j(x)=0\}\rangle,\ \ 
P_k^+ = \langle\{x\in P_k \ : \ \prod_{j=1}^k\nu_j(x)>0\}\rangle. 
\end{align*}

It is easy to see that $P_k^0$ and $P_k^+$ are the $\mathcal{A}$-submodules of $P_k$ and 
$QP_k =QP_k^0 \oplus  QP_k^+.$
Here $QP_k^0 = P_k^0/\mathcal A^+P_k^0$ and  $QP_k^+ = P_k^+/\mathcal A^+P_k^+$.

We set $$\mathcal N_k =\{(i;I) : I=(i_1,i_2,\ldots,i_r),1 \leqslant  i < i_1 <  \ldots < i_r\leqslant  k,\ 0\leqslant r <k\}.$$ 
For each $(i;I) \in \mathcal N_k$, we define the homomorphism of algebras $p_{(i;I)}: P_k \to P_{k-1}$ by 
\begin{equation}\label{ct23}
p_{(i;I)}(x_j) =\begin{cases} x_j, &\text{ if } 1 \leqslant j < i,\\
\sum_{s\in I}x_{s-1}, &\text{ if }  j = i,\\  
x_{j-1},&\text{ if } i< j \leqslant k.
\end{cases}
\end{equation}
Then $p_{(i;I)}$ is a homomorphism of $\mathcal A$-algebras. 
\begin{lem}[Ph\'uc-Sum \cite{sp}]\label{bdm} For any monomial $u$ in $P_k$, we have $$p_{(i;I)}(u) \in P_{k-1}(\omega(u)).$$ 
\end{lem}

It is easy to see that if $\omega$ is a weight vector of degree $m$ and $u \in P_k(\omega)$, then $p_{(i;I)}(u) \in P_{k-1}(\omega)$. Moreover, $p_{(i;I)}$ passes to a homomorphism  
\begin{align*}
&p_{(i;I)}^{(\omega)} :QP_k(\omega)\longrightarrow QP_{k-1}(\omega),\ \
p_{(i;I)} :(QP_k)_m\longrightarrow (QP_{k-1})_m.
\end{align*} 
We set 
$$\widetilde {SF}_k(\omega) = \bigcap_{(i;I)\in \mathcal N_k}  \mbox{Ker}(p_{(i;I)}^{(\omega)}) \mbox{ and } \widetilde{QP}_k(\omega) = QP_k(\omega)/\widetilde {SF}_k(\omega).$$

We see that $\widetilde {SF}_k(\omega)$ is a subspace of the spike-free module $SF^{\omega}_k$ (see Walker and Wood \cite[Chap. 30]{wa2}). Then we have
$$QP_k(\omega) \cong \widetilde{QP}_k(\omega) \bigoplus \widetilde {SF}_k(\omega).$$

For $J= (j_1, j_2, \ldots, j_s) : 1 \leqslant j_1 <\ldots < j_s \leqslant k$, we define a monomorphism of $\mathcal A$-algebras $\theta_J: P_s \to P_k$ by substituting 
\begin{equation}\label{ctbs}
\theta_J(x_t) = x_{j_t} \ \mbox{ for } \ 1 \leqslant t \leqslant s.
\end{equation} 
Obviously, for any weight vector $\omega$ of degree $m$, 
\[Q\theta_J(P_s^+)(\omega) \cong  QP_s^+(\omega)\mbox{ and } (Q\theta_J(P_s^+))_m \cong (QP_s^+)_m\] 
for $1 \leqslant s \leqslant k$, where $Q\theta_J(P_s^+) = \theta_J(P_s^+)/\mathcal A^+\theta_J(P_s^+)$. Hence, by a simple computation using the result in Wood \cite{wo} and (\ref{ct2.1}), we get the following.
\begin{prop}[{See Walker and Wood~\cite{wa1}}]\label{mdbs} Let $\omega$ be a weight vector of degree $m$. Then, we have 
\begin{align*}
\dim QP_k(\omega) &= \sum_{\mu(m) \leqslant s\leqslant k}{k\choose s}\dim QP_s^+(\omega),\\
\dim (QP_k)_m &= \sum_{\mu(m) \leqslant s\leqslant k}{k\choose s}\dim (QP_s^+)_m.
\end{align*}
\end{prop}

Set $J_t = (1,\ldots,\hat t,\ldots,k)$ for $1 \leqslant t \leqslant k$.

\begin{prop}[See Mothebe and Uys \cite{mo}]\label{mdmo} Let $t, d$ be positive integers such that $1 \leqslant t \leqslant k$. If $u$ is an admissible monomial in $P_{k-1}$ then so is $x_t^{2^d-1}\theta_{J_t}(u)$  in $P_{k}$.
\end{prop}

\begin{nota} Denote by $B_{k}(m)$ the set of all admissible monomials of degree $m$  in $P_k$. We set 
$B_{k}^0(m) = B_{k}(m)\cap P_k^0,\ B_{k}^+(m) = B_{k}(m)\cap P_k^+.$
 
For any weight vector $\omega$ of degree $m$, we denote 
$$B_k(\omega) = B_{k}(m)\cap P_k(\omega),\ B_k^+(\omega) = B_{k}^+(m)\cap P_k(\omega).$$
Let $D$ be a subset in $P_k$. We denote $[D] = \{[f] : f \in D\}$. For $D \subset P_k(\omega)$, we set $[D]_\omega = \{[f]_\omega : f \in D\}$. 
Then, $[B_k(\omega)]_\omega$ and $[B_k^+(\omega)]_\omega$, are the basses of the $\mathbb F_2$-vector spaces $QP_k(\omega)$ and $QP_k^+(\omega) := QP_k(\omega)\cap QP_k^+$ respectively.
\end{nota}

\section{The indecomposables of $P_5$ in degree $2^{d}$}\label{s3}
\setcounter{equation}{0}

\subsection{The weight vector of admissible monomials of degree $2^d$}\

\medskip
First of all, we determine the weight vectors of the admissible monomials of degree $2^{d}$ for $d \geqslant 5$.
\begin{lems}\label{bdd5} If $x$ be an admissible monomial of degree $2^{d}$ in $P_5$ for $d \geqslant 5$, then either $\omega(x) = (2)|(1)|^{d-1}$ or $\omega(x) = (4)|(2)|^{d-2}$ or $\omega(x) = (4)|^2|(3)|^{d-4}|(1)$.
\end{lems}

\begin{proof} The lemma is proved by Ph\'uc \cite{ph21} for $d = 5$. Observe  that $z=x_1^{2^d-1}x_2$ is the minimal spike of degree $2^{d}$ in $P_5$ and $ \omega (z) = (2)|(1)|^{d-1}$. Since $2^{d}$ is even, using Theorem \ref{dlsig}, we obtain $\omega_1(x)=2$ or $\omega_1(x)=4$. 
	
If $\omega_1(x) = 2$, then $x=x_ix_jy^2$ with $1 \leqslant i < j \leqslant 5$ and $y$ an admissible monomial of degree $2^{d-1}-1$. By using Proposition \ref{mdcb3} and \cite[Lemma 4.1]{sux} we get either $\omega(y) = (1)|^{d-1}$ or $\omega(y) = (3)|(2)|^{d-3}$ or $\omega(y) = (3,4,3,1)$ for $d=6$. If $\omega(y) = (3,4,3,1)$, then $x = wu^8$ with $w$ a monomial of weight vector $(2,3,4)$ and $u$ a monomial of weight vector $(3,1)$. Since $\mu(24) = 4$, by Theorem \ref{dlsig} $w$ is hit, hence $x$ is inadmissible. If $\omega(y) = (3)|(2)|^{d-3}$, then $x = w_1u_1^{16}$, with $w_1$ a monomial of weight vector $(2,3,2,2)$ and $u_1$ a monomial of weight vector $(2)|^{d-5}$. From Ph\'uc~\cite{ph21} we see that $w_1$ is strictly inadmissible. By Theorem \ref{dlcb1}, $x$ is inadmissible. Thus, $\omega(y) = (1)|^{d-1}$ and $\omega(x) = (2)|(1)|^{d-1}$.

If $\omega_1(x) = 4$. Then $x = X_iy^2$ with $y$ an admissible monomial of degree $2^{d-1}-2$ in $P_5$ and $1 \leqslant i \leqslant 5$. From \cite[Lemma 3.1]{sux} and \cite{smo} we see that either $\omega(y) = (2)|^{d-2}$ or $\omega(y) = (4)|(2)|^{d-4}|(1)$ or  $y = x_1^3x_2^5x_3^6x_4^6x_5^{10}$ for $d=6$. By a direct computation we see that if $y = x_1^3x_2^5x_3^6x_4^6x_5^{10}$, then $x = X_iy^2$ is strictly inadmissible.  So, we obtain either $\omega(x) = (4)|(2)|^{d-2}$ or $\omega(y) = (4)|^2||(3)|^{d-4}|(1)$.	
 The lemma is proved.
\end{proof}

By Lemma \ref{bdd5}, we obtain
\begin{corls} For any integer $d\geqslant 5$, we have
\begin{align*}(QP_5)_{2^{d}} = (QP_5^0)_{2^{d}}&\oplus QP_5^+((2)|(1)|^{d-1})\\ &\oplus QP_5^+((4)|(2)|^{d-2})\oplus QP_5^+((4)|^2|(3)|^{d-4}|(1)).
\end{align*}
\end{corls}
	
\subsection{Computation of $QP_5((2)|(1)|^{d-1})$}\

\medskip
From Kameko \cite{ka} and our work \cite{su2}, for any $d \geqslant 5$, we have
 
\medskip
$B_2^+((2)|(1)|^{d-1}) = \left\{x_1x_2^{2^d-1},\, x_1^{2^d-1}x_2,\, x_1^3x_2^{2^d-3}\right\}$,

$B_3^+((2)|(1)|^{d-1}) = \left\{x_1x_2x_3^{2^d-2},\, x_1x_2^{2^d-2}x_3,\, x_1x_2^2x_3^{2^d-3},\, x_1x_2^3x_3^{2^d-4},\, x_1^3x_2x_3^{2^d-4}\right\}$,

$B_4^+((2)|(1)|^{d-1}) = \left\{u_t=u_{d,t} : 1 \leqslant t\leqslant 7\right\}$, where

\medskip 
\centerline{\begin{tabular}{lll}
$u_{1} = x_1x_2x_3^{2}x_4^{2^d-4}$ &$u_{2} = x_1x_2^{2}x_3x_4^{2^d-4}$ &$u_{3} = x_1x_2^{2}x_3^{2^d-4}x_4$\cr  
$u_{4} = x_1x_2^{2}x_3^{4}x_4^{2^d-7}$ &$u_{5} = x_1x_2^{2}x_3^{5}x_4^{2^d-8}$ &$u_{6} = x_1x_2^{3}x_3^{4}x_4^{2^d-8}$\cr 
$u_{7} = x_1^{3}x_2x_3^{4}x_4^{2^d-8}$. & &\cr 
\end{tabular}}

\medskip
Hence, we get 
\begin{align*}
&\dim QP_5^0((2)|^d) = 3{5\choose 2} + 5{5\choose 3} + 7{5\choose 4} = 115.
\end{align*}
So, we need to determine the space $QP_5^+((2)|(1)|^{d-1})$.
In this subsection we prove the following.
\begin{thms}\label{mdd51}
For any $d \geqslant 5$, $B_5^+((2)|(1)|^{d-1}) =\{a_{d,t}:\, 1 \leqslant t \leqslant 9\}$, where the monomials $a_t = a_{d,t}$ are determined as follows:

\medskip 
\centerline{\begin{tabular}{lll}
$a_{1} = x_1x_2x_3^{2}x_4^{4}x_5^{2^d-8}$ &$a_{2} = x_1x_2^{2}x_3x_4^{4}x_5^{2^d-8}$ &$a_{3} = x_1x_2^{2}x_3^{4}x_4x_5^{2^d-8}$\cr  
$a_{4} = x_1x_2^{2}x_3^{4}x_4^{2^d-8}x_5$ &$a_{5} = x_1x_2^{2}x_3^{4}x_4^{8}x_5^{2^d-15}$ &$a_{6} = x_1x_2^{2}x_3^{4}x_4^{9}x_5^{2^d-16}$\cr  
$a_{7} = x_1x_2^{2}x_3^{5}x_4^{8}x_5^{2^d-16}$ &$a_{8} = x_1x_2^{3}x_3^{4}x_4^{8}x_5^{2^d-16}$ &$a_{9} = x_1^{3}x_2x_3^{4}x_4^{8}x_5^{2^d-16}$.\cr 
\end{tabular}}

\medskip
Consequently, 
$$\dim QP_5^+((2)|(1)|^{d-1}) = 9,\ \dim QP_5((2)|(1)|^{d-1}) = 124.$$
\end{thms}
We need a lemma for the proof of the theorem.

\begin{lems}[See \cite{su3}]\label{bda61} Let $(i,j,t,u,v)$ be an arbitrary permutation of $(1,2,3,4,5)$. The following monomials are strictly inadmissible:
	
\smallskip
\ \! {\rm i)} $x_i^2x_jx_t,\, x_i^3x_j^{12}x_t,\, x_i^3x_j^5x_t^8,\, x_i^3x_j^4x_t^9,\ i<j<t.$

\smallskip
\ {\rm ii)}  $x_i^3x_j^4x_ix_u^8,\, x_i^3x_j^4x_i^8x_u,\ i < j < t < u$.
			
\end{lems}

\begin{proof}[Proof of Theorem $\ref{mdd51}$] Let $x$ be an admissible monomial of the weight vector $(2)|(1)|^{d-1}$ in $P_5^+$ with $d \geqslant 5$. Then $x = x_ix_jy^2$ with $1 \leqslant i < j \leqslant 5$ and $y$ an admissible monomial of the weight vector $(1)|^{d-1}$. By a direct computation we see that if $z$ is an admissible monomial of weight vector $(1)|^{d-1}$ and $x_ix_jz^2 \ne a_{d,t}$ for all $t,\, 1 \leqslant t \leqslant 9$, then there is a monomial $w$ as given in Lemma \ref{bda61} such that $x = wz_1^{2^{r}}$ with $r$ an integer, $2\leqslant r \leqslant 4$, and $z_1$ a suitable monomial of the weight vector $(1)|^{d-r}$. By Theorem \ref{dlcb1}, $x_ix_jz^2$ is inadmissible. Since $x = x_ix_jy^2$ is admissible, we have $x = a_{d,t}$ for some $t,\, 1 \leqslant t \leqslant 9$. Hence,
$B_5^+((2)|(1)|^{d-1}) \subset \{a_{d,t}: 1\leqslant t\leqslant 9\}.$
	
We now prove the set $\{[a_{d,t}]: 1\leqslant t\leqslant 9 \}$ is linearly independent in the vector space $QP_5((2)|(1)|^{d-1}) \subset QP_5$. Suppose there is a linear relation
\begin{equation}\label{ctd61}
\mathcal S:= \sum_{1\leqslant t \leqslant 9}\gamma_ta_{d,t} \equiv 0,
\end{equation}
where $\gamma_t \in \mathbb F_2$. We denote $\gamma_{\mathbb J} = \sum_{t \in \mathbb J}\gamma_t$ for any $\mathbb J \subset \{t\in \mathbb N:1\leqslant t \leqslant 9\}$.

Consider the homomorphism $p_{(i;I)}:P_5\to P_4$ defined by \eqref{ct23} for $k=5$. By applying $p_{(1;2)}$ to \eqref{ctd61}, we obtain 
\begin{align}\label{c41}
p_{(1;2)}(\mathcal S) \equiv \gamma_{\{3,4,5,6,7\}}&u_{ d,1} +  \gamma_{3}u_{ d,2} +   \gamma_{4}u_{ d,3} +  \gamma_{5}u_{ d,4}\notag\\ &+  \gamma_{6}u_{ d,5} +  \gamma_{7}u_{ d,6} +  \gamma_{\{2,3,4,5,6,7\}}u_{ d,7}  \equiv 0.
\end{align}
From \eqref{c41}, we get
\begin{equation}\label{c411}
\gamma_t = 0 \mbox{ for } t \in \{2,\, 3,\, 4,\, 5,\, 6,\, 7\}.
\end{equation}
Applying $p_{(1;3)}$ to \eqref{ctd61} and using \eqref{c411} we obtain
\begin{align}\label{c42}
p_{(1;3)}(\mathcal S) &\equiv \gamma_{8}u_{ d,1} +  \gamma_{\{1,8,9\}}u_{ d,6} +   \gamma_{8}u_{ d,7}  \equiv 0.
\end{align}
The relation \eqref{c42} implies
\begin{equation}\label{c421}
\gamma_8 = 0 \mbox{ and } \gamma_1 =\gamma_9.
\end{equation}
By a simple computation using \eqref{c411} and \eqref{c421} we obtain
\begin{align}\label{c43}
p_{(2;3)}(\mathcal S) &\equiv \gamma_{1}u_{ d,1} +  \gamma_{1}u_{ d,7} \equiv 0.
\end{align}
From \eqref{c43} it implies $\gamma_1 =0$. Hence, we get $\gamma_t =0$ for all $t, \ 1 \leqslant t \leqslant 9$. The proof is completed.
\end{proof}

\subsection{Computation of $QP_5^+((4)|(2)|^{d-2})$ for $d \geqslant 6$}\label{s32}\

\medskip
From our work \cite{su2}, we have $|B_4^+((4)|(2)|^{d-2})| = 35$ for any $d \geqslant 5$, hence we get $|B_5^0((4)|(3)|^{d-2})| = 35{5\choose 4} = 175$. In this subsection we prove the following theorem.

\begin{thms}\label{mdd52} For any $d \geqslant 6$, $B_5^+((4)|(2)|^{d-2})$ is the set of $290$ admissible monomials which are determined as in Subsection $\ref{s51}$. Consequently, 
$$\dim QP_5^+((4)|(2)|^{d-2}) = 290,\ \dim QP_5((4)|(2)|^{d-2}) = 465.$$
\end{thms}
The theorem is proved by using Theorems \ref{dlcb1}, \ref{dlsig} and the following lemmas.
\begin{lems}[See Ph\'uc \cite{ph21}]\label{bda62} Let $(i,j,t,u,v)$ be an arbitrary permutation of $(1,2,3,4,5)$. The following monomials are strictly inadmissible:
	
\smallskip
\ \! {\rm i)} $x_i^2x_jx_tx_ux_v^3,\, i<j<t<u$.
	
\smallskip
\ {\rm ii)} $x_r^{15}\theta_{J_r}(w),\, 1\leqslant r\leqslant 5,$ with $w$ one of the monomials: $x_1^{3}x_2^{12}x_3x_4$, $x_1^{3}x_2^{4}x_3x_4^{9}$, $x_1^{3}x_2^{4}x_3^{9}x_4$, and 

\medskip
\centerline{\begin{tabular}{lllll}
$x_1^{3}x_2^{4}x_3^{9}x_4^{3}x_5^{13}$& $x_1^{3}x_2^{4}x_3^{9}x_4^{7}x_5^{9}$& $x_1^{3}x_2^{5}x_3^{8}x_4^{7}x_5^{9}$& $x_1^{3}x_2^{5}x_3^{14}x_4x_5^{9} $\cr  $x_1^{3}x_2^{5}x_3^{14}x_4^{9}x_5$& $x_1^{3}x_2^{12}x_3x_4^{3}x_5^{13}$& $x_1^{3}x_2^{12}x_3^{3}x_4x_5^{13}$& $x_1^{3}x_2^{12}x_3^{3}x_4^{13}x_5 $\cr  $x_1^{7}x_2^{11}x_3^{4}x_4x_5^{9}$& $x_1^{7}x_2^{11}x_3^{4}x_4^{9}x_5$& $x_1^{7}x_2^{11}x_3^{12}x_4x_5$.& \cr  
\end{tabular}}
\end{lems}

\begin{lems}\label{bda63} Let $(i,j,t,u,v)$ be an arbitrary permutation of $(1,2,3,4,5)$. The following monomials are strictly inadmissible:
	
\smallskip
\ {\rm i)} $x_r^{31}\theta_{J_r}(w),\, 1\leqslant r\leqslant 5$, with $w$ one of the following monomials: $x_1^{3}x_2^{5}x_3^{8}x_4^{17}$, $x_1^{3}x_2^{5}x_3^{9}x_4^{16}$, $x_1^{3}x_2^{5}x_3^{24}x_4$. 
	
\smallskip
\ {\rm ii)} $x_i^{3}x_j^{7}x_t^{8}x_u^{17}x_v^{29}$, $x_i^{3}x_j^{7}x_t^{9}x_u^{16}x_v^{29}$, $x_i^{3}x_j^{7}x_t^{9}x_u^{17}x_v^{28}$, $ i,\, j \leqslant 2 < t,\, u,\, v$ and 

\smallskip
\centerline{\begin{tabular}{lllll}
$x_1^{3}x_2^{4}x_3^{7}x_4^{25}x_5^{25}$& $x_1^{3}x_2^{4}x_3^{11}x_4^{17}x_5^{29}$& $x_1^{3}x_2^{4}x_3^{11}x_4^{29}x_5^{17}$& $x_1^{3}x_2^{5}x_3^{6}x_4^{25}x_5^{25} $\cr  $x_1^{3}x_2^{5}x_3^{8}x_4^{19}x_5^{29}$& $x_1^{3}x_2^{5}x_3^{9}x_4^{17}x_5^{30}$& $x_1^{3}x_2^{5}x_3^{9}x_4^{18}x_5^{29}$& $x_1^{3}x_2^{5}x_3^{9}x_4^{19}x_5^{28} $\cr  $x_1^{3}x_2^{5}x_3^{9}x_4^{30}x_5^{17}$& $x_1^{3}x_2^{5}x_3^{10}x_4^{17}x_5^{29}$& $x_1^{3}x_2^{5}x_3^{10}x_4^{29}x_5^{17}$& $x_1^{3}x_2^{5}x_3^{11}x_4^{16}x_5^{29} $\cr  $x_1^{3}x_2^{5}x_3^{11}x_4^{17}x_5^{28}$& $x_1^{3}x_2^{5}x_3^{11}x_4^{28}x_5^{17}$& $x_1^{3}x_2^{5}x_3^{11}x_4^{29}x_5^{16}$& $x_1^{3}x_2^{5}x_3^{24}x_4^{3}x_5^{29} $\cr  $x_1^{3}x_2^{5}x_3^{25}x_4^{2}x_5^{29}$& $x_1^{3}x_2^{5}x_3^{25}x_4^{3}x_5^{28}$& $x_1^{3}x_2^{5}x_3^{25}x_4^{30}x_5$& $x_1^{3}x_2^{7}x_3^{4}x_4^{25}x_5^{25} $\cr  $x_1^{3}x_2^{7}x_3^{12}x_4^{17}x_5^{25}$& $x_1^{3}x_2^{7}x_3^{12}x_4^{25}x_5^{17}$& $x_1^{3}x_2^{7}x_3^{13}x_4^{16}x_5^{25}$& $x_1^{3}x_2^{7}x_3^{13}x_4^{24}x_5^{17} $\cr  $x_1^{3}x_2^{7}x_3^{13}x_4^{25}x_5^{16}$& $x_1^{3}x_2^{7}x_3^{24}x_4x_5^{29}$& $x_1^{3}x_2^{7}x_3^{24}x_4^{5}x_5^{25}$& $x_1^{3}x_2^{7}x_3^{24}x_4^{29}x_5 $\cr  $x_1^{3}x_2^{7}x_3^{28}x_4x_5^{25}$& $x_1^{3}x_2^{7}x_3^{28}x_4^{25}x_5$& $x_1^{3}x_2^{7}x_3^{29}x_4^{24}x_5$& $x_1^{3}x_2^{28}x_3^{3}x_4^{5}x_5^{25} $\cr  $x_1^{7}x_2^{3}x_3^{4}x_4^{25}x_5^{25}$& $x_1^{7}x_2^{3}x_3^{12}x_4^{17}x_5^{25}$& $x_1^{7}x_2^{3}x_3^{12}x_4^{25}x_5^{17}$& $x_1^{7}x_2^{3}x_3^{13}x_4^{16}x_5^{25} $\cr  $x_1^{7}x_2^{3}x_3^{13}x_4^{24}x_5^{17}$& $x_1^{7}x_2^{3}x_3^{13}x_4^{25}x_5^{16}$& $x_1^{7}x_2^{3}x_3^{24}x_4x_5^{29}$& $x_1^{7}x_2^{3}x_3^{24}x_4^{5}x_5^{25} $\cr  $x_1^{7}x_2^{3}x_3^{24}x_4^{29}x_5$& $x_1^{7}x_2^{3}x_3^{28}x_4x_5^{25}$& $x_1^{7}x_2^{3}x_3^{28}x_4^{25}x_5$& $x_1^{7}x_2^{3}x_3^{29}x_4^{24}x_5 $\cr  $x_1^{7}x_2^{11}x_3^{5}x_4^{16}x_5^{25}$& $x_1^{7}x_2^{11}x_3^{5}x_4^{24}x_5^{17}$& $x_1^{7}x_2^{11}x_3^{5}x_4^{25}x_5^{16}$& $x_1^{7}x_2^{11}x_3^{21}x_4^{8}x_5^{17} $\cr  $x_1^{7}x_2^{11}x_3^{21}x_4^{9}x_5^{16}$& $x_1^{7}x_2^{11}x_3^{21}x_4^{24}x_5$& $x_1^{7}x_2^{27}x_3^{5}x_4^{8}x_5^{17}$& $x_1^{7}x_2^{27}x_3^{5}x_4^{9}x_5^{16} $\cr  $x_1^{7}x_2^{27}x_3^{5}x_4^{24}x_5$& & &\cr   
\end{tabular}}
\end{lems}
\begin{proof} We prove Part i) for $r = 5$. We have
\begin{align*}
x_1^{3}x_2^{5}&x_3^{24}x_4x_5^{31} = x_1^{2}x_2^{5}x_3^{25}x_4x_5^{31} + x_1^{3}x_2^{4}x_3^{25}x_4x_5^{31} + x_1^{3}x_2^{5}x_3^{17}x_4^{8}x_5^{31}\\
&  +  Sq^1\big(x_1^{3}x_2^{3}x_3^{25}x_4x_5^{31}\big)  + Sq^2\big(x_1^{2}x_2^{3}x_3^{25}x_4x_5^{31} + x_1^{5}x_2^{3}x_3^{21}x_4^{2}x_5^{31}\\ 
& + x_1^{5}x_2^{3}x_3^{22}x_4x_5^{31}\big) +  Sq^4\big(x_1^{3}x_2^{3}x_3^{21}x_4^{2}x_5^{31} + x_1^{3}x_2^{3}x_3^{22}x_4x_5^{31} + x_1^{3}x_2^{9}x_3^{13}x_4^{4}x_5^{31}\big)\\ 
& +  Sq^8\big(x_1^{3}x_2^{5}x_3^{13}x_4^{4}x_5^{31} + x_1^{3}x_2^{5}x_3^{10}x_4x_5^{37}\big)\ \mbox{mod}\big(P_5^-((4)|(2)|^{4}\big),\\
x_1^{3}x_2^{5}&x_3^{8}x_4^{17}x_5^{31} = x_1^{2}x_2^{5}x_3x_4^{25}x_5^{31} + x_1^{3}x_2^{4}x_3x_4^{25}x_5^{31} + x_1^{3}x_2^{5}x_3x_4^{24}x_5^{31}\\ 
& +  Sq^1\big(x_1^{3}x_2^{3}x_3x_4^{25}x_5^{31}\big) +  Sq^2\big(x_1^{2}x_2^{3}x_3x_4^{25}x_5^{31} + x_1^{5}x_2^{3}x_3x_4^{22}x_5^{31}\\ 
& + x_1^{5}x_2^{3}x_3^2x_4^{21}x_5^{31}\big) +  Sq^4\big(x_1^{3}x_2^{3}x_3x_4^{22}x_5^{31} + x_1^{3}x_2^{3}x_3^{2}x_4^{21}x_5^{31} + x_1^{3}x_2^{9}x_3^{4}x_4^{13}x_5^{31}\big)\\ 
& +  Sq^8\big(x_1^{3}x_2^{3}x_3x_4^{14}x_5^{35} + x_1^{3}x_2^{3}x_3^{2}x_4^{13}x_5^{35} + x_1^{3}x_2^{5}x_3^{4}x_4^{13}x_5^{31}\\ 
&\quad + x_1^{3}x_2^{5}x_3^{4}x_4^{21}x_5^{23}\big) +  Sq^{16}\big(x_1^{3}x_2^{5}x_3^{4}x_4^{13}x_5^{23}\big)\ \mbox{mod}\big(P_5^-((4)|(2)|^{4}\big), \\
x_1^{3}x_2^{5}&x_3^{9}x_4^{16}x_5^{31} = x_1^{2}x_2^{5}x_3^{9}x_4^{17}x_5^{31} + x_1^{3}x_2^{4}x_3^{9}x_4^{17}x_5^{31} + x_1^{3}x_2^{5}x_3^{8}x_4^{17}x_5^{31}\\
& + Sq^1\big(x_1^{3}x_2^{3}x_3^{3}x_4^{23}x_5^{31}\big) +  Sq^2\big(x_1^{2}x_2^{3}x_3^{5}x_4^{21}x_5^{31} + x_1^{4}x_2^{5}x_3^{3}x_4^{15}x_5^{35}\\
& + x_1^{4}x_2^{5}x_3^{3}x_4^{19}x_5^{31} + x_1^{5}x_2^{3}x_3^{3}x_4^{20}x_5^{31} + x_1^{5}x_2^{3}x_3^{4}x_4^{15}x_5^{35} + x_1^{5}x_2^{3}x_3^{4}x_4^{19}x_5^{31}\\
& + x_1^{5}x_2^{4}x_3^{3}x_4^{15}x_5^{35} + x_1^{5}x_2^{4}x_3^{3}x_4^{19}x_5^{31}\big) +  Sq^4\big(x_1^{2}x_2^{9}x_3^{5}x_4^{13}x_5^{31} + x_1^{3}x_2^{3}x_3^{3}x_4^{20}x_5^{31}\\
& + x_1^{3}x_2^{3}x_3^{8}x_4^{15}x_5^{31} + x_1^{3}x_2^{8}x_3^{3}x_4^{15}x_5^{31} + x_1^{3}x_2^{8}x_3^{5}x_4^{13}x_5^{31} + x_1^{3}x_2^{9}x_3^{4}x_4^{13}x_5^{31}\\
& + x_1^{3}x_2^{9}x_3^{5}x_4^{12}x_5^{31} + x_1^{8}x_2^{3}x_3^{3}x_4^{15}x_5^{31}\big) +  Sq^8\big(x_1^{2}x_2^{5}x_3^{5}x_4^{13}x_5^{31} + x_1^{2}x_2^{5}x_3^{5}x_4^{21}x_5^{23}\\
& + x_1^{3}x_2^{3}x_3^{3}x_4^{12}x_5^{35} + x_1^{3}x_2^{3}x_3^{3}x_4^{15}x_5^{32} + x_1^{3}x_2^{3}x_3^{4}x_4^{11}x_5^{35} + x_1^{3}x_2^{3}x_3^{4}x_4^{15}x_5^{31}\\
& + x_1^{3}x_2^{3}x_3^{4}x_4^{23}x_5^{23} + x_1^{3}x_2^{3}x_3^{8}x_4^{19}x_5^{23} + x_1^{3}x_2^{4}x_3^{3}x_4^{11}x_5^{35} + x_1^{3}x_2^{4}x_3^{3}x_4^{15}x_5^{31}\\
& + x_1^{3}x_2^{4}x_3^{3}x_4^{23}x_5^{23} + x_1^{3}x_2^{4}x_3^{5}x_4^{13}x_5^{31} + x_1^{3}x_2^{4}x_3^{5}x_4^{21}x_5^{23} + x_1^{3}x_2^{5}x_3^{4}x_4^{13}x_5^{31}\\
& + x_1^{3}x_2^{5}x_3^{4}x_4^{21}x_5^{23} + x_1^{3}x_2^{5}x_3^{5}x_4^{12}x_5^{31} + x_1^{3}x_2^{5}x_3^{5}x_4^{20}x_5^{23} + x_1^{3}x_2^{8}x_3^{3}x_4^{19}x_5^{23}\\
& + x_1^{4}x_2^{3}x_3^{3}x_4^{11}x_5^{35} + x_1^{4}x_2^{3}x_3^{3}x_4^{15}x_5^{31} + x_1^{4}x_2^{3}x_3^{3}x_4^{23}x_5^{23} + x_1^{8}x_2^{3}x_3^{3}x_4^{19}x_5^{23}\big) \\
&+ Sq^{16}\big(x_1^{2}x_2^{5}x_3^{5}x_4^{13}x_5^{23} + x_1^{3}x_2^{3}x_3^{4}x_4^{15}x_5^{23} + x_1^{3}x_2^{4}x_3^{3}x_4^{15}x_5^{23} + x_1^{3}x_2^{4}x_3^{5}x_4^{13}x_5^{23}\\
& + x_1^{3}x_2^{5}x_3^{4}x_4^{13}x_5^{23} + x_1^{3}x_2^{5}x_3^{5}x_4^{12}x_5^{23} + x_1^{4}x_2^{3}x_3^{3}x_4^{15}x_5^{23}\big) \ \mbox{mod}\big(P_5^-((4)|(2)|^{4}\big).
\end{align*}
Hence, the monomials $x_1^{3}x_2^{5}x_3^{24}x_4x_5^{31}$, $x_1^{3}x_2^{5}x_3^{8}x_4^{17}x_5^{31}$, $x_1^{3}x_2^{5}x_3^{9}x_4^{16}x_5^{31}$ are strictly inadmissible.

We prove Part ii) for $x = x_1^{3}x_2^{7}x_3^{8}x_4^{17}x_5^{29}$, $y = x_1^{3}x_2^{5}x_3^{10}x_4^{17}x_5^{29}$, $z=x_1^{3}x_2^{5}x_3^{9}x_4^{19}x_5^{28} $ and $t = x_1^{7}x_2^{3}x_3^{12}x_4^{17}x_5^{25}$. The others can be proved by a similar computation. By using the Cartan formula we have
\begin{align*}
x &= x_1^{2}x_2^{11}x_3x_4^{21}x_5^{29} + x_1^{2}x_2^{13}x_3x_4^{19}x_5^{29} + x_1^{3}x_2^{4}x_3x_4^{27}x_5^{29} + x_1^{3}x_2^{5}x_3x_4^{25}x_5^{30}\\ 
&+ x_1^{3}x_2^{5}x_3^{2}x_4^{25}x_5^{29} + x_1^{3}x_2^{5}x_3^{8}x_4^{19}x_5^{29} + x_1^{3}x_2^{7}x_3x_4^{24}x_5^{29} +  Sq^1\big(x_1^{3}x_2^{11}x_3x_4^{19}x_5^{29}\big)\\ 
&+  Sq^2\big(x_1^{2}x_2^{11}x_3x_4^{19}x_5^{29} + x_1^{5}x_2^{7}x_3x_4^{19}x_5^{30} + x_1^{5}x_2^{7}x_3^{2}x_4^{19}x_5^{29}\big) +  Sq^4\big(x_1^{3}x_2^{4}x_3x_4^{23}x_5^{29}\\ 
&+ x_1^{3}x_2^{5}x_3x_4^{21}x_5^{30} + x_1^{3}x_2^{5}x_3^{2}x_4^{21}x_5^{29} + x_1^{3}x_2^{5}x_3^{4}x_4^{19}x_5^{29} + x_1^{3}x_2^{7}x_3x_4^{19}x_5^{30}\\ 
&+ x_1^{3}x_2^{7}x_3x_4^{20}x_5^{29} + x_1^{3}x_2^{7}x_3^{2}x_4^{19}x_5^{29} + x_1^{3}x_2^{11}x_3^{4}x_4^{13}x_5^{29} + x_1^{3}x_2^{12}x_3x_4^{15}x_5^{29}\big)\\ 
&+  Sq^8\big(x_1^{3}x_2^{7}x_3x_4^{11}x_5^{34} + x_1^{3}x_2^{7}x_3^{2}x_4^{11}x_5^{33} + x_1^{3}x_2^{7}x_3^{4}x_4^{13}x_5^{29} + x_1^{3}x_2^{7}x_3^{4}x_4^{21}x_5^{21}\\ 
&+ x_1^{3}x_2^{8}x_3x_4^{15}x_5^{29}\big) +  Sq^{16}\big(x_1^{3}x_2^{7}x_3^{4}x_4^{13}x_5^{21}\big)\ \mbox{mod}\big(P_5^-((4)|(2)|^{4}\big),\\
y &= x_1^{2}x_2^{3}x_3^{9}x_4^{21}x_5^{29} + x_1^{2}x_2^{5}x_3^{9}x_4^{19}x_5^{29} + x_1^{3}x_2^{3}x_3^{8}x_4^{21}x_5^{29} + x_1^{3}x_2^{3}x_3^{9}x_4^{20}x_5^{29}\\ 
& + x_1^{3}x_2^{4}x_3^{9}x_4^{19}x_5^{29} + x_1^{3}x_2^{5}x_3^{8}x_4^{19}x_5^{29} + x_1^{3}x_2^{5}x_3^{9}x_4^{17}x_5^{30} +  Sq^1\big(x_1^{3}x_2^{3}x_3^{9}x_4^{19}x_5^{29}\big)\\ 
& +  Sq^2\big(x_1^{2}x_2^{3}x_3^{9}x_4^{19}x_5^{29} + x_1^{5}x_2^{3}x_3^{5}x_4^{19}x_5^{30}\big) +  Sq^4\big(x_1^{3}x_2^{3}x_3^{5}x_4^{19}x_5^{30} + x_1^{3}x_2^{3}x_3^{6}x_4^{19}x_5^{29}\\ 
& + x_1^{3}x_2^{9}x_3^{5}x_4^{13}x_5^{30} + x_1^{3}x_2^{9}x_3^{6}x_4^{13}x_5^{29}\big) +  Sq^8\big(x_1^{3}x_2^{3}x_3^{5}x_4^{11}x_5^{34} + x_1^{3}x_2^{3}x_3^{6}x_4^{11}x_5^{33}\\ 
& + x_1^{3}x_2^{5}x_3^{5}x_4^{13}x_5^{30} + x_1^{3}x_2^{5}x_3^{5}x_4^{21}x_5^{22} + x_1^{3}x_2^{5}x_3^{6}x_4^{13}x_5^{29} + x_1^{3}x_2^{5}x_3^{6}x_4^{21}x_5^{21}\big)\\ 
& +  Sq^{16}\big(x_1^{3}x_2^{5}x_3^{5}x_4^{13}x_5^{22} + x_1^{3}x_2^{5}x_3^{6}x_4^{13}x_5^{21}\big)\ \mbox{mod}\big(P_5^-((4)|(2)|^{4}\big),\\
z &= x_1x_2^{3}x_3^{3}x_4^{28}x_5^{29} + x_1x_2^{3}x_3^{3}x_4^{29}x_5^{28} + x_1x_2^{3}x_3^{4}x_4^{27}x_5^{29} + x_1x_2^{3}x_3^{4}x_4^{29}x_5^{27}\\ 
&  + x_1x_2^{3}x_3^{5}x_4^{25}x_5^{30} + x_1x_2^{3}x_3^{5}x_4^{30}x_5^{25} + x_1x_2^{4}x_3^{3}x_4^{27}x_5^{29} + x_1x_2^{4}x_3^{3}x_4^{29}x_5^{27}\\ 
&  + x_1^{2}x_2x_3^{5}x_4^{27}x_5^{29} + x_1^{2}x_2x_3^{5}x_4^{29}x_5^{27} + x_1^{2}x_2^{3}x_3^{9}x_4^{21}x_5^{29} + x_1^{2}x_2^{5}x_3x_4^{27}x_5^{29}\\ 
&  + x_1^{2}x_2^{5}x_3x_4^{29}x_5^{27} + x_1^{2}x_2^{5}x_3^{3}x_4^{25}x_5^{29} + x_1^{2}x_2^{5}x_3^{5}x_4^{27}x_5^{25} + x_1^{2}x_2^{5}x_3^{9}x_4^{19}x_5^{29}\\ 
&  + x_1^{3}x_2x_3^{3}x_4^{28}x_5^{29} + x_1^{3}x_2x_3^{3}x_4^{29}x_5^{28} + x_1^{3}x_2x_3^{4}x_4^{27}x_5^{29} + x_1^{3}x_2x_3^{4}x_4^{29}x_5^{27}\\ 
&  + x_1^{3}x_2x_3^{5}x_4^{25}x_5^{30} + x_1^{3}x_2x_3^{5}x_4^{30}x_5^{25} + x_1^{3}x_2^{2}x_3^{5}x_4^{25}x_5^{29} + x_1^{3}x_2^{2}x_3^{5}x_4^{29}x_5^{25}\\ 
&  + x_1^{3}x_2^{3}x_3x_4^{28}x_5^{29} + x_1^{3}x_2^{3}x_3x_4^{29}x_5^{28} + x_1^{3}x_2^{3}x_3^{4}x_4^{25}x_5^{29} + x_1^{3}x_2^{3}x_3^{5}x_4^{24}x_5^{29}\\ 
&  + x_1^{3}x_2^{3}x_3^{5}x_4^{28}x_5^{25} + x_1^{3}x_2^{3}x_3^{9}x_4^{20}x_5^{29} + x_1^{3}x_2^{4}x_3x_4^{27}x_5^{29} + x_1^{3}x_2^{4}x_3x_4^{29}x_5^{27}\\ 
&  + x_1^{3}x_2^{4}x_3^{3}x_4^{25}x_5^{29} + x_1^{3}x_2^{4}x_3^{5}x_4^{27}x_5^{25} + x_1^{3}x_2^{4}x_3^{9}x_4^{19}x_5^{29} + x_1^{3}x_2^{5}x_3x_4^{25}x_5^{30}\\ 
&  + x_1^{3}x_2^{5}x_3x_4^{30}x_5^{25} + x_1^{3}x_2^{5}x_3^{2}x_4^{25}x_5^{29} + x_1^{3}x_2^{5}x_3^{2}x_4^{29}x_5^{25} + x_1^{3}x_2^{5}x_3^{3}x_4^{25}x_5^{28}\\ 
&  + x_1^{3}x_2^{5}x_3^{4}x_4^{25}x_5^{27} + x_1^{3}x_2^{5}x_3^{4}x_4^{27}x_5^{25} + x_1^{3}x_2^{5}x_3^{6}x_4^{25}x_5^{25} + x_1^{3}x_2^{5}x_3^{8}x_4^{19}x_5^{29}\\ 
&  + x_1^{3}x_2^{5}x_3^{8}x_4^{21}x_5^{27} + x_1^{3}x_2^{5}x_3^{8}x_4^{25}x_5^{23} +  
Sq^1\big(x_1x_2^{3}x_3^{3}x_4^{23}x_5^{33} + x_1x_2^{3}x_3^{3}x_4^{27}x_5^{29}\\ 
& + x_1x_2^{3}x_3^{3}x_4^{29}x_5^{27} + x_1x_2^{3}x_3^{3}x_4^{33}x_5^{23} + x_1^{3}x_2x_3^{3}x_4^{23}x_5^{33} + x_1^{3}x_2x_3^{3}x_4^{27}x_5^{29}\\ 
& + x_1^{3}x_2x_3^{3}x_4^{29}x_5^{27} + x_1^{3}x_2x_3^{3}x_4^{33}x_5^{23} + x_1^{3}x_2^{3}x_3x_4^{23}x_5^{33} + x_1^{3}x_2^{3}x_3x_4^{27}x_5^{29}\\ 
& + x_1^{3}x_2^{3}x_3x_4^{29}x_5^{27} + x_1^{3}x_2^{3}x_3x_4^{33}x_5^{23} + x_1^{3}x_2^{3}x_3^{3}x_4^{25}x_5^{29} + x_1^{3}x_2^{3}x_3^{3}x_4^{27}x_5^{27}\\ 
& + x_1^{3}x_2^{3}x_3^{5}x_4^{19}x_5^{33} + x_1^{3}x_2^{3}x_3^{5}x_4^{25}x_5^{27} + x_1^{3}x_2^{3}x_3^{9}x_4^{19}x_5^{29}
\big) + Sq^2\big(x_1x_2^{5}x_3^{3}x_4^{23}x_5^{30}\\ 
& + x_1x_2^{5}x_3^{3}x_4^{30}x_5^{23} + x_1^{2}x_2x_3^{3}x_4^{27}x_5^{29} + x_1^{2}x_2x_3^{3}x_4^{29}x_5^{27} + x_1^{2}x_2^{3}x_3x_4^{27}x_5^{29}\\ 
& + x_1^{2}x_2^{3}x_3x_4^{29}x_5^{27} + x_1^{2}x_2^{3}x_3^{3}x_4^{29}x_5^{25} + x_1^{2}x_2^{3}x_3^{5}x_4^{23}x_5^{29} + x_1^{2}x_2^{3}x_3^{5}x_4^{27}x_5^{25}\\ 
& + x_1^{2}x_2^{3}x_3^{5}x_4^{29}x_5^{23} + x_1^{2}x_2^{3}x_3^{9}x_4^{19}x_5^{29} + x_1^{3}x_2^{3}x_3^{3}x_4^{24}x_5^{29} + x_1^{3}x_2^{3}x_3^{3}x_4^{26}x_5^{27}\\ 
& + x_1^{3}x_2^{3}x_3^{3}x_4^{28}x_5^{25} + x_1^{3}x_2^{3}x_3^{5}x_4^{22}x_5^{29} + x_1^{3}x_2^{3}x_3^{5}x_4^{24}x_5^{27} + x_1^{3}x_2^{3}x_3^{5}x_4^{26}x_5^{25}\\ 
& + x_1^{3}x_2^{3}x_3^{5}x_4^{28}x_5^{23} + x_1^{4}x_2^{3}x_3^{3}x_4^{23}x_5^{29} + x_1^{4}x_2^{3}x_3^{3}x_4^{29}x_5^{23} + x_1^{4}x_2^{5}x_3^{3}x_4^{23}x_5^{27}\\ 
& + x_1^{4}x_2^{5}x_3^{3}x_4^{27}x_5^{23} + x_1^{5}x_2x_3^{3}x_4^{23}x_5^{30} + x_1^{5}x_2x_3^{3}x_4^{30}x_5^{23} + x_1^{5}x_2^{2}x_3^{3}x_4^{23}x_5^{29}\\ 
& + x_1^{5}x_2^{2}x_3^{3}x_4^{29}x_5^{23} + x_1^{5}x_2^{3}x_3x_4^{23}x_5^{30} + x_1^{5}x_2^{3}x_3x_4^{30}x_5^{23} + x_1^{5}x_2^{3}x_3^{2}x_4^{23}x_5^{29}\\ 
& + x_1^{5}x_2^{3}x_3^{2}x_4^{29}x_5^{23} + x_1^{5}x_2^{3}x_3^{3}x_4^{21}x_5^{30} + x_1^{5}x_2^{3}x_3^{3}x_4^{23}x_5^{28} + x_1^{5}x_2^{3}x_3^{4}x_4^{23}x_5^{27}\\ 
& + x_1^{5}x_2^{3}x_3^{4}x_4^{27}x_5^{23} + x_1^{5}x_2^{3}x_3^{5}x_4^{19}x_5^{30} + x_1^{5}x_2^{3}x_3^{6}x_4^{19}x_5^{29} + x_1^{5}x_2^{3}x_3^{6}x_4^{21}x_5^{27}\\ 
& + x_1^{5}x_2^{3}x_3^{6}x_4^{25}x_5^{23} + x_1^{5}x_2^{4}x_3^{3}x_4^{23}x_5^{27} + x_1^{5}x_2^{4}x_3^{3}x_4^{27}x_5^{23}  
\big) + Sq^4\big(x_1x_2^{3}x_3^{3}x_4^{23}x_5^{30}\\ 
&  + x_1x_2^{3}x_3^{3}x_4^{30}x_5^{23} + x_1^{2}x_2^{3}x_3^{3}x_4^{23}x_5^{29} + x_1^{2}x_2^{3}x_3^{3}x_4^{29}x_5^{23} + x_1^{3}x_2x_3^{3}x_4^{23}x_5^{30}\\ 
&  + x_1^{3}x_2x_3^{3}x_4^{30}x_5^{23} + x_1^{3}x_2^{2}x_3^{3}x_4^{23}x_5^{29} + x_1^{3}x_2^{2}x_3^{3}x_4^{29}x_5^{23} + x_1^{3}x_2^{3}x_3x_4^{23}x_5^{30}\\ 
&  + x_1^{3}x_2^{3}x_3x_4^{30}x_5^{23} + x_1^{3}x_2^{3}x_3^{2}x_4^{23}x_5^{29} + x_1^{3}x_2^{3}x_3^{2}x_4^{29}x_5^{23} + x_1^{3}x_2^{3}x_3^{3}x_4^{21}x_5^{30}\\ 
&  + x_1^{3}x_2^{3}x_3^{3}x_4^{23}x_5^{28} + x_1^{3}x_2^{3}x_3^{5}x_4^{19}x_5^{30} + x_1^{3}x_2^{3}x_3^{6}x_4^{19}x_5^{29} + x_1^{3}x_2^{3}x_3^{8}x_4^{23}x_5^{23}\\ 
&  + x_1^{3}x_2^{3}x_3^{10}x_4^{21}x_5^{23} + x_1^{3}x_2^{8}x_3^{3}x_4^{23}x_5^{23} + x_1^{3}x_2^{9}x_3^{3}x_4^{22}x_5^{23} + x_1^{3}x_2^{9}x_3^{5}x_4^{15}x_5^{28}\\ 
&  + x_1^{8}x_2^{3}x_3^{3}x_4^{23}x_5^{23} + x_1^{9}x_2^{3}x_3^{3}x_4^{22}x_5^{23}\big) + Sq^8\big(x_1^{3}x_2^{3}x_3^{3}x_4^{13}x_5^{34} + x_1^{3}x_2^{3}x_3^{3}x_4^{15}x_5^{32}\\ 
& + x_1^{3}x_2^{3}x_3^{4}x_4^{23}x_5^{23} + x_1^{3}x_2^{3}x_3^{6}x_4^{21}x_5^{23} + x_1^{3}x_2^{4}x_3^{3}x_4^{23}x_5^{23} + x_1^{3}x_2^{5}x_3^{3}x_4^{22}x_5^{23}\\ 
& + x_1^{3}x_2^{5}x_3^{5}x_4^{15}x_5^{28} + x_1^{3}x_2^{5}x_3^{5}x_4^{23}x_5^{20} + x_1^{4}x_2^{3}x_3^{3}x_4^{23}x_5^{23} + x_1^{5}x_2^{3}x_3^{3}x_4^{22}x_5^{23}\big)\\ 
& + Sq^{16}\big(x_1^{3}x_2^{5}x_3^{5}x_4^{15}x_5^{20}\big)\ \mbox{mod}\big(P_5^-((4)|(2)|^{4}\big),\\
t &= x_1^{4}x_2^{3}x_3^{3}x_4^{25}x_5^{29} + x_1^{5}x_2x_3^{3}x_4^{25}x_5^{30} + x_1^{5}x_2x_3^{3}x_4^{26}x_5^{29} + x_1^{5}x_2x_3^{10}x_4^{19}x_5^{29}\\ 
& + x_1^{5}x_2^{2}x_3^{3}x_4^{25}x_5^{29} + x_1^{5}x_2^{3}x_3^{3}x_4^{24}x_5^{29} + x_1^{5}x_2^{3}x_3^{5}x_4^{24}x_5^{27} + x_1^{5}x_2^{3}x_3^{9}x_4^{17}x_5^{30}\\ 
& + x_1^{5}x_2^{3}x_3^{9}x_4^{24}x_5^{23} + x_1^{5}x_2^{3}x_3^{10}x_4^{17}x_5^{29} + x_1^{5}x_2^{3}x_3^{12}x_4^{17}x_5^{27} + x_1^{7}x_2x_3^{3}x_4^{24}x_5^{29}\\ 
& + x_1^{7}x_2x_3^{8}x_4^{19}x_5^{29} + x_1^{7}x_2^{2}x_3^{9}x_4^{17}x_5^{29} + x_1^{7}x_2^{3}x_3^{9}x_4^{17}x_5^{28} + x_1^{7}x_2^{3}x_3^{9}x_4^{20}x_5^{25}\\ 
& +  Sq^1\big(x_1^{7}x_2x_3^{5}x_4^{21}x_5^{29} + x_1^{7}x_2^{3}x_3^{9}x_4^{17}x_5^{27} + x_1^{11}x_2^{3}x_3^{3}x_4^{17}x_5^{29}\big)\\ 
& + Sq^2\big(x_1^{7}x_2x_3^{3}x_4^{21}x_5^{30} + x_1^{7}x_2x_3^{3}x_4^{22}x_5^{29} + x_1^{7}x_2x_3^{6}x_4^{19}x_5^{29} + x_1^{7}x_2^{2}x_3^{3}x_4^{21}x_5^{29}\\ 
& + x_1^{7}x_2^{2}x_3^{9}x_4^{17}x_5^{27} + x_1^{7}x_2^{5}x_3^{3}x_4^{17}x_5^{30} + x_1^{7}x_2^{5}x_3^{3}x_4^{18}x_5^{29} + x_1^{7}x_2^{5}x_3^{5}x_4^{18}x_5^{27}\\ 
& + x_1^{7}x_2^{5}x_3^{6}x_4^{17}x_5^{27} + x_1^{7}x_2^{5}x_3^{9}x_4^{18}x_5^{23} + x_1^{7}x_2^{5}x_3^{10}x_4^{17}x_5^{23}\big) + Sq^4\big(x_1^{4}x_2^{3}x_3^{3}x_4^{21}x_5^{29}\\ 
& + x_1^{5}x_2x_3^{3}x_4^{21}x_5^{30} + x_1^{5}x_2x_3^{3}x_4^{22}x_5^{29} + x_1^{5}x_2x_3^{6}x_4^{19}x_5^{29} + x_1^{5}x_2^{2}x_3^{3}x_4^{21}x_5^{29}\\ 
& + x_1^{5}x_2^{3}x_3^{3}x_4^{20}x_5^{29} + x_1^{5}x_2^{3}x_3^{5}x_4^{17}x_5^{30} + x_1^{5}x_2^{3}x_3^{5}x_4^{20}x_5^{27} + x_1^{5}x_2^{3}x_3^{6}x_4^{17}x_5^{29}\\ 
& + x_1^{5}x_2^{3}x_3^{9}x_4^{20}x_5^{23} + x_1^{5}x_2^{3}x_3^{12}x_4^{17}x_5^{23} + x_1^{7}x_2^{3}x_3^{3}x_4^{17}x_5^{30} + x_1^{7}x_2^{3}x_3^{3}x_4^{18}x_5^{29}\\ 
& + x_1^{7}x_2^{3}x_3^{4}x_4^{17}x_5^{29} + x_1^{7}x_2^{8}x_3^{3}x_4^{13}x_5^{29} + x_1^{11}x_2^{3}x_3^{5}x_4^{18}x_5^{23} + x_1^{11}x_2^{3}x_3^{6}x_4^{17}x_5^{23}\\ 
& + x_1^{12}x_2^{3}x_3^{3}x_4^{13}x_5^{29}\big) +  Sq^8\big)x_1^{7}x_2^{3}x_3^{3}x_4^{9}x_5^{34} + x_1^{7}x_2^{3}x_3^{3}x_4^{10}x_5^{33} + x_1^{7}x_2^{3}x_3^{5}x_4^{18}x_5^{23}\\ 
& + x_1^{7}x_2^{3}x_3^{6}x_4^{17}x_5^{23} + x_1^{7}x_2^{4}x_3^{3}x_4^{13}x_5^{29} + x_1^{7}x_2^{4}x_3^{3}x_4^{21}x_5^{21} + x_1^{8}x_2^{3}x_3^{3}x_4^{13}x_5^{29}\\ 
& +  Sq^{16}\big(x_1^{7}x_2^{4}x_3^{3}x_4^{13}x_5^{21}\big)\ \mbox{mod}\big(P_5^-((4)|(2)|^{4}\big).
\end{align*}
The above equalities show that the monomials $x,\, y,\, z,\, t$ are strictly inadmissible.
\end{proof}

\begin{proof}[Proof of Theorem $\ref{mdd52}$] Let $A(d)$ and $B(d)$ be as in Subsection \ref{s51} with $d \geqslant 6$ and let $x \in P_5^+(\omega)$ be an admissible monomial with $\omega := (4)|(2)|^{d-2}$, then $x = X_iy^2$ with $1\leqslant i \leqslant 5$ and $y$ a monomial of weight vector $(2)|^{d-2}$. Since $x$ is admissible, by Theorem \ref{dlcb1}, $y$ is admissible. By a direct computation we see that if $\bar y$ is an admissible monomial of weight vector $(2)|^{d-2}$ and  $X_i\bar y^2 \notin A(d)\cup B(d)$, then there is a monomial $w$ as given in one of Lemmas \ref{bda62} or \ref{bda63} such that $X_i\bar y^2 = wz^{2^r}$ with $r$ an integer, $2\leqslant r \leqslant 6$, and $z$ a monomial of weight vector $(2)|^{d-r}$. By Theorem \ref{dlcb1}, $X_i\bar y^2$ is inadmissible. Since $x = X_iy^2$ with $y$ an admissible monomial of weight vector $(2)|^{d-2}$, we get $x\in A(d)\cup B(d)$ and $B_5^+(\omega) \subset A(d)\cup B(d)$.
	
Now we prove that the set $[A(d)\cup B(d)]_{\omega}$ is linearly independent in $QP_5(\omega)$.

Consider the subspaces $\langle [A(d)]_{\omega}\rangle \subset QP_5(\omega)$ and $\langle [B(d)]_{\omega}\rangle \subset QP_5(\omega)$. It is easy to see that for any $x\in A(d)$, we have $x = \theta_{J_i}(\bar x)$ with $\bar x$ an admissible monomial of weight vector $(4)|(2)|^{d-2}$ in $P_4$. By Proposition \ref{mdmo}, $x$ is admissible. This implies $\dim \langle [A(d)]_{\omega}\rangle = 85$. Since $\nu(x) = 2^{d-1}-1$ for all $x\in A(d)$ and $\nu(x) < 2^{d-1}-1$ for all $x\in B(d)$, we obtain $\langle [A(d)]_{\omega}\rangle \cap \langle [B(d)]_{\omega}\rangle = \{0\}$. Hence, we need only to prove the set $[B(d)]_{\omega}=\{[b_{d,t}]_{\omega}: 1 \leqslant t \leqslant 205 \}$ is linearly independent in $QP_5(\omega)$, where the monomials $b_{d,t}: 1 \leqslant t \leqslant 205$, are determined as in Subsection \ref{s51}. 

Suppose there is a linear relation
\begin{equation}\label{ctd612}
\mathcal S:= \sum_{1\leqslant t \leqslant 205}\gamma_tb_{d,t} \equiv_{\omega} 0,
\end{equation}
where $\gamma_t \in \mathbb F_2$. We denote $\gamma_{\mathbb J} = \sum_{t \in \mathbb J}\gamma_t$ for any $\mathbb J \subset \{t\in \mathbb N:1\leqslant t \leqslant 75\}$.
	
Let $v_u = v_{d,u},\, 1\leqslant u \leqslant 35$, be as in Subsection \ref{s51} and the homomorphism $p_{(i;I)}:P_5\to P_4$ which is defined by \eqref{ct23} for $k=5$. From Lemma \ref{bdm}, we see that $p_{(i;I)}$ passes to a homomorphism from $QP_5(\omega)$ to $QP_4(\omega)$. By applying $p_{(1;j)}$, $2\leqslant j \leqslant 5,$ to (\ref{ctd612}), we obtain
\begin{align*}
&p_{(1;2)}(S) \equiv_{\omega} \gamma_{7}v_{10} +  \gamma_{8}v_{12} +   \gamma_{9}v_{14} +  \gamma_{10}v_{15} +  \gamma_{11}v_{17} +  \gamma_{\{36,107\}}v_{22}\\
&\qquad +  \gamma_{\{37,108\}}v_{23} +  \gamma_{\{38,109\}}v_{24} +  \gamma_{\{39,110\}}v_{25} +  \gamma_{\{40,111\}}v_{26}\\
&\qquad +  \gamma_{\{41,112\}}v_{27} +  \gamma_{\{42,113\}}v_{28} +  \gamma_{61}v_{30} +  \gamma_{62}v_{32} +  \gamma_{63}v_{33} +  \gamma_{64}v_{34} \equiv_{\omega} 0,\\
&p_{(1;3)}(S) \equiv_{\omega} \gamma_{\{3,70\}}v_{4} +  \gamma_{\{6,81\}}v_{6} +   \gamma_{\{14,97,188\}}v_{10} +  \gamma_{17}v_{12} +  \gamma_{\{22,91,97,188\}}v_{14}\\
&\qquad +  \gamma_{\{23,92,97,188\}}v_{15} +  \gamma_{28}v_{17} +  \gamma_{\{33,100\}}v_{19} +  \gamma_{\{34,101\}}v_{20} +  \gamma_{\{35,102\}}v_{21}\\
&\qquad +  \gamma_{45}v_{23} +  \gamma_{46}v_{24} +  \gamma_{\{51,143\}}v_{25} +  \gamma_{54}v_{26} +  \gamma_{55}v_{27} +  \gamma_{56}v_{28} \equiv_{\omega} 0,\\
&p_{(1;4)}(S) \equiv_{\omega} \gamma_{\{2,68,79,168\}}v_{3} +   \gamma_{\{13,89,139,182\}}v_{9} +  \gamma_{\{15,98,104,120,127,134\}}v_{10}\\
&\qquad +  \gamma_{\{4,77\}}v_{4} +  \gamma_{\{19,93,98,120,127,144\}}v_{12} +  \gamma_{\{21,95,122,186\}}v_{13} +  \gamma_{24}v_{14}\\
&\qquad +  \gamma_{\{26,98,104,124,127,134\}}v_{15} +  \gamma_{\{30,104,134,148,160\}}v_{17} +  \gamma_{\{32,106,136,150,162\}}v_{18}\\
&\qquad +  \gamma_{\{43,137\}}v_{22} +  \gamma_{47}v_{23} +  \gamma_{\{49,141,160\}}v_{24} +  \gamma_{52}v_{25} +  \gamma_{58}v_{27} +  \gamma_{60}v_{28} \equiv_{\omega} 0,\\
&p_{(1;5)}(S) \equiv_{\omega} \gamma_{\{1,69,73,78,84,171,174,180\}}v_{2} +  \gamma_{\{5,80\}}v_{4}+  \gamma_{\{16,96,145,146,149,161\}}v_{10} \\
&\qquad  +   \gamma_{\{12,90,116,138,154,185,198,203\}}v_{8}+  \gamma_{\{18,94,121,187,190,191,194,199\}}v_{11}\\
&\qquad +  \gamma_{\{20,96,99,123,128,145\}}v_{12} +  \gamma_{\{25,99,105,125,146\}}v_{14} +  \gamma_{27}v_{15}\\
&\qquad +  \gamma_{\{29,103,133,147,159,196,200,205\}}v_{16} +  \gamma_{\{31,105,135,149,161\}}v_{17} +  \gamma_{\{44,140\}}v_{22}\\
&\qquad +  \gamma_{\{48,142,145,146,149,161\}}v_{23} +  \gamma_{50}v_{24} +  \gamma_{53}v_{25} +  \gamma_{\{57,151\}}v_{26} +  \gamma_{59}v_{27}  \equiv_{\omega} 0.
\end{align*}
From the above equalities, we get
\begin{align}\label{c611}
\gamma_j = 0,\ j \in \mathbb J_1 \mbox{ and } \gamma_s = \gamma_t,\ (s,t) \in \mathbb K_1,
\end{align}
where $\mathbb J_1 = \{$7, 8, 9, 10, 11, 17, 24, 27, 28, 45, 46, 47, 50, 52, 53, 54, 55, 56, 58, 59, 60, 61, 62, 63, 64$\}$ and $\mathbb K_1 = \{$(3,70), (4,77), (5,80), (6,81), (33,100), (34,101), (35,102), (36,107), (37,108), (38,109), (39,110), (40,111), (41,112), (42,113), (43,137), (44,140), (51,143), (57,151)$\}$.

By applying the homomorphism $p_{(2;j)}$, $3\leqslant j \leqslant 5,$ to \eqref{ctd612} and using \eqref{c611}, we obtain
\begin{align*}
&p_{(2;3)}(S) \equiv_{\omega} \gamma_{\{3,36\}}v_{4} +  \gamma_{6}v_{6} +   \gamma_{\{39,51,67,97,126\}}v_{10} +  \gamma_{\{3,36\}}v_{12} +  \gamma_{166}v_{23} \\
&\qquad +  \gamma_{\{37,39,51,75,91,97,118,126\}}v_{14} +  \gamma_{\{38,39,51,76,92,97,119,126\}}v_{15} +  \gamma_{6}v_{17}\\
&\qquad +  \gamma_{\{33,40,86,130,156\}}v_{19} +  \gamma_{\{34,41,87,131,157\}}v_{20} +  \gamma_{\{35,42,88,132,158\}}v_{21}\\
&\qquad +  \gamma_{167}v_{24} +  \gamma_{\{172,188\}}v_{25} +  \gamma_{175}v_{26} +  \gamma_{176}v_{27} +  \gamma_{177}v_{28} \equiv_{\omega} 0,\\
&p_{(2;4)}(S) \equiv_{\omega} \gamma_{\{2,15,26,37\}}v_{3} +  \gamma_{\{4,40\}}v_{4} +   \gamma_{\{36,43,66,89,115,153,184,197,202\}}v_{9}\\
&\qquad +  \gamma_{\{37,68,98,104,120\}}v_{10} +  \gamma_{\{72,93,98,189\}}v_{12} +  \gamma_{\{74,95,141\}}v_{13} +  \gamma_{\{4,40\}}v_{14}\\
&\qquad +  \gamma_{\{79,98,104,124\}}v_{15} +  \gamma_{\{83,104,193\}}v_{17} +  \gamma_{\{85,106,195\}}v_{18} +  \gamma_{\{164,182\}}v_{22}\\
&\qquad +  \gamma_{168}v_{23} +  \gamma_{\{170,186\}}v_{24} +  \gamma_{173}v_{25} +  \gamma_{179}v_{27} +  \gamma_{181}v_{28}  \equiv_{\omega} 0,\\
&p_{(2;5)}(S) \equiv_{\omega} \gamma_{\{1,16,20,25,31,38,39,41\}}v_{2} +  \gamma_{\{5,42\}}v_{4} +   \gamma_{\{44,65,90,183\}}v_{8}\\
&\qquad +  \gamma_{\{38,69,96,123,190,191,194\}}v_{10} +  \gamma_{\{71,94,142,145,146,149\}}v_{11}\\
&\qquad +  \gamma_{\{39,41,78,99,105,125,128,135,191\}}v_{14} +  \gamma_{\{5,42\}}v_{15} +  \gamma_{\{57,82,103,192\}}v_{16}\\
&\qquad +  \gamma_{\{84,105,194\}}v_{17} +  \gamma_{\{165,185\}}v_{22} +  \gamma_{\{169,187,190,191,194\}}v_{23} +  \gamma_{171}v_{24}\\
&\qquad  +  \gamma_{\{73,96,99,190\}}v_{12} +  \gamma_{174}v_{25} +  \gamma_{\{178,196\}}v_{26} +  \gamma_{180}v_{27}  \equiv_{\omega} 0.
\end{align*}
These equalities implies
\begin{align}\label{c612}
\gamma_j = 0,\ j \in \mathbb J_2 \mbox{ and } \gamma_s = \gamma_t,\ (s,t) \in \mathbb K_2,
\end{align}
where $\mathbb J_2 = \{$6, 81, 166, 167, 168, 171, 173, 174, 175, 176, 177, 179, 180, 181$\}$ and $\mathbb K_2 = \{$(3,36), (4,40), (5,42), (164,182), (165,185), (170,186), (172,188), (178,196)$\}$.

By applying the homomorphism $p_{(i;j)}$, $3\leqslant i < j \leqslant 5,$ to \eqref{ctd612} and using \eqref{c611}, \eqref{c612}, we have
\begin{align*}
&p_{(3;4)}(S) \equiv_{\omega} \gamma_{\{13,14,15,19,22,30,33,49,51\}}v_{3} +   \gamma_{\{3,4,66,67,68,72,75,83,86,170,172\}}v_{9}\\
&\qquad +  \gamma_{43}v_{4} +  \gamma_{\{33,89,91,93,157,160\}}v_{10} +  \gamma_{\{117,119,122,124,126,127,131,134\}}v_{13}\\
&\qquad +  \gamma_{\{115,118,120,130\}}v_{12} +  \gamma_{43}v_{14} +  \gamma_{\{51,139,141,144,148,157,160\}}v_{15}\\
&\qquad +  \gamma_{\{153,156,157,160\}}v_{17} +  \gamma_{\{155,158,162,163\}}v_{18} +  \gamma_{164}v_{22} +  \gamma_{164}v_{23}\\
&\qquad +  \gamma_{\{170,172,184,189,193\}}v_{24} +  \gamma_{197}v_{25} +  \gamma_{202}v_{27} +  \gamma_{204}v_{28}  \equiv_{\omega} 0,\\
&p_{(3;5)}(S) \equiv_{\omega} \gamma_{\{12,14,16,23,35,48\}}v_{2} +  \gamma_{44}v_{4} +   \gamma_{\{5,65,67,69,76,88,169\}}v_{8}\\
&\qquad +  \gamma_{\{35,90,92,96,119,123,132,158,199\}}v_{10} +  \gamma_{\{114,118,121,125\}}v_{11}\\
&\qquad +  \gamma_{\{116,119,123,132\}}v_{12} +  \gamma_{\{138,142,157,161,199\}}v_{14} +  \gamma_{44}v_{15}\\
&\qquad +  \gamma_{\{152,156,159,163\}}v_{16} +  \gamma_{\{154,157,158,161\}}v_{17} +  \gamma_{165}v_{22}\\
&\qquad +  \gamma_{\{183,187,199\}}v_{23} +  \gamma_{165}v_{24} +  \gamma_{198}v_{25} +  \gamma_{\{201,205\}}v_{26} +  \gamma_{203}v_{27} \equiv_{\omega} 0,\\
&p_{(4;5)}(S) \equiv_{\omega} \gamma_{\{29,30,31,32\}}v_{2} +  \gamma_{57}v_{4} +   \gamma_{\{82,83,84,85\}}v_{8} +  \gamma_{\{103,104,105,106\}}v_{10}\\
&\qquad +  \gamma_{129}v_{11} +  \gamma_{\{133,134,135,136\}}v_{12} +  \gamma_{\{147,148,149,150\}}v_{14} +  \gamma_{57}v_{15}\\
&\qquad +  \gamma_{\{152,153,154,155\}}v_{16} +  \gamma_{\{159,160,161,162\}}v_{17} +  \gamma_{\{192,193,194,195\}}v_{23}\\
&\qquad +  \gamma_{178}v_{22} +  \gamma_{178}v_{24} +  \gamma_{200}v_{25} +  \gamma_{\{201,202,203,204\}}v_{26} +  \gamma_{205}v_{27} \equiv_{\omega} 0.
\end{align*}
From the above equalities, we obtain
\begin{align}\label{c613}
\gamma_j = 0 \mbox{ for } j \in \{43, &44, 57, 129, 137, 140, 151, 164, 165, 178,\notag\\ &182, 185, 196, 197, 198, 200, 201, 202, 203, 204, 205\}.
\end{align}
By applying the homomorphism $p_{(1;(u,v))}$, $2\leqslant u < v \leqslant 4,$ to \eqref{ctd612} and using \eqref{c611}, \eqref{c612}, \eqref{c613}, we obtain
\begin{align*}
&p_{(1;(2,3))}(S) \equiv_{\omega} \gamma_{\{3,14,67,97\}}v_{10} +  \gamma_{\{22,75,91,97\}}v_{14} +   \gamma_{\{23,76,92,97\}}v_{15} +  \gamma_{86}v_{19}\\
&\qquad +  \gamma_{87}v_{20} +  \gamma_{88}v_{21} +  \gamma_{118}v_{23} +  \gamma_{119}v_{24} +  \gamma_{126}v_{25} +  \gamma_{130}v_{26} +  \gamma_{131}v_{27}\\
&\qquad +  \gamma_{132}v_{28} +  \gamma_{156}v_{32} +  \gamma_{157}v_{33} +  \gamma_{158}v_{34} +  \gamma_{163}v_{35} \equiv_{\omega} 0,\\
&p_{(1;(2,4))}(S) \equiv_{\omega} \gamma_{\{2,68,79\}}v_{3} +   \gamma_{\{15,37,39,41,68,98,104,120,124,127,134\}}v_{10}\\
&\qquad +  \gamma_{\{13,66,89,117,139,184\}}v_{9} +  \gamma_{\{19,37,39,72,93,98,120,127,144,189\}}v_{12}\\
&\qquad +  \gamma_{\{21,38,74,95,122\}}v_{13} +  \gamma_{\{26,39,41,79,98,104,124,127,134\}}v_{15}\\
&\qquad +  \gamma_{\{4,30,41,83,104,134,148,160,193\}}v_{17} +  \gamma_{\{5,32,85,106,136,150,162,195\}}v_{18}\\
&\qquad +  \gamma_{115}v_{22} +  \gamma_{120}v_{23} +  \gamma_{\{49,122,141,160\}}v_{24} +  \gamma_{127}v_{25} +  \gamma_{134}v_{27}\\
&\qquad +  \gamma_{136}v_{28} +  \gamma_{153}v_{30} +  \gamma_{155}v_{31} +  \gamma_{160}v_{33} +  \gamma_{162}v_{34} \equiv_{\omega} 0,\\
&p_{(1;(3,4))}(S) \equiv_{\omega} \gamma_{\{2,66,67,68,74,76,79,170\}}v_{3} +  \gamma_{\{72,75,172\}}v_{4} +   \gamma_{\{83,86,87\}}v_{6}\\
&\qquad +  \gamma_{\{85,88\}}v_{7} +  \gamma_{\{14,15,34,89,104,115,118,120,126,127,131,134,141,153,156,193\}}v_{10}\\
&\qquad +  \gamma_{\{13,89,139\}}v_{9} +  \gamma_{\{19,51,91,93,97,98,115,118,120,126,127,144,172,189\}}v_{12}\\
&\qquad +  \gamma_{\{21,92,95,117,119,122,170,184\}}v_{13} +  \gamma_{\{22,91,93,97,98,130,172,189\}}v_{14}\\
&\qquad +  \gamma_{\{23,26,34,92,95,104,124,126,127,131,134,170,184,193\}}v_{15} +  \gamma_{\{104,193\}}v_{20}\\
&\qquad +  \gamma_{\{30,33,34,104,130,131,134,148,153,156,157,160,193\}}v_{17} +  \gamma_{\{106,163,195\}}v_{21}\\
&\qquad +  \gamma_{\{32,35,106,132,136,150,155,158,162,195\}}v_{18} +  \gamma_{\{49,139,141,157,160\}}v_{24}\\
&\qquad +  \gamma_{144}v_{25} +  \gamma_{148}v_{27} +  \gamma_{150}v_{28} \equiv_{\omega} 0.
\end{align*}
These equalities imply
\begin{equation}\label{c62}
\gamma_j = 0,\ j \in \mathbb J_3 \mbox{ and } \gamma_s = \gamma_t,\ (s,t) \in \mathbb K_3,
\end{equation}
where $\mathbb J_3 = \{$83, 85, 86, 87, 88, 115, 118, 119, 120, 122, 126, 127, 130, 131, 132, 134, 136, 139, 144, 148, 150, 153, 155, 156, 157, 158, 160, 162, 163$\}$ and $\mathbb K_3 = \{$(4,33), (5,32), (5,35), (5,106), (5,195), (13,89), (30,104), (30,193), (34,41), (49,51), (49,141), (82,84), (116,123), (117,124), (133,135), (147,149), (152,154), (152,159), (152,161)$\}$.

Applying the homomorphism $p_{(1;(u,5))}$, $2\leqslant u  \leqslant 4,$ to \eqref{ctd612} and using \eqref{c611}, \eqref{c612}, \eqref{c613}, \eqref{c62} gives
\begin{align*}
&p_{(1;(2,5))}(S) \equiv_{\omega} \gamma_{\{1,69,73,78,82\}}v_{2} +   \gamma_{\{5,16,69,96,125,145,146,147,152,190,191,194\}}v_{10} \\
&\qquad +  \gamma_{\{3,12,65,90,114,116,138,152,183\}}v_{8} +  \gamma_{\{18,37,71,94,121,169,187,190,191,194,199\}}v_{11}\\
&\qquad +  \gamma_{\{20,38,39,73,96,99,116,128,145,190\}}v_{12} +  \gamma_{\{25,78,99,105,125,146,191\}}v_{14}\\
&\qquad +  \gamma_{\{4,29,82,103,133,147,152,192\}}v_{16} +  \gamma_{\{5,31,34,82,105,133,147,152,194\}}v_{17} +  \gamma_{116}v_{22}\\
&\qquad +  \gamma_{\{48,121,142,145,146,147,152,169,187,190,191,194,199\}}v_{23} +  \gamma_{116}v_{24} +  \gamma_{128}v_{25}\\
&\qquad +  \gamma_{133}v_{26} +  \gamma_{133}v_{27} +  \gamma_{152}v_{29} +  \gamma_{152}v_{30} +  \gamma_{152}v_{32} +  \gamma_{152}v_{33} \equiv_{\omega} 0,\\
&p_{(1;(3,5))}(S) \equiv_{\omega} \gamma_{\{1,3,65,67,69,71,73,75,78,82,169,172\}}v_{2} +  \gamma_{\{73,76,172\}}v_{4} +   \gamma_{82}v_{5}\\
&\qquad +  \gamma_{82}v_{6} +  \gamma_{\{12,90,116,138,152\}}v_{8} +  \gamma_{\{18,91,94,114,121,172,183,187,190,191,194,199\}}v_{11}\\
&\qquad  +  \gamma_{\{5,14,16,49,90,92,96,97,99,116,142,145,146,147,172,190,199\}}v_{10}  +  \gamma_{\{103,192\}}v_{19}\\
&\qquad +  \gamma_{\{20,49,92,96,97,99,128,145,172,190\}}v_{12} +  \gamma_{\{23,92,96,97,99,172,190\}}v_{15}\\
&\qquad +  \gamma_{\{22,25,34,91,94,105,125,172,183,187,190,191,199\}}v_{14} +  \gamma_{\{4,29,103,133,147,192\}}v_{16}\\
&\qquad +  \gamma_{\{5,31,34,105,133,147,194\}}v_{17} +  \gamma_{\{48,49,138,142,145,146,147,152,199\}}v_{23}\\
&\qquad +  \gamma_{\{105,194\}}v_{20} +  \gamma_{145}v_{25} +  \gamma_{147}v_{26} +  \gamma_{147}v_{27} \equiv_{\omega} 0,\\
&p_{(1;(4,5))}(S) \equiv_{\omega}  \gamma_{\{1,4,68,69,71,72,73,78,169,170\}}v_{2} +   \gamma_{\{2,5,68,69,74,79,169,170\}}v_{3}\\
&\qquad + \gamma_{\{65,66\}}v_{1} +  \gamma_{\{12,13,90,114,116,138,183,184\}}v_{8} +  \gamma_{\{90,117,183,184\}}v_{9}\\
&\qquad +  \gamma_{\{78,79\}}v_{4} +  \gamma_{\{15,16,30,95,96,98,99,105,121,128,133,145,146,170,187,191\}}v_{10}\\
&\qquad +  \gamma_{\{18,30,93,94,121,170,187,189,190,191,192,194,199\}}v_{11} +  \gamma_{138}v_{22}\\
&\qquad +  \gamma_{\{5,19,20,30,93,94,95,96,116,121,189,190,191,192,194,199\}}v_{12} +  \gamma_{\{142,152,199\}}v_{24}\\
&\qquad +  \gamma_{\{5,21,95,96,116,170,187\}}v_{13} +  \gamma_{\{25,30,98,99,105,117,125,133,146,191\}}v_{14}\\
&\qquad +  \gamma_{\{26,30,98,99,105,117,125,128,133,191\}}v_{15} +  \gamma_{\{29,103,133,147,152\}}v_{16} +  \gamma_{146}v_{25}\\
&\qquad +  \gamma_{\{5,30,31,103,133,147,152\}}v_{17} +  \gamma_{\{48,49,142,145,146,152,199\}}v_{23} \equiv_{\omega} 0.
\end{align*}
These equalities imply
\begin{equation}\label{c63}
\gamma_j = 0,\ j \in \mathbb J_4 \mbox{ and } \gamma_s = \gamma_t,\ (s,t) \in \mathbb K_4,
\end{equation}
where $\mathbb J_4 = \{$82, 84, 116, 123, 128, 133, 135, 138, 145, 146, 147, 149, 152, 154, 159, 161$\}$ and $\mathbb K_4 = \{$(4,29), (4,103), (4,192), (12,90), (16,96), (31,105), (31,194), (48,49), (48,142), (48,199), (65,66), (78,79)$\}$.

Now, by applying the homomorphism $p_{(i;(u,v))}$, $2\leqslant i< u < v \leqslant 5,$ to \eqref{ctd612} and using \eqref{c611}, \eqref{c612}, \eqref{c613}, \eqref{c62}, \eqref{c63}, we have
\begin{align*}
&p_{(2;(3,4))}(S) \equiv_{\omega} \gamma_{\{2,3,13,14,15,21,23,26,37,38,48\}}v_{3} +  \gamma_{\{19,22,37,39,48\}}v_{4} +   \gamma_{30}v_{6}\\
&\qquad +  \gamma_{5}v_{7} +  \gamma_{\{13,30,34,37,39,67,68,170\}}v_{10} +  \gamma_{\{48,72,91,93,97,98,172,189\}}v_{12}\\
&\qquad +  \gamma_{\{3,13,65,184\}}v_{9} +  \gamma_{\{48,74,92,95,117\}}v_{13} +  \gamma_{\{37,39,48,75,91,93,97,98\}}v_{14}\\
&\qquad +  \gamma_{\{30,34,38,39,48,76,78,92,95,117\}}v_{15} +  \gamma_{\{4,34\}}v_{17} +  \gamma_{5}v_{18} +  \gamma_{30}v_{20}\\
&\qquad +  \gamma_{5}v_{21} +  \gamma_{184}v_{24} +  \gamma_{189}v_{25} +  \gamma_{30}v_{27} +  \gamma_{5}v_{28} \equiv_{\omega} 0,\\
&p_{(2;(3,5))}(S) \equiv_{\omega} \gamma_{\{1,3,12,14,16,18,20,22,25,31,34,37,38,39\}}v_{2} +  \gamma_{\{20,23,38,39,48\}}v_{4}\\
&\qquad +  \gamma_{4}v_{5}  +  \gamma_{\{12,65,183\}}v_{8} +  \gamma_{\{5,12,16,31,38,39,67,69,92,97,99,172,187,190,191\}}v_{10}\\
&\qquad +  \gamma_{31}v_{6} +  \gamma_{\{16,48,73,92,97,99,172,190\}}v_{12} +  \gamma_{\{31,37,75,78,91,94,121,125\}}v_{14}\\
&\qquad +  \gamma_{\{71,91,94,114\}}v_{11} +  \gamma_{\{16,38,39,48,76,92,97,99\}}v_{15} +  \gamma_{4}v_{16} +  \gamma_{\{5,34\}}v_{17} +  \gamma_{4}v_{19}\\
&\qquad +  \gamma_{31}v_{20} +  \gamma_{\{31,48,169,172,183,187,190,191\}}v_{23} +  \gamma_{190}v_{25} +  \gamma_{4}v_{26} +  \gamma_{31}v_{27} \equiv_{\omega} 0,\\
&p_{(2;(4,5))}(S) \equiv_{\omega}  \gamma_{\{1,15,16,18,19,20,25,30,31,34,37,38,39\}}v_{2} +   \gamma_{\{2,15,16,21,26,37,38\}}v_{3}\\
&\qquad +  \gamma_{\{25,26,34,39\}}v_{4} +  \gamma_{\{12,13,65,114,183\}}v_{8} +  \gamma_{\{3,12,13,65,117,184\}}v_{9}\\
&\qquad +  \gamma_{\{4,5,16,37,38,68,69,95,98,99,121,189,190,191\}}v_{10} +  \gamma_{\{16,72,73,93,94,95\}}v_{12}\\
&\qquad + \gamma_{\{3,12,13\}}v_{1} +  \gamma_{\{16,74,95\}}v_{13} +  \gamma_{\{30,31,34,39,78,98,99,117,125,191\}}v_{14}\\
&\qquad +  \gamma_{\{30,31,78,98,99,117,125\}}v_{15} +  \gamma_{\{4,30,31,48,169,170,187,189,190,191\}}v_{23}\\
&\qquad  +  \gamma_{\{71,93,94\}}v_{11} +  \gamma_{\{183,184\}}v_{22} +  \gamma_{\{5,48,187\}}v_{24} +  \gamma_{191}v_{25} \equiv_{\omega} 0,\\
&p_{(3;(4,5))}(S) \equiv_{\omega} \gamma_{\{1,2,3,4,5\}}v_{1} +  \gamma_{\{5,12,14,15,16,18,22,23,25,26,34,48\}}v_{2} +  \gamma_{48}v_{4}\\
&\qquad +   \gamma_{\{4,13,14,15,16,19,20,21,22,23,25,26,30,31,34\}}v_{3}  +  \gamma_{\{4,5,65,67,68,69,71,75,76,169\}}v_{8}\\
&\qquad +  \gamma_{\{3,4,5,65,67,68,69,72,73,74,75,76,170,172\}}v_{9} +  \gamma_{\{114,117,121,125\}}v_{11}\\
&\qquad +  \gamma_{\{12,13,16,30,31,34,48,91,92,93,94,95,97,98,99,117,125\}}v_{10} +  \gamma_{121}v_{12} +  \gamma_{125}v_{13}\\
&\qquad +  \gamma_{48}v_{14} +  \gamma_{48}v_{15} +  \gamma_{\{169,170,172\}}v_{22} +  \gamma_{\{4,48,170,183,187,191\}}v_{23}\\
&\qquad +  \gamma_{\{5,30,31,170,172,184,187,189,190,191\}}v_{24} +  \gamma_{48}v_{25}  \equiv_{\omega} 0.
\end{align*}
By a direct computation using the above equalities we get $\gamma_t =0$ for all $t, \ 1 \leqslant t \leqslant 205$. The theorem is proved.
\end{proof}

\subsection{Computation of $QP_5((4)|^2|(3)|^{d-4}|(1))$}\

\medskip
Let $d \geqslant 8$. From Proposition \ref{mdd72} below and the monomials as given in Subsection \ref{s52}, we see that for any $x \in B_5((4)|^2|(3)|^{d-4})$ there exists uniquely a sequence $J_x = (j_1^x,j_2^x,j_3^x)$ such that $1 \leqslant j_1^x < j_2^x < j_3^x\leqslant 5$ and $\nu_{j_s^x}(x)> 32$ for $s = 1,\, 2,\, 3$. Consider the homomorphism $\theta_J: P_r \to P_k$ defined by \eqref{ctbs} for $r =3,\, k = 5$. The main result of this subsection is the following.

\begin{thms}\label{thd73} For any $d \geqslant 8$, we have
$$B_5((4)|^2|(3)|^{d-2}|(1)) = \left\{x\theta_{J_x}\big(x_u^{2^{d-2}}\big): x \in B_5((4)|^2|(3)|^{d-4}),\, u = 1,\, 2,\, 3\right\}.$$
Consequently, 
$\dim QP_5((4)|^2|(3)|^{d-2}|(1)) = 1395.$
\end{thms}

\begin{proof} Let $y$ be an admissible monomial of weight vector $(4)|^2|(3)|^{d-4}|(1)$, then $y = X_i\bar x^2x_t^{2^d-2}$ with $1 \leqslant i,\, t \leqslant 5$ and $\bar x$ an admissible monomial of weight vector $(4)|(3)|^{d-4}$. From our work \cite[Proposition 3.3.1]{sux} we see that there exists uniquely a sequence $J_{\bar x} = (j_1^{\bar x},j_2^{\bar x},j_3^{\bar x})$ such that $1 \leqslant j_1^{\bar x} < j_2^{\bar x} < j_3^{\bar x}\leqslant 5$ and $\nu_{j_s^{\bar x}}(\bar x)> 16$ for $s = 1,\, 2,\, 3$ and $x_t = \theta_{J_{\bar x}}(x_u)$ for suitable $1\leqslant u \leqslant 3$. Then, by Theorem \ref{dlcb1}, $x = X_i\bar x^2$ is an admissible monomial of weit vector $(4)|^2|(3)|^{d-4}$, $J_x = J_{\bar x}$ and $\nu_{j_s^x}(x)> 32$ for $s = 1,\, 2,\, 3$. Thus, we have proved that 
$$B_5((4)|^2|(3)|^{d-2}|(1)) \subset \left\{x\theta_{J_x}\left(x_u^{2^{d-2}}\right): x \in B_5((4)|^2|(3)|^{d-4}),\, u = 1,\, 2,\, 3\right\}.$$
We now prove that the set $\left\{\big[x\theta_{J_x}\big(x_u^{2^{d-2}}\big)\big]_\omega: x \in B_5((4)|^2|(3)|^{d-4}),\, u = 1,\, 2,\, 3\right\}$ is linearly independent in $QP_5(\omega)$ with $\omega := (4)|^2|(3)|^{d-4}|(1)$. Suppose there is a linear relation
\begin{equation}\label{ctd711}
\mathcal S:= \sum_{x\in B_5((4)|^2|(3)|^{d-4})\atop u = 1,2,3}\gamma_{x,u}x\theta_{J_x}\big(x_u^{2^{d-2}}\big) \equiv_{\omega} 0,
\end{equation}
with $\gamma_{x,u} \in \mathbb F_2$. Consider the set $B_4(\omega)$ as given in \cite{su2} and the homomorphism $p_{(i;I)}:P_5\to P_4$ as defined by \eqref{ct23} for $k=5$. We compute $p_{(i;I)}(\mathcal S)$ in terms of the elements in $B_4(\omega)$ (mod($P_4^-(\omega)+\mathcal A^+P_4$). From the relation $p_{(i;I)}(\mathcal S) \equiv_\omega 0$ with all $(i,I)\in \mathcal N_5$ we obtain $\gamma_{x,u} = 0$ for all $x\in B_5((4)|^2|(3)|^{d-4})$ and $u = 1,2,3$. The theorem is proved.	
\end{proof}

In the remaining part of this subsection, we explicitly determine the space $QP_5((4)|^2|(3)|^{d-4})$ for any $d \geqslant 8$. 

From our work \cite{su2}, we have $|B_4((4)|^2|(3)|^{d-4})| = |B_4^+((4)|^2|(3)|^{d-4})| = 15$, hence we get $|B_5^0((4)|^2|(3)|^{d-4})| = 15{5\choose 4} = 75$. In this subsection we prove the following.

\begin{props}\label{mdd72} For any $d \geqslant 8$, $B_5^+((4)|^2|(3)|^{d-4})$ is the set of $390$ admissible monomials which are determined as in Subsection $\ref{s52}$. Consequently, 
	$$\dim QP_5^+((4)|^2|(3)|^{d-4}) = 390,\ \dim QP_5((4)|(3)|^{d-4}) = 465.$$
\end{props}

The proposition is proved by using Theorems \ref{dlcb1}, \ref{dlsig} and some technical lemmas. 

\begin{lems}\label{bda72} Let $(i,j,t,u,v)$ be an arbitrary permutation of $(1,2,3,4,5)$. The following monomials are strictly inadmissible:
	$$x_i^2xjx_t^3x_u^3x_v^3,\, i< j;\ x_i^3x_j^4x_t^3x_u^7x_v^7,\, i<j<t.$$
\end{lems}
\begin{proof} We have
\begin{align*}
x_i^2xjx_t^3x_u^3x_v^3 = x_ixj^2x_t^3x_u^3x_v^3	+ Sq^1(x_ixjx_t^3x_u^3x_v^3)\ \mbox{mod}\big(P_5^-((4)|^2)\big).
\end{align*}
Since $i < j$, the monomial $x_i^2xjx_t^3x_u^3x_v^3$ is strictly inadmissible.
\begin{align*}
x_i^3x_j^4x_t^3x_u^7x_v^7 &= x_i^2x_j^3x_t^5x_u^7x_v^7 + x_i^2x_j^5x_t^3x_u^7x_v^7 + x_i^3x_j^3x_t^4x_u^7x_v^7\\ &\quad + Sq^1(x_i^3x_j^3x_t^3x_u^7x_v^7) + Sq^2(x_i^2x_j^3x_t^3x_u^7x_v^7)\ \mbox{mod}\big(P_5^-((4)|^2|(3)\big).
\end{align*}
Since $i < j < t$, the monomial $x_i^3x_j^4x_t^3x_u^7x_v^7$ is strictly inadmissible. The lemma is proved.
\end{proof}

\begin{lems}\label{bda73} Let $(i,j,t,u,v)$ be an arbitrary permutation of $(1,2,3,4,5)$. The following monomials are strictly inadmissible:
	
\smallskip
\ \! {\rm i)} $x_i^{7}x_j^{9}x_t^{2}x_u^{15}x_v^{15},\, x_i^{3}x_j^{7}x_t^{8}x_u^{15}x_v^{15},\, x_i^{7}x_j^{3}x_t^{8}x_u^{15}x_v^{15},\, x_i^{7}x_j^{8}x_t^{3}x_u^{15}x_v^{15},\ i<j<t.$
	
\medskip
\ {\rm ii)} $x_r^{15}\theta_{J_r}(w)$, $1 \leqslant r \leqslant 5$, with $w$ one of the monomials:
$$x_1^{7}x_2^{8}x_3^{7}x_4^{11},\, x_1^{7}x_2^{9}x_3^{3}x_4^{14},\, x_1^{7}x_2^{9}x_3^{6}x_4^{11},\, x_1^{7}x_2^{9}x_3^{7}x_4^{10},\, x_1^{7}x_2^{9}x_3^{14}x_4^{3},\, x_1^{7}x_2^{11}x_3^{7}x_4^{8}.$$
	
{\rm iii)} $x_1^{7}x_2^{9}x_3^{7}x_4^{14}x_5^{11},\, x_1^{7}x_2^{9}x_3^{14}x_4^{7}x_5^{11}$.
\end{lems}
\begin{proof} Each monomial in the lemma is of weight vector $(4)|^2(3)|^2$. We prove the lemma for some monomials. The others can be proved by a similar computation. By using the Cartan formula, we have
\begin{align*}
&x_i^{3}x_j^{7}x_t^{8}x_u^{15}x_v^{15} = x_i^{2}x_j^{7}x_t^{9}x_u^{15}x_v^{15} + x_i^{2}x_j^{9}x_t^{7}x_u^{15}x_v^{15} + x_i^{3}x_j^{4}x_t^{11}x_u^{15}x_v^{15}\\ 
&\qquad + Sq^1\big(x_i^{3}x_j^{7}x_t^{7}x_u^{15}x_v^{15}\big) +  Sq^2\big(x_i^{2}x_j^{7}x_t^{7}x_u^{15}x_v^{15}\big)\\ 
&\qquad +  Sq^4\big(x_i^{3}x_j^{4}x_t^{7}x_u^{15}x_v^{15}\big)\ \mbox{mod}\big(P_5^-((4)|^2|(3)|^{2}\big),\\
&x_i^{7}x_j^{3}x_t^{8}x_u^{15}x_v^{15} = x_i^{4}x_j^{3}x_t^{11}x_u^{15}x_v^{15} + x_i^{5}x_j^{2}x_t^{11}x_u^{15}x_v^{15} + x_i^{7}x_j^{2}x_t^{9}x_u^{15}x_v^{15}\\ 
&\qquad +  Sq^1\big(x_i^{7}x_j^{3}x_t^{7}x_u^{15}x_v^{15}\big) +  Sq^2\big(x_i^{7}x_j^{2}x_t^{7}x_u^{15}x_v^{15}\big)\\ 
&\qquad +  Sq^4\big(x_i^{4}x_j^{3}x_t^{7}x_u^{15}x_v^{15} + x_i^{5}x_j^{2}x_t^{7}x_u^{15}x_v^{15}\big)\ \mbox{mod}\big(P_5^-((4)|^2|(3)|^{2}\big),\\
&x_i^{7}x_j^{8}x_t^{3}x_u^{15}x_v^{15} = x_i^{5}x_j^{3}x_t^{10}x_u^{15}x_v^{15} + x_i^{5}x_j^{10}x_t^{3}x_u^{15}x_v^{15} + x_i^{7}x_j^{3}x_t^{8}x_u^{15}x_v^{15}\\ 
&\qquad + Sq^1\big(x_i^{7}x_j^{5}x_t^{5}x_u^{15}x_v^{15}\big) +  Sq^2\big(x_i^{7}x_j^{6}x_t^{3}x_u^{15}x_v^{15} + x_i^{7}x_j^{3}x_t^{6}x_u^{15}x_v^{15}\big) \\ 
&\qquad+  Sq^4\big(x_i^{5}x_j^{6}x_t^{3}x_u^{15}x_v^{15} + x_i^{5}x_j^{3}x_t^{6}x_u^{15}x_v^{15}\big)\ \mbox{mod}\big(P_5^-((4)|^2|(3)|^{2}\big),\\
&x_i^{7}x_j^{9}x_t^{2}x_u^{15}x_v^{15} = x_i^{4}x_j^{11}x_t^{3}x_u^{15}x_v^{15} + x_i^{5}x_j^{11}x_t^{2}x_u^{15}x_v^{15} + x_i^{7}x_j^{8}x_t^{3}x_u^{15}x_v^{15}\\ 
&\qquad + Sq^1\big(x_i^{7}x_j^{7}x_t^{4}x_u^{15}x_v^{15}\big) +  Sq^2\big(x_i^{7}x_j^{7}x_t^{2}x_u^{15}x_v^{15}\big)\\ 
&\qquad +  Sq^4\big(x_i^{4}x_j^{7}x_t^{3}x_u^{15}x_v^{15} + x_i^{5}x_j^{7}x_t^{2}x_u^{15}x_v^{15}\big)\ \mbox{mod}\big(P_5^-((4)|^2|(3)|^{2}\big).
\end{align*}
Since $i < j < t$, the monomials in Part i) are strictly inadmissible.
	
We prove Part ii) for $x = x_1^{7}x_2^{9}x_3^{3}x_4^{14}x_5^{15}$ and $y = x_1^{7}x_2^{11}x_3^{7}x_4^{8}x_5^{15}$. We have
\begin{align*}
x &= x_1^{4}x_2^{7}x_3^{11}x_4^{11}x_5^{15} + x_1^{4}x_2^{11}x_3^{11}x_4^{7}x_5^{15} + x_1^{5}x_2^{11}x_3^{3}x_4^{14}x_5^{15} + x_1^{5}x_2^{11}x_3^{6}x_4^{11}x_5^{15}\\ 
&\quad + x_1^{5}x_2^{11}x_3^{10}x_4^{7}x_5^{15} + x_1^{7}x_2^{5}x_3^{10}x_4^{11}x_5^{15} + x_1^{7}x_2^{7}x_3^{8}x_4^{11}x_5^{15} + x_1^{7}x_2^{7}x_3^{10}x_4^{9}x_5^{15}\\ 
&\quad + x_1^{7}x_2^{7}x_3^{11}x_4^{8}x_5^{15} + x_1^{7}x_2^{8}x_3^{11}x_4^{7}x_5^{15} +  Sq^1\big(x_1^{7}x_2^{7}x_3^{5}x_4^{13}x_5^{15} + x_1^{7}x_2^{7}x_3^{11}x_4^{7}x_5^{15}\big)\\ 
&\quad +  Sq^2\big(x_1^{7}x_2^{3}x_3^{6}x_4^{7}x_5^{23} +x_1^{7}x_2^{7}x_3^{3}x_4^{14}x_5^{15} + x_1^{7}x_2^{7}x_3^{6}x_4^{11}x_5^{15} + x_1^{7}x_2^{7}x_3^{10}x_4^{7}x_5^{15}\big)\\ 
&\quad +  Sq^4\big(x_1^{4}x_2^{7}x_3^{11}x_4^{7}x_5^{15} + x_1^{5}x_2^{7}x_3^{3}x_4^{14}x_5^{15} + x_1^{5}x_2^{7}x_3^{6}x_4^{11}x_5^{15} + x_1^{5}x_2^{7}x_3^{10}x_4^{7}x_5^{15}\\ 
&\quad + x_1^{11}x_2^{5}x_3^{6}x_4^{7}x_5^{15}\big) +  Sq^8\big(x_1^{7}x_2^{5}x_3^{6}x_4^{7}x_5^{15}\big)\ \mbox{mod}\big(P_5^-((4)|^2|(3)|^{2}\big),\\
y &= x_1^{4}x_2^{11}x_3^{7}x_4^{11}x_5^{15} + x_1^{4}x_2^{11}x_3^{11}x_4^{7}x_5^{15} + x_1^{7}x_2^{10}x_3^{7}x_4^{9}x_5^{15} + x_1^{7}x_2^{10}x_3^{9}x_4^{7}x_5^{15}\\ 
&\quad + x_1^{5}x_2^{10}x_3^{7}x_4^{11}x_5^{15} + x_1^{5}x_2^{10}x_3^{11}x_4^{7}x_5^{15} + x_1^{7}x_2^{7}x_3^{8}x_4^{11}x_5^{15} + x_1^{7}x_2^{11}x_3^{4}x_4^{11}x_5^{15}\\ 
&\quad + Sq^1\big(x_1^{7}x_2^{11}x_3^{7}x_4^{7}x_5^{15}\big) + Sq^2\big(x_1^{7}x_2^{10}x_3^{7}x_4^{7}x_5^{15} + x_1^{7}x_2^{7}x_3^{2}x_4^{7}x_5^{23}\big)\\ 
&\quad +  Sq^4\big(x_1^{4}x_2^{11}x_3^{7}x_4^{7}x_5^{15} + x_1^{5}x_2^{10}x_3^{7}x_4^{7}x_5^{15} + x_1^{11}x_2^{7}x_3^{4}x_4^{7}x_5^{15}\big)\\ 
&\quad +  Sq^8\big(x_1^{7}x_2^{7}x_3^{4}x_4^{7}x_5^{15}\big)\ \mbox{mod}\big(P_5^-((4)|^2|(3)|^{2}\big).
\end{align*}
The above equalities show that the monomials $x,\, y$ are strictly inadmissible. 
	
Based on the Cartan formula we get
\begin{align*}
&x_1^{7}x_2^{9}x_3^{7}x_4^{14}x_5^{11} = x_1^{5}x_2^{7}x_3^{11}x_4^{11}x_5^{14} + x_1^{5}x_2^{7}x_3^{11}x_4^{14}x_5^{11} + x_1^{5}x_2^{11}x_3^{7}x_4^{11}x_5^{14}\\ 
&\qquad + x_1^{5}x_2^{11}x_3^{7}x_4^{14}x_5^{11} + x_1^{7}x_2^{7}x_3^{9}x_4^{11}x_5^{14} + x_1^{7}x_2^{7}x_3^{9}x_4^{14}x_5^{11} + x_1^{7}x_2^{9}x_3^{7}x_4^{11}x_5^{14}\\ 
&\qquad + Sq^1\big(x_1^{7}x_2^{7}x_3^{7}x_4^{13}x_5^{13}\big) +  Sq^2\big(x_1^{7}x_2^{7}x_3^{7}x_4^{11}x_5^{14} + x_1^{7}x_2^{7}x_3^{7}x_4^{14}x_5^{11}\big)\\ 
&\qquad +  Sq^4\big(x_1^{5}x_2^{7}x_3^{7}x_4^{11}x_5^{14} + x_1^{5}x_2^{7}x_3^{7}x_4^{14}x_5^{11}\big)\ \mbox{mod}\big(P_5^-((4)|^2|(3)|^{2}\big).
\end{align*}
This equality shows that the monomial $x_1^{7}x_2^{9}x_3^{7}x_4^{14}x_5^{11}$ is strictly inadmissible.
\end{proof}

\begin{lems}\label{bda74}The following monomials are strictly inadmissible:
	
\medskip
\ \! {\rm i)} $x_r^{31}\theta_{J_r}(w)$, $1 \leqslant r \leqslant 5$, with $w$ one of the monomials:
	
\medskip
\centerline{\begin{tabular}{llll} 
$x_1^{7}x_2^{7}x_3^{27}x_4^{24}$& $x_1^{7}x_2^{15}x_3^{16}x_4^{27}$& $x_1^{7}x_2^{15}x_3^{17}x_4^{26}$& $x_1^{15}x_2^{7}x_3^{16}x_4^{27} $\cr  $x_1^{15}x_2^{7}x_3^{17}x_4^{26}$& $x_1^{15}x_2^{15}x_3^{16}x_4^{19}$& $x_1^{15}x_2^{15}x_3^{17}x_4^{18}$& $x_1^{15}x_2^{15}x_3^{19}x_4^{16} $.\cr
\end{tabular}}
	
\medskip
\ {\rm ii)} $x_1^{7}x_2^{15}x_3^{17}x_4^{30}x_5^{27}$, $x_1^{15}x_2^{7}x_3^{17}x_4^{30}x_5^{27}$, $x_1^{15}x_2^{15}x_3^{17}x_4^{30}x_5^{19}$, $x_1^{15}x_2^{19}x_3^{15}x_4^{21}x_5^{26}$.
\end{lems}

\begin{proof} It is easy to see that each monomial in this lemma is of weight vector $(4)|^2(3)|^3$. We prove the lemma for some monomials. The others can be proved by a similar computation. 
	
We prove Part i) for $r = 5$ and $w = x_1^{7}x_2^{7}x_3^{27}x_4^{24},\, x_1^{7}x_2^{15}x_3^{17}x_4^{26}$. A direct computation using the Cartan formula gives  
\begin{align*}
&x_1^{7}x_2^{7}x_3^{27}x_4^{24}x_5^{31} = x_1^{4}x_2^{7}x_3^{27}x_4^{27}x_5^{31} + x_1^{4}x_2^{11}x_3^{27}x_4^{23}x_5^{31} + x_1^{5}x_2^{3}x_3^{27}x_4^{30}x_5^{31}\\ 
&\qquad + x_1^{5}x_2^{3}x_3^{30}x_4^{27}x_5^{31} + x_1^{5}x_2^{6}x_3^{27}x_4^{27}x_5^{31} + x_1^{5}x_2^{10}x_3^{27}x_4^{23}x_5^{31} + x_1^{5}x_2^{11}x_3^{19}x_4^{30}x_5^{31}\\ 
&\qquad + x_1^{5}x_2^{11}x_3^{22}x_4^{27}x_5^{31} + x_1^{5}x_2^{11}x_3^{26}x_4^{23}x_5^{31} + x_1^{7}x_2^{3}x_3^{25}x_4^{30}x_5^{31} + x_1^{7}x_2^{3}x_3^{30}x_4^{25}x_5^{31}\\ 
&\qquad + x_1^{7}x_2^{5}x_3^{26}x_4^{27}x_5^{31} + x_1^{7}x_2^{6}x_3^{27}x_4^{27}x_5^{31} + x_1^{7}x_2^{7}x_3^{24}x_4^{27}x_5^{31} + x_1^{7}x_2^{7}x_3^{26}x_4^{25}x_5^{31}\\ 
&\qquad +  Sq^1\big(x_1^{7}x_2^{5}x_3^{29}x_4^{23}x_5^{31} + x_1^{7}x_2^{7}x_3^{21}x_4^{29}x_5^{31} + x_1^{7}x_2^{7}x_3^{27}x_4^{23}x_5^{31}\big)\\ 
&\qquad + Sq^2\big(x_1^{7}x_2^{3}x_3^{15}x_4^{30}x_5^{39} + x_1^{7}x_2^{3}x_3^{15}x_4^{38}x_5^{31} + x_1^{7}x_2^{3}x_3^{22}x_4^{23}x_5^{39} + x_1^{7}x_2^{3}x_3^{23}x_4^{30}x_5^{31}\\ 
&\qquad + x_1^{7}x_2^{3}x_3^{30}x_4^{23}x_5^{31} + x_1^{7}x_2^{6}x_3^{27}x_4^{23}x_5^{31} + x_1^{7}x_2^{7}x_3^{19}x_4^{30}x_5^{31} + x_1^{7}x_2^{7}x_3^{22}x_4^{27}x_5^{31}\\ 
&\qquad + x_1^{7}x_2^{7}x_3^{26}x_4^{23}x_5^{31}\big) +  Sq^4\big(x_1^{4}x_2^{7}x_3^{27}x_4^{23}x_5^{31} + x_1^{5}x_2^{3}x_3^{23}x_4^{30}x_5^{31} + x_1^{5}x_2^{3}x_3^{30}x_4^{23}x_5^{31}\\ 
&\qquad + x_1^{5}x_2^{6}x_3^{27}x_4^{23}x_5^{31} + x_1^{5}x_2^{7}x_3^{19}x_4^{30}x_5^{31} + x_1^{5}x_2^{7}x_3^{22}x_4^{27}x_5^{31}\\ 
&\qquad + x_1^{5}x_2^{7}x_3^{26}x_4^{23}x_5^{31} + x_1^{11}x_2^{5}x_3^{15}x_4^{30}x_5^{31} + x_1^{11}x_2^{5}x_3^{22}x_4^{23}x_5^{31}\big)\\ 
&\qquad + Sq^8\big(x_1^{7}x_2^{5}x_3^{15}x_4^{30}x_5^{31} + x_1^{7}x_2^{5}x_3^{22}x_4^{23}x_5^{31}\big)\ \mbox{mod}\big(P_5^-((4)|^2|(3)|^{3}\big)\\
&x_1^{7}x_2^{15}x_3^{17}x_4^{26}x_5^{31} = x_1^{4}x_2^{15}x_3^{19}x_4^{27}x_5^{31} + x_1^{4}x_2^{19}x_3^{15}x_4^{27}x_5^{31} + x_1^{5}x_2^{15}x_3^{19}x_4^{26}x_5^{31}\\ 
&\qquad + x_1^{5}x_2^{19}x_3^{15}x_4^{26}x_5^{31} + x_1^{7}x_2^{8}x_3^{23}x_4^{27}x_5^{31} + x_1^{7}x_2^{9}x_3^{23}x_4^{26}x_5^{31}\\ 
&\qquad + x_1^{7}x_2^{15}x_3^{16}x_4^{27}x_5^{31} + Sq^1\big(x_1^{7}x_2^{15}x_3^{15}x_4^{27}x_5^{31}\big) +  Sq^2\big(x_1^{7}x_2^{15}x_3^{15}x_4^{26}x_5^{31}\big)\\ 
&\qquad +  Sq^4\big(x_1^{4}x_2^{15}x_3^{15}x_4^{27}x_5^{31} + x_1^{5}x_2^{15}x_3^{15}x_4^{26}x_5^{31} + Sq^8\big(x_1^{7}x_2^{8}x_3^{15}x_4^{27}x_5^{31}\\ 
&\qquad + x_1^{7}x_2^{9}x_3^{15}x_4^{26}x_5^{31}\big)\ \mbox{mod}\big(P_5^-((4)|^2|(3)|^{3}\big).
\end{align*}
Hence, the monomials $x_1^{7}x_2^{7}x_3^{27}x_4^{24}x_5^{31},\, x_1^{7}x_2^{15}x_3^{17}x_4^{26}x_5^{31}$ are strictly inadmissible.
	
We prove Part ii) for $x_1^{7}x_2^{15}x_3^{17}x_4^{30}x_5^{27}$. We have 
\begin{align*}
&x_1^{7}x_2^{15}x_3^{17}x_4^{30}x_5^{27} = x_1^{5}x_2^{15}x_3^{19}x_4^{27}x_5^{30} + x_1^{5}x_2^{15}x_3^{19}x_4^{30}x_5^{27} + x_1^{5}x_2^{19}x_3^{15}x_4^{27}x_5^{30}\\ 
&\qquad + x_1^{5}x_2^{19}x_3^{15}x_4^{30}x_5^{27} + x_1^{7}x_2^{9}x_3^{23}x_4^{27}x_5^{30} + x_1^{7}x_2^{9}x_3^{23}x_4^{30}x_5^{27}\\ 
&\qquad + x_1^{7}x_2^{15}x_3^{17}x_4^{27}x_5^{30} +  Sq^1\big(x_1^{7}x_2^{15}x_3^{15}x_4^{29}x_5^{29}\big) +  Sq^2\big(x_1^{7}x_2^{15}x_3^{15}x_4^{27}x_5^{30}\\ 
&\qquad + x_1^{7}x_2^{15}x_3^{15}x_4^{30}x_5^{27}\big) +  Sq^4\big(x_1^{5}x_2^{15}x_3^{15}x_4^{27}x_5^{30} + x_1^{5}x_2^{15}x_3^{15}x_4^{30}x_5^{27}\big)\\ 
&\qquad +  Sq^8\big(x_1^{7}x_2^{9}x_3^{15}x_4^{27}x_5^{30} + x_1^{7}x_2^{9}x_3^{15}x_4^{30}x_5^{27}\big)\ \mbox{mod}\big(P_5^-((4)|^2|(3)|^{3}\big).
\end{align*}
This equality implies that the monomial $x_1^{7}x_2^{15}x_3^{17}x_4^{30}x_5^{27}$ is strictly inadmissible.
\end{proof}

\begin{lems}\label{bda75}The following monomials are strictly inadmissible:
	
\medskip
\centerline{\begin{tabular}{llll} 
$x_1^{7}x_2^{15}x_3^{55}x_4^{59}x_5^{56}$& $x_1^{15}x_2^{7}x_3^{55}x_4^{59}x_5^{56}$& $x_1^{15}x_2^{15}x_3^{48}x_4^{55}x_5^{59}$& $x_1^{15}x_2^{15}x_3^{55}x_4^{48}x_5^{59} $\cr  $x_1^{15}x_2^{15}x_3^{55}x_4^{49}x_5^{58}$& $x_1^{15}x_2^{15}x_3^{55}x_4^{59}x_5^{48}$& $x_1^{15}x_2^{31}x_3^{35}x_4^{53}x_5^{58}$& $x_1^{15}x_2^{55}x_3^{7}x_4^{59}x_5^{56} $\cr  $x_1^{31}x_2^{15}x_3^{35}x_4^{53}x_5^{58}$& $x_1^{31}x_2^{31}x_3^{35}x_4^{37}x_5^{58}$.& &\cr
\end{tabular}}
\end{lems}
\begin{proof} Observe that all monomials in this lemma are of weight vector $(4)|^2(3)|^4$. We prove the lemma for $u = x_1^{7}x_2^{15}x_3^{55}x_4^{59}x_5^{56}$ and $v = x_1^{31}x_2^{31}x_3^{35}x_4^{37}x_5^{58}$. By a direct computation using the Cartan we have  
\begin{align*}
&u = x_1^{4}x_2^{15}x_3^{55}x_4^{59}x_5^{59} + x_1^{4}x_2^{15}x_3^{59}x_4^{55}x_5^{59} + x_1^{5}x_2^{15}x_3^{51}x_4^{59}x_5^{62} + x_1^{5}x_2^{15}x_3^{51}x_4^{62}x_5^{59}\\ 
&\quad + x_1^{5}x_2^{15}x_3^{54}x_4^{59}x_5^{59} + x_1^{5}x_2^{15}x_3^{58}x_4^{55}x_5^{59} + x_1^{5}x_2^{15}x_3^{59}x_4^{55}x_5^{58} + x_1^{5}x_2^{15}x_3^{59}x_4^{59}x_5^{54}\\ 
&\quad + x_1^{5}x_2^{15}x_3^{59}x_4^{62}x_5^{51} + x_1^{7}x_2^{11}x_3^{53}x_4^{62}x_5^{59} + x_1^{7}x_2^{11}x_3^{57}x_4^{62}x_5^{55} + x_1^{7}x_2^{15}x_3^{49}x_4^{62}x_5^{59}\\ 
&\quad + x_1^{7}x_2^{15}x_3^{51}x_4^{57}x_5^{62} + x_1^{7}x_2^{15}x_3^{51}x_4^{62}x_5^{57} + x_1^{7}x_2^{15}x_3^{53}x_4^{59}x_5^{58} + x_1^{7}x_2^{15}x_3^{54}x_4^{57}x_5^{59}\\ 
&\quad + x_1^{7}x_2^{15}x_3^{55}x_4^{56}x_5^{59} + x_1^{7}x_2^{15}x_3^{55}x_4^{57}x_5^{58} +  Sq^1\big(x_1^{7}x_2^{15}x_3^{53}x_4^{55}x_5^{61} + x_1^{7}x_2^{15}x_3^{55}x_4^{55}x_5^{59}\\ 
&\quad + x_1^{7}x_2^{15}x_3^{55}x_4^{61}x_5^{53}\big) +  Sq^2\big(x_1^{7}x_2^{15}x_3^{43}x_4^{70}x_5^{55} + x_1^{7}x_2^{15}x_3^{43}x_4^{78}x_5^{47}\\ 
&\quad + x_1^{7}x_2^{15}x_3^{51}x_4^{55}x_5^{62} + x_1^{7}x_2^{15}x_3^{51}x_4^{62}x_5^{55} + x_1^{7}x_2^{15}x_3^{51}x_4^{70}x_5^{47} + x_1^{7}x_2^{15}x_3^{54}x_4^{55}x_5^{59}\\ 
&\quad + x_1^{7}x_2^{15}x_3^{55}x_4^{55}x_5^{58} + x_1^{7}x_2^{15}x_3^{55}x_4^{59}x_5^{54} + x_1^{7}x_2^{15}x_3^{55}x_4^{62}x_5^{51} + x_1^{7}x_2^{23}x_3^{51}x_4^{55}x_5^{54}\big)\\ 
&\quad +  Sq^4\big(x_1^{4}x_2^{15}x_3^{55}x_4^{55}x_5^{59} + x_1^{5}x_2^{15}x_3^{43}x_4^{78}x_5^{47} + x_1^{5}x_2^{15}x_3^{51}x_4^{55}x_5^{62} + x_1^{5}x_2^{15}x_3^{51}x_4^{62}x_5^{55}\\ 
&\quad + x_1^{5}x_2^{15}x_3^{54}x_4^{55}x_5^{59} + x_1^{5}x_2^{15}x_3^{55}x_4^{55}x_5^{58} + x_1^{5}x_2^{15}x_3^{55}x_4^{59}x_5^{54} + x_1^{5}x_2^{15}x_3^{55}x_4^{62}x_5^{51}\\ 
&\quad + x_1^{11}x_2^{15}x_3^{45}x_4^{62}x_5^{55} + x_1^{11}x_2^{15}x_3^{53}x_4^{55}x_5^{54} + x_1^{11}x_2^{15}x_3^{53}x_4^{62}x_5^{47}\big)\\ 
&\quad +  Sq^8\big(x_1^{7}x_2^{11}x_3^{45}x_4^{62}x_5^{59} + x_1^{7}x_2^{11}x_3^{57}x_4^{62}x_5^{47} + x_1^{7}x_2^{15}x_3^{53}x_4^{55}x_5^{54} + x_1^{7}x_2^{23}x_3^{45}x_4^{62}x_5^{47}\big)\\ 
&\quad +  Sq^{16}\big(x_1^{7}x_2^{15}x_3^{45}x_4^{62}x_5^{47}\big)\ \mbox{mod}\big(P_5^-((4)|^2|(3)|^{4}\big),\\
&v = x_1^{16}x_2^{31}x_3^{47}x_4^{39}x_5^{59} + x_1^{16}x_2^{47}x_3^{31}x_4^{39}x_5^{59} + x_1^{17}x_2^{31}x_3^{47}x_4^{39}x_5^{58} + x_1^{17}x_2^{47}x_3^{31}x_4^{39}x_5^{58}\\ 
&\quad + x_1^{19}x_2^{31}x_3^{47}x_4^{36}x_5^{59} + x_1^{19}x_2^{31}x_3^{47}x_4^{37}x_5^{58} + x_1^{19}x_2^{47}x_3^{31}x_4^{36}x_5^{59} + x_1^{19}x_2^{47}x_3^{31}x_4^{37}x_5^{58}\\ 
&\quad + x_1^{23}x_2^{24}x_3^{47}x_4^{39}x_5^{59} + x_1^{23}x_2^{25}x_3^{47}x_4^{39}x_5^{58} + x_1^{23}x_2^{27}x_3^{47}x_4^{36}x_5^{59} + x_1^{23}x_2^{27}x_3^{47}x_4^{37}x_5^{58}\\ 
&\quad + x_1^{23}x_2^{40}x_3^{31}x_4^{39}x_5^{59} + x_1^{23}x_2^{41}x_3^{31}x_4^{39}x_5^{58} + x_1^{23}x_2^{43}x_3^{31}x_4^{36}x_5^{59} + x_1^{23}x_2^{43}x_3^{31}x_4^{37}x_5^{58}\\ 
&\quad + x_1^{31}x_2^{24}x_3^{39}x_4^{39}x_5^{59} + x_1^{31}x_2^{25}x_3^{39}x_4^{39}x_5^{58} + x_1^{31}x_2^{27}x_3^{39}x_4^{36}x_5^{59} + x_1^{31}x_2^{27}x_3^{39}x_4^{37}x_5^{58}\\ 
&\quad + x_1^{31}x_2^{31}x_3^{32}x_4^{39}x_5^{59} + x_1^{31}x_2^{31}x_3^{33}x_4^{39}x_5^{58} + x_1^{31}x_2^{31}x_3^{35}x_4^{36}x_5^{59}\\ 
&\quad + Sq^1\big(x_1^{31}x_2^{31}x_3^{31}x_4^{39}x_5^{59}\big) +  Sq^2\big(x_1^{31}x_2^{31}x_3^{31}x_4^{39}x_5^{58}\big) +  Sq^4\big(x_1^{31}x_2^{31}x_3^{31}x_4^{36}x_5^{59}\\ 
&\quad + x_1^{31}x_2^{31}x_3^{31}x_4^{37}x_5^{58}\big) +  Sq^8\big(x_1^{31}x_2^{24}x_3^{31}x_4^{39}x_5^{59} + x_1^{31}x_2^{25}x_3^{31}x_4^{39}x_5^{58}\\ 
&\quad + x_1^{31}x_2^{27}x_3^{31}x_4^{36}x_5^{59} + x_1^{31}x_2^{27}x_3^{31}x_4^{37}x_5^{58}\big) +  Sq^{16}\big(x_1^{16}x_2^{31}x_3^{31}x_4^{39}x_5^{59}\\ 
&\quad + x_1^{17}x_2^{31}x_3^{31}x_4^{39}x_5^{58} + x_1^{19}x_2^{31}x_3^{31}x_4^{36}x_5^{59} + x_1^{19}x_2^{31}x_3^{31}x_4^{37}x_5^{58} + x_1^{23}x_2^{24}x_3^{31}x_4^{39}x_5^{59}\\ 
&\quad + x_1^{23}x_2^{25}x_3^{31}x_4^{39}x_5^{58} + x_1^{23}x_2^{27}x_3^{31}x_4^{36}x_5^{59} + x_1^{23}x_2^{27}x_3^{31}x_4^{37}x_5^{58}\big)\ \mbox{mod}\big(P_5^-((4)|^2|(3)|^{4}\big).
\end{align*}
Thus, the monomials $u, v$ are strictly inadmissible.
\end{proof}
\begin{proof}[Proof of Proposition $\ref{mdd72}$] Let $\bar A(d)$ and $\bar B(d)$ be as in Subsection \ref{s52} with $d \geqslant 8$ and let $x \in P_5^+(\bar\omega)$ be an admissible monomial with $\bar\omega := (4)|^2|(3)|^{d-4}$. Then $x = X_ry^2$ with $1 \leqslant r \leqslant 5$ and $y$ an admissible monomial of weight vector $(4)|(3)|^{d-4}$. By a direct computation we see that if $x \notin \bar A(d)\cup \bar B(d)$, then there is a monomial $w$ as given in one of Lemmas \ref{bda72}, \ref{bda73}, \ref{bda74} and \ref{bda75} such that $x = wz_1^{2^u}$ with $u$ nonnegative integers, $2\leqslant u \leqslant 6$, and $z_1$ a monomial of weight vector $(3)|^{d-u}$. By Theorem \ref{dlcb1}, $x$ is inadmissible. This contradicts the fact that $x$ is admissible. Hence, $B_5^+(\bar\omega) \subset \bar A(d)\cup \bar B(d)$.
	
Now we prove that the set $[\bar A(d)\cup \bar B(d)]_{\bar\omega}$ is linearly independent in $QP_5(\bar\omega)$.
	
Consider the subspaces $\langle [\bar A(d)]_{\bar\omega}\rangle \subset QP_5(\bar\omega)$ and $\langle [\bar B(d)]_{\bar\omega}\rangle \subset QP_5(\bar\omega)$. It is easy to see that for any $x\in \bar A(d)$, we have $x = x_i^{2^{d-2}-1}\theta_{J_i}(y)$ with $y$ an admissible monomial of weight vector $(3)|^2|(2)|^{d-4}$ in $P_4$. By Proposition \ref{mdmo}, $x$ is admissible. This implies $\dim \langle [\bar A(d)]_{\bar\omega}\rangle = 275$. Since $\nu(x) = 2^{d-2}-1$ for all $x\in \bar A(d)$ and $\nu(x) < 2^{d-2}-1$ for all $x\in \bar B(d)$, we obtain $\langle [\bar A(d)]_{\bar\omega}\rangle \cap \langle [\bar B(d)]_{\bar\omega}\rangle = \{0\}$. Hence, we need only to prove the set $[\bar B(d)]_{\bar\omega}=\{[\bar b_{d,t}]_{\bar\omega}: 1 \leqslant t \leqslant 115\}$ is linearly independent in $QP_5(\bar\omega)$, where the monomials $\bar b_{d,t}: 1 \leqslant t \leqslant 115$, are determined as in Subsection \ref{s52}. 
	
Suppose there is a linear relation
\begin{equation}\label{ctd712}
\mathcal S:= \sum_{1\leqslant t \leqslant 115}\gamma_t\bar b_{d,t} \equiv_{\bar\omega} 0,
\end{equation}
where $\gamma_t \in \mathbb F_2$. We denote $\gamma_{\mathbb J} = \sum_{t \in \mathbb J}\gamma_t$ for any $\mathbb J \subset \{t\in \mathbb N:1\leqslant t \leqslant 115\}$.
	
Let $w_u = w_{d,u},\, 1\leqslant u \leqslant 15$, be as in Subsection \ref{s52} and the homomorphism $p_{(i;I)}:P_5\to P_4$ which is defined by \eqref{ct23} for $k=5$. From Lemma \ref{bdm}, we see that $p_{(i;I)}$ passes to a homomorphism from $QP_5(\bar\omega)$ to $QP_4(\bar\omega)$. By applying $p_{(i;j)}$, $1\leqslant i < j \leqslant 5,$ to (\ref{ctd712}), we obtain
\begin{align*}
&p_{(1;2)}(S) \equiv_{\bar\omega} \gamma_{8}w_{14} +  \gamma_{4}w_{15}    \equiv_{\bar\omega} 0,\\
&p_{(1;3)}(S) \equiv_{\bar\omega} \gamma_{3}w_{13} +  \gamma_{5}w_{15}    \equiv_{\bar\omega} 0,\\
&p_{(1;4)}(S) \equiv_{\bar\omega} \gamma_{\{2,53\}}w_{12} +  \gamma_{6}w_{15} \equiv_{\bar\omega} 0,\\
&p_{(1;5)}(S) \equiv_{\bar\omega} \gamma_{\{1,18\}}w_{11} +  \gamma_{7}w_{15} \equiv_{\bar\omega} 0,\\
&p_{(2;3)}(S) \equiv_{\bar\omega} \gamma_{75}w_{13} +  \gamma_{105}w_{15} \equiv_{\bar\omega} 0,\\
&p_{(2;4)}(S) \equiv_{\bar\omega} \gamma_{\{74,86\}}w_{12} +  \gamma_{106}w_{15} \equiv_{\bar\omega} 0,\\
&p_{(2;5)}(S) \equiv_{\bar\omega} \gamma_{\{73,79\}}w_{11} +  \gamma_{107}w_{15}  \equiv_{\bar\omega} 0,\\
&p_{(3;4)}(S) \equiv_{\bar\omega} \gamma_{\{94,98\}}w_{12} +  \gamma_{112}w_{15}  \equiv_{\bar\omega} 0,\\
&p_{(3;5)}(S) \equiv_{\bar\omega} \gamma_{\{93,96\}}w_{11} +  \gamma_{113}w_{15}  \equiv_{\bar\omega} 0,\\
&p_{(4;5)}(S) \equiv_{\bar\omega} \gamma_{\{101,102,103,104\}}w_{11} +  \gamma_{115}w_{15} \equiv_{\bar\omega} 0.
\end{align*}
	
From the above equalities, we get
\begin{equation}\label{c71}
\gamma_j = 0,\ j \in \mathbb J_5 \mbox{ and } \gamma_s = \gamma_t,\ (s,t) \in \mathbb K_5,
\end{equation}
where $\mathbb J_5 = \{$3, 4, 5, 6, 7, 8, 75, 105, 106, 107, 112, 113, 115$\}$ and $\mathbb K_5 = \{$(1,18), (2,53), (73,79), (74,86), (93,96), (94,98)$\}$.
	
By applying the homomorphism $p_{(1;(u,v))}$, $2\leqslant u < v \leqslant 5,$ to \eqref{ctd712} and using \eqref{c71}, we obtain
\begin{align*}
p_{(1;(2,3))}(S) &\equiv_{\bar\omega} \gamma_{33}w_{10} +  \gamma_{14}w_{13} +   \gamma_{30}w_{14} +  \gamma_{22}w_{15}   \equiv_{\bar\omega} 0,\\
p_{(1;(2,4))}(S) &\equiv_{\bar\omega} \gamma_{\{29,94\}}w_{9} +  \gamma_{13}w_{12} +  \gamma_{31}w_{14} +  \gamma_{23}w_{15}  \equiv_{\bar\omega} 0,\\
p_{(1;(3,4))}(S) &\equiv_{\bar\omega} \gamma_{\{11,38,74\}}w_{7} +  \gamma_{\{16,89\}}w_{12} +  \gamma_{\{20,89\}}w_{13} +  \gamma_{25}w_{15}  \equiv_{\bar\omega} 0,\\
p_{(1;(2,5))}(S) &\equiv_{\bar\omega} \gamma_{\{28,64\}}w_{8} +  \gamma_{\{12,42,46\}}w_{11} +  \gamma_{32}w_{14} +  \gamma_{24}w_{15}   \equiv_{\bar\omega} 0,\\
p_{(1;(3,5))}(S) &\equiv_{\bar\omega} \gamma_{\{10,37,40,49,71,84\}}w_{6} +  \gamma_{\{15,51,57,67,90\}}w_{11}\\ 
&\hskip4cm +  \gamma_{\{21,90\}}w_{13} +  \gamma_{26}w_{15} \equiv_{\bar\omega} 0,\\
p_{(1;(4,5))}(S) &\equiv_{\bar\omega} \gamma_{\{9,39,40,41,45,46,47,50,51,52,63,64,65,88,91,92,100\}}w_{5} +  \gamma_{27}w_{15}\\ 
&\quad  +  \gamma_{\{2,17,54,55,58,59,68,101,102,103\}}w_{11} +  \gamma_{\{19,54,104\}}w_{12}  \equiv_{\bar\omega} 0.
\end{align*}
	
These equalities imply
\begin{equation}\label{c73}
\gamma_j = 0,\ j \in \mathbb J_6 \mbox{ and } \gamma_s = \gamma_t,\ (s,t) \in \mathbb K_6,
\end{equation}
where $\mathbb J_6 = \{$13, 14, 22, 23, 24, 25, 26, 27, 30, 31, 32, 33$\}$ and $\mathbb K_6 = \{$(16,20), (16,89), (21,90), (28,64), (29,94)$\}$.
	
By applying the homomorphism $p_{(i;(u,v))}$, $2\leqslant i< u < v \leqslant 5,$ to \eqref{ctd712} and using \eqref{c71}, \eqref{c73}, we have
\begin{align*}
p_{(2;(3,4))}(S) &\equiv_{\bar\omega} \gamma_{\{2,36,38,44,48,56,62,66,70\}}w_{7}\\ 
&\hskip2cm  +  \gamma_{\{16,77\}}w_{12} +   \gamma_{\{16,81\}}w_{13} +  \gamma_{108}w_{15} \equiv_{\bar\omega} 0,\\
p_{(2;(3,5))}(S) &\equiv_{\bar\omega}   \gamma_{\{28,35,37,40,42,43,46,51,61\}}w_{6}\\ 
&\hskip2cm +  \gamma_{\{21,76,84\}}w_{11} +  \gamma_{\{21,82\}}w_{13} +  \gamma_{109}w_{15}   \equiv_{\bar\omega} 0,\\
p_{(2;(4,5))}(S) &\equiv_{\bar\omega} \gamma_{\{34,39,40,41,55,83,84,85,91\}}w_{5} +  \gamma_{\{74,78,87,88,101,102,103,111\}}w_{11}\\ 
&\hskip2cm +  \gamma_{\{80,87,104\}}w_{12} +  \gamma_{110}w_{15}    \equiv_{\bar\omega} 0,\\
p_{(3;(4,5))}(S) &\equiv_{\bar\omega} \gamma_{\{28,60,61,62,63,65,66,67,68,69,70,71,72\}}w_{5} +  \gamma_{\{29,95,99,101,102\}}w_{11}\\ 
&\hskip2cm +  \gamma_{\{97,99,100,103,104\}}w_{12} +  \gamma_{114}w_{15}  \equiv_{\bar\omega} 0.
\end{align*}
The above equalities imply
\begin{equation}\label{c74}
\gamma_j = 0,\ j \in \{108,\, 109,\, 110,\, 114\},\ \gamma_{81} = \gamma_{77} = \gamma_{16},\, \gamma_{82} = \gamma_{21}.
\end{equation}
By applying the homomorphisms $p_{(1;(2,3,4))}$, $p_{(1;(2,3,5))}$ to \eqref{ctd612} and using \eqref{c71}, \eqref{c73}, \eqref{c74}, we obtain
\begin{align*}
p_{(1;(2,3,4))}(S) &\equiv_{\bar\omega} \gamma_{72}w_{4} +  \gamma_{\{11,36,38\}}w_{7} +   \gamma_{\{29,62\}}w_{9} +  \gamma_{70}w_{10}\\ &\qquad +  \gamma_{\{16,44\}}w_{12} +  \gamma_{\{16,48\}}w_{13} +  \gamma_{66}w_{14} +  \gamma_{56}w_{15} \equiv_{\bar\omega} 0,\\
p_{(1;(2,3,5))}(S) &\equiv_{\bar\omega} \gamma_{\{69,102\}}w_{3} +  \gamma_{\{10,21,35,37,40,49,71,76,84\}}w_{6} +   \gamma_{61}w_{8} \\ 
&\qquad +  \gamma_{71}w_{10} +  \gamma_{\{12,15,21,42,43,46,51,57,67,76,84\}}w_{11}\\ &\qquad +  \gamma_{\{21,49\}}w_{13} +  \gamma_{67}w_{14} +  \gamma_{57}w_{15} \equiv_{\bar\omega} 0.
\end{align*}
A direct computation from the above equalities shows
\begin{equation}\label{c75}
\begin{cases}
\gamma_j = 0,\ j \in \{56, 57, 61, 66, 67, 70, 71, 72\},\\ \gamma_{s} = \gamma_{t},\ (s,t) \in \{(16,44), (16,48), (21,49), (29,62), (69,102)\}.
\end{cases}
\end{equation}
Applying the homomorphisms $p_{(1;(2,4,5))}$, $p_{(1;(3,4,5))}$ and $p_{(2;(3,4,5))}$ to \eqref{ctd712} and using \eqref{c71}-\eqref{c75}, we have
\begin{align*}
&p_{(1;(2,4,5))}(S) \equiv_{\bar\omega} \gamma_{\{60,93,95,97\}}w_{2} +   \gamma_{\{29,63,69,99,100,101,103\}}w_{8} +  \gamma_{\{65,99,104\}}w_{9}\\
&\quad +  \gamma_{\{9,28,34,39,40,41,42,45,46,47,50,51,52,63,65,74,78,80,83,84,85,88,91,92,100,111\}}w_{5}\\
&\quad  +  \gamma_{\{2,12,17,42,45,46,54,55,58,59,68,69,74,78,87,88,92,101,103,111\}}w_{11}\\
&\quad +  \gamma_{\{19,47,54,80,87,104\}}w_{12} +  \gamma_{68}w_{14} +  \gamma_{58}w_{15} \equiv_{\bar\omega} 0,\\
&p_{(1;(3,4,5))}(S) \equiv_{\bar\omega} \gamma_{\{21,34,35,36,73,76,78,80\}}w_{1} +  \gamma_{\{11,38,41,74,85,87,88,104,111\}}w_{7}\\ 
&\quad +  \gamma_{\{9,21,28,37,38,39,40,41,43,45,46,47,50,51,52,60,63,65,83,84,85,88,91,92,95,97,100,111\}}w_{5}\\ 
&\quad +   \gamma_{\{10,16,21,37,39,40,74,83,84,87,101\}}w_{6}+  \gamma_{\{19,21,52,54,91,92,97,99,100,104\}}w_{12} \\ 
&\quad +  \gamma_{\{2,15,16,17,21,29,50,51,54,55,58,59,68,69,91,95,99,101,103\}}w_{11}\\ &\quad+  \gamma_{\{55,91,103\}}w_{13} +  \gamma_{59}w_{15}   \equiv_{\bar\omega} 0,\\
&p_{(2;(3,4,5))}(S) \equiv_{\bar\omega}   \gamma_{\{16,21,29,34,37,38,39,40,41,50,51,52,55,59,60,83,84,85,91,95,97\}}w_{5}\\ 
&\quad + \gamma_{\{1,9,10,11,12,15,17,19,21,28,29\}}w_{1} +   \gamma_{\{2,28,35,37,39,40,42,43,45,46,50,51,54,63,101\}}w_{6}\\ 
&\quad +  \gamma_{\{2,21,29,36,38,41,47,52,54,55,58,59,65,68,104\}}w_{7} + \gamma_{\{21,80,85,87,91,97,99,104\}}w_{12}\\ 
&\quad +  \gamma_{\{16,21,29,69,74,76,78,83,84,87,88,91,92,95,99,100,101,103,111\}}w_{11}\\ 
&\quad +  \gamma_{\{88,91,92,100,103\}}w_{13} +  \gamma_{111}w_{15}  \equiv_{\bar\omega} 0.
\end{align*}
By a direct computation using the above equalities we get 
\begin{equation}\label{c76}
\gamma_j = 0,\ j \in \{58, 59, 68, 111\}
\end{equation}
Now, by applying the homomorphism $p_{(1;(2,3,4,5))}$ to \eqref{ctd612} and using \eqref{c71}-\eqref{c76}, we get
\begin{align*}
&p_{(1;(2,3,4,5))} \equiv_{\bar\omega}	\gamma_{\{21,34,35,36,74,76,78,80\}}w_{1} +  \gamma_{\{29,60,95,97\}}w_{2} +   \gamma_{101}w_{3} +  \gamma_{104}w_{4}\\ 
&\ +  \gamma_{\{9,16,21,28,34,37,38,39,40,41,42,43,45,46,47,50,51,52,60,63,65,74,76,78,80,88,91,92,95,97,100\}}w_{5}\\ 
&\ +  \gamma_{\{10,35,37,39,40,74,76,78,83,84,87,101\}}w_{6} +  \gamma_{\{11,21,36,38,41,80,85,87,88,104\}}w_{7}\\ 
&\ +  \gamma_{\{29,63,69,95,99,100,101,103\}}w_{8} +  \gamma_{\{65,97,99,104\}}w_{9} +  \gamma_{103}w_{10} +  \gamma_{\{55,88,91,103\}}w_{13}\\ 
&\ +  \gamma_{\{2,12,15,16,17,21,29,42,43,45,46,50,51,54,55,69,74,76,78,83,84,87,88,91,92,95,99,101,103\}}w_{11}\\ 
&\ +  \gamma_{\{19,21,47,52,54,80,85,87,91,92,97,99,100,104\}}w_{12} +  \gamma_{100}w_{14} +  \gamma_{92}w_{15} \equiv_{\bar\omega} 0.
\end{align*}
By computing from the above equalities we obtain $\gamma_t =0$ for all $t, \ 1 \leqslant t \leqslant 115$. The proposition is proved.
\end{proof}

\section{Appendix 1}\label{sect5}

In this section, we list all needed admissible monomials in $P_5$ which are used in the proofs of the main results.

\subsection{The admissible monomials of weight vector $(4)|(2)|^{d-2}$ in $P_5$}\label{s51}\

\medskip
For any $d \geqslant 6$,
$B_4^+((4)|(2)|^{d-2}) = \{v_{d,t}: 1 \leqslant t \leqslant 35 \}$, where

\medskip
\centerline{
\begin{tabular}{lll}
$v_{1} = x_1x_2x_3^{2^{d-2}-1}x_4^{2^{d-2}-1}$ &$v_{2} = x_1x_2^{3}x_3^{2^{d-2}-3}x_4^{2^{d-2}-1}$\cr  $v_{3} = x_1x_2^{3}x_3^{2^{d-2}-1}x_4^{2^{d-2}-3}$ &$v_{4} = x_1x_2^{7}x_3^{2^{d-2}-5}x_4^{2^{d-2}-3}$\cr  $v_{5} = x_1x_2^{2^{d-2}-1}x_3x_4^{2^{d-2}-1}$ &$v_{6} = x_1x_2^{2^{d-2}-1}x_3^{3}x_4^{2^{d-2}-3}$\cr  $v_{7} = x_1x_2^{2^{d-2}-1}x_3^{2^{d-2}-1}x_4$ &$v_{8} = x_1^{3}x_2x_3^{2^{d-2}-3}x_4^{2^{d-2}-1}$\cr  $v_{9} = x_1^{3}x_2x_3^{2^{d-2}-1}x_4^{2^{d-2}-3}$ &$v_{10} = x_1^{3}x_2^{3}x_3^{2^{d-2}-3}x_4^{2^{d-2}-3}$\cr  $v_{11} = x_1^{3}x_2^{5}x_3^{2^{d-2}-7}x_4^{2^{d-2}-1}$ &$v_{12} = x_1^{3}x_2^{5}x_3^{2^{d-2}-5}x_4^{2^{d-2}-3}$\cr  $v_{13} = x_1^{3}x_2^{5}x_3^{2^{d-2}-1}x_4^{2^{d-2}-7}$ &$v_{14} = x_1^{3}x_2^{7}x_3^{2^{d-2}-7}x_4^{2^{d-2}-3}$\cr  $v_{15} = x_1^{3}x_2^{7}x_3^{2^{d-2}-3}x_4^{2^{d-2}-7}$ &$v_{16} = x_1^{3}x_2^{2^{d-2}-3}x_3x_4^{2^{d-2}-1}$\cr  $v_{17} = x_1^{3}x_2^{2^{d-2}-3}x_3^{3}x_4^{2^{d-2}-3}$ &$v_{18} = x_1^{3}x_2^{2^{d-2}-3}x_3^{2^{d-2}-1}x_4$\cr  $v_{19} = x_1^{3}x_2^{2^{d-2}-1}x_3x_4^{2^{d-2}-3}$ &$v_{20} = x_1^{3}x_2^{2^{d-2}-1}x_3^{5}x_4^{2^{d-2}-7}$\cr  $v_{21} = x_1^{3}x_2^{2^{d-2}-1}x_3^{2^{d-2}-3}x_4$ &$v_{22} = x_1^{7}x_2x_3^{2^{d-2}-5}x_4^{2^{d-2}-3}$\cr  $v_{23} = x_1^{7}x_2^{3}x_3^{2^{d-2}-7}x_4^{2^{d-2}-3}$ &$v_{24} = x_1^{7}x_2^{3}x_3^{2^{d-2}-3}x_4^{2^{d-2}-7}$\cr  $v_{25} = x_1^{7}x_2^{11}x_3^{2^{d-2}-11}x_4^{2^{d-2}-7}$ &$v_{26} = x_1^{7}x_2^{2^{d-2}-5}x_3x_4^{2^{d-2}-3}$\cr  $v_{27} = x_1^{7}x_2^{2^{d-2}-5}x_3^{5}x_4^{2^{d-2}-7}$ &$v_{28} = x_1^{7}x_2^{2^{d-2}-5}x_3^{2^{d-2}-3}x_4$\cr  $v_{29} = x_1^{2^{d-2}-1}x_2x_3x_4^{2^{d-2}-1}$ &$v_{30} = x_1^{2^{d-2}-1}x_2x_3^{3}x_4^{2^{d-2}-3}$\cr  $v_{31} = x_1^{2^{d-2}-1}x_2x_3^{2^{d-2}-1}x_4$ &$v_{32} = x_1^{2^{d-2}-1}x_2^{3}x_3x_4^{2^{d-2}-3}$\cr  $v_{33} = x_1^{2^{d-2}-1}x_2^{3}x_3^{5}x_4^{2^{d-2}-7}$ &$v_{34} = x_1^{2^{d-2}-1}x_2^{3}x_3^{2^{d-2}-3}x_4$\cr  $v_{35} = x_1^{2^{d-2}-1}x_2^{2^{d-2}-1}x_3x_4$ &\cr  
\end{tabular}}

\bigskip
We have $B_5^0((4)|(2)|^{d-2}) = \left\{\theta_{J_i}(x): x \in B_4^+((4)|(2)|^{d-2}),\, 1 \leqslant i \leqslant 5\right\}.$

For any $d \geqslant 6$, we set  $B_4^+((3)|(1)^{d-2}) = \{\bar v_t= \bar v_{d,t}: 1 \leqslant t \leqslant 17 \}$, where

\medskip
\centerline{\begin{tabular}{lll}
$\bar v_{1} = x_1x_2x_3x_4^{2^{d-1}-2}$ &$\bar v_{2} = x_1x_2x_3^{2^{d-1}-2}x_4$ &$\bar v_{3} = x_1x_2^{2^{d-1}-2}x_3x_4$\cr  $\bar v_{4} = x_1x_2x_3^{2}x_4^{2^{d-1}-3}$ &$\bar v_{5} = x_1x_2^{2}x_3x_4^{2^{d-1}-3}$ &$\bar v_{6} = x_1x_2^{2}x_3^{2^{d-1}-3}x_4$\cr  $\bar v_{7} = x_1x_2x_3^{3}x_4^{2^{d-1}-4}$ &$\bar v_{8} = x_1x_2^{3}x_3x_4^{2^{d-1}-4}$ &$\bar v_{9} = x_1x_2^{3}x_3^{2^{d-1}-4}x_4$\cr  $\bar v_{10} = x_1^{3}x_2x_3x_4^{2^{d-1}-4}$ &$\bar v_{11} = x_1^{3}x_2x_3^{2^{d-1}-4}x_4$ &$\bar v_{12} = x_1x_2^{2}x_3^{5}x_4^{2^{d-1}-7}$\cr  $\bar v_{13} = x_1x_2^{3}x_3^{4}x_4^{2^{d-1}-7}$ &$\bar v_{14} = x_1^{3}x_2x_3^{4}x_4^{2^{d-1}-7}$ &$\bar v_{15} = x_1x_2^{3}x_3^{5}x_4^{2^{d-1}-8}$\cr  $\bar v_{16} = x_1^{3}x_2x_3^{5}x_4^{2^{d-1}-8}$ &$\bar v_{17} = x_1^{3}x_2^{5}x_3x_4^{2^{d-1}-8}$. \cr  
\end{tabular}}

\medskip
We have $B_5^+((4)|(2)|^{d-2}) = A(d)\cup B(d)$, where 
$$A(d) = \left\{x_i^{2^{d-1}-1}\theta_{J_i}(x): x \in B_4^+((3)|(1)^{d-2}),\, 1 \leqslant i \leqslant 5\right\},$$
and $B(d) = \left\{b_{t} = b_{d,t}: 1 \leqslant t \leqslant 205 \right\}$. Here, the monomials $b_{t} = b_{d,t}$ are determined as follows:

\medskip
\centerline{\begin{tabular}{lll}
$b_{1} = x_1x_2x_3^{3}x_4^{2^{d-1}-3}x_5^{2^{d-1}-2}$ & &$b_{2} = x_1x_2x_3^{3}x_4^{2^{d-1}-2}x_5^{2^{d-1}-3}$\cr  $b_{3} = x_1x_2x_3^{6}x_4^{2^{d-1}-5}x_5^{2^{d-1}-3}$ & &$b_{4} = x_1x_2x_3^{7}x_4^{2^{d-1}-6}x_5^{2^{d-1}-3}$\cr  $b_{5} = x_1x_2x_3^{7}x_4^{2^{d-1}-5}x_5^{2^{d-1}-4}$ & &$b_{6} = x_1x_2x_3^{2^{d-1}-2}x_4^{3}x_5^{2^{d-1}-3}$\cr  $b_{7} = x_1x_2^{2}x_3^{3}x_4^{2^{d-1}-3}x_5^{2^{d-1}-3}$ & &$b_{8} = x_1x_2^{2}x_3^{5}x_4^{2^{d-1}-5}x_5^{2^{d-1}-3}$\cr  $b_{9} = x_1x_2^{2}x_3^{7}x_4^{2^{d-1}-7}x_5^{2^{d-1}-3}$ & &$b_{10} = x_1x_2^{2}x_3^{7}x_4^{2^{d-1}-3}x_5^{2^{d-1}-7}$\cr  $b_{11} = x_1x_2^{2}x_3^{2^{d-1}-3}x_4^{3}x_5^{2^{d-1}-3}$ & &$b_{12} = x_1x_2^{3}x_3x_4^{2^{d-1}-3}x_5^{2^{d-1}-2}$\cr 
\end{tabular}}
\centerline{\begin{tabular}{lll} 
$b_{13} = x_1x_2^{3}x_3x_4^{2^{d-1}-2}x_5^{2^{d-1}-3}$ & &$b_{14} = x_1x_2^{3}x_3^{2}x_4^{2^{d-1}-3}x_5^{2^{d-1}-3}$\cr  $b_{15} = x_1x_2^{3}x_3^{3}x_4^{2^{d-1}-4}x_5^{2^{d-1}-3}$ & &$b_{16} = x_1x_2^{3}x_3^{3}x_4^{2^{d-1}-3}x_5^{2^{d-1}-4}$\cr  $b_{17} = x_1x_2^{3}x_3^{4}x_4^{2^{d-1}-5}x_5^{2^{d-1}-3}$ & &$b_{18} = x_1x_2^{3}x_3^{5}x_4^{2^{d-1}-7}x_5^{2^{d-1}-2}$\cr  $b_{19} = x_1x_2^{3}x_3^{5}x_4^{2^{d-1}-6}x_5^{2^{d-1}-3}$ & &$b_{20} = x_1x_2^{3}x_3^{5}x_4^{2^{d-1}-5}x_5^{2^{d-1}-4}$\cr  $b_{21} = x_1x_2^{3}x_3^{5}x_4^{2^{d-1}-2}x_5^{2^{d-1}-7}$ & &$b_{22} = x_1x_2^{3}x_3^{6}x_4^{2^{d-1}-7}x_5^{2^{d-1}-3}$\cr  $b_{23} = x_1x_2^{3}x_3^{6}x_4^{2^{d-1}-3}x_5^{2^{d-1}-7}$ & &$b_{24} = x_1x_2^{3}x_3^{7}x_4^{2^{d-1}-8}x_5^{2^{d-1}-3}$\cr  $b_{25} = x_1x_2^{3}x_3^{7}x_4^{2^{d-1}-7}x_5^{2^{d-1}-4}$ & &$b_{26} = x_1x_2^{3}x_3^{7}x_4^{2^{d-1}-4}x_5^{2^{d-1}-7}$\cr  $b_{27} = x_1x_2^{3}x_3^{7}x_4^{2^{d-1}-3}x_5^{2^{d-1}-8}$ & &$b_{28} = x_1x_2^{3}x_3^{2^{d-1}-4}x_4^{3}x_5^{2^{d-1}-3}$\cr
\end{tabular}}
\centerline{\begin{tabular}{lll} $b_{93} = x_1^{3}x_2^{3}x_3^{5}x_4^{2^{d-1}-8}x_5^{2^{d-1}-3}$ & &$b_{94} = x_1^{3}x_2^{3}x_3^{5}x_4^{2^{d-1}-7}x_5^{2^{d-1}-4}$\cr  $b_{95} = x_1^{3}x_2^{3}x_3^{5}x_4^{2^{d-1}-4}x_5^{2^{d-1}-7}$ & &$b_{96} = x_1^{3}x_2^{3}x_3^{5}x_4^{2^{d-1}-3}x_5^{2^{d-1}-8}$\cr  $b_{97} = x_1^{3}x_2^{3}x_3^{12}x_4^{2^{d-1}-11}x_5^{2^{d-1}-7}$ & &$b_{98} = x_1^{3}x_2^{3}x_3^{13}x_4^{2^{d-1}-12}x_5^{2^{d-1}-7}$\cr  $b_{99} = x_1^{3}x_2^{3}x_3^{13}x_4^{2^{d-1}-11}x_5^{2^{d-1}-8}$ & &$b_{100} = x_1^{3}x_2^{3}x_3^{2^{d-1}-4}x_4x_5^{2^{d-1}-3}$\cr  $b_{101} = x_1^{3}x_2^{3}x_3^{2^{d-1}-4}x_4^{5}x_5^{2^{d-1}-7}$ & &$b_{102} = x_1^{3}x_2^{3}x_3^{2^{d-1}-4}x_4^{2^{d-1}-3}x_5$\cr  $b_{103} = x_1^{3}x_2^{3}x_3^{2^{d-1}-3}x_4x_5^{2^{d-1}-4}$ & &$b_{104} = x_1^{3}x_2^{3}x_3^{2^{d-1}-3}x_4^{4}x_5^{2^{d-1}-7}$\cr  $b_{105} = x_1^{3}x_2^{3}x_3^{2^{d-1}-3}x_4^{5}x_5^{2^{d-1}-8}$ & &$b_{106} = x_1^{3}x_2^{3}x_3^{2^{d-1}-3}x_4^{2^{d-1}-4}x_5$\cr  $b_{107} = x_1^{3}x_2^{4}x_3x_4^{2^{d-1}-5}x_5^{2^{d-1}-3}$ & &$b_{108} = x_1^{3}x_2^{4}x_3^{3}x_4^{2^{d-1}-7}x_5^{2^{d-1}-3}$\cr
\end{tabular}}
\centerline{\begin{tabular}{lll} 
$b_{173} = x_1^{7}x_2x_3^{11}x_4^{2^{d-1}-12}x_5^{2^{d-1}-7}$ & &$b_{174} = x_1^{7}x_2x_3^{11}x_4^{2^{d-1}-11}x_5^{2^{d-1}-8}$\cr  $b_{175} = x_1^{7}x_2x_3^{2^{d-1}-6}x_4x_5^{2^{d-1}-3}$ & &$b_{176} = x_1^{7}x_2x_3^{2^{d-1}-6}x_4^{5}x_5^{2^{d-1}-7}$\cr  $b_{177} = x_1^{7}x_2x_3^{2^{d-1}-6}x_4^{2^{d-1}-3}x_5$ & &$b_{178} = x_1^{7}x_2x_3^{2^{d-1}-5}x_4x_5^{2^{d-1}-4}$\cr  $b_{179} = x_1^{7}x_2x_3^{2^{d-1}-5}x_4^{4}x_5^{2^{d-1}-7}$ & &$b_{180} = x_1^{7}x_2x_3^{2^{d-1}-5}x_4^{5}x_5^{2^{d-1}-8}$\cr  $b_{181} = x_1^{7}x_2x_3^{2^{d-1}-5}x_4^{2^{d-1}-4}x_5$ & &$b_{182} = x_1^{7}x_2^{3}x_3x_4^{2^{d-1}-8}x_5^{2^{d-1}-3}$\cr  $b_{183} = x_1^{7}x_2^{3}x_3x_4^{2^{d-1}-7}x_5^{2^{d-1}-4}$ & &$b_{184} = x_1^{7}x_2^{3}x_3x_4^{2^{d-1}-4}x_5^{2^{d-1}-7}$\cr  $b_{185} = x_1^{7}x_2^{3}x_3x_4^{2^{d-1}-3}x_5^{2^{d-1}-8}$ & &$b_{186} = x_1^{7}x_2^{3}x_3^{5}x_4^{2^{d-1}-8}x_5^{2^{d-1}-7}$\cr  $b_{187} = x_1^{7}x_2^{3}x_3^{5}x_4^{2^{d-1}-7}x_5^{2^{d-1}-8}$ & &$b_{188} = x_1^{7}x_2^{3}x_3^{8}x_4^{2^{d-1}-11}x_5^{2^{d-1}-7}$\cr
\end{tabular}}

\medskip 
We have $|B_5^0((4)|(2)|^{d-2})| = 175,\, |A(d)| = 85$. Hence, we obtain
$$|B_5((4)|(2)|^{d-2})| = 175 + 85 +205 = 465.$$

\subsection{The admissible monomials of weight vector $(4)|^2|(3)|^{d-4}$ in $P_5$}\label{s52}\

\medskip
For any $d \geqslant 8$, $B_4((4)|^2|(3)|^{d-4}) = \{w_t= w_{d,t} : 1\leqslant t \leqslant 15\}$, where

\medskip
\centerline{\begin{tabular}{lll}
$w_{1} = x_1^{3}x_2^{2^{d-2}-1}x_3^{2^{d-2}-1}x_4^{2^{d-2}-1}$ &$w_{2} = x_1^{7}x_2^{2^{d-2}-5}x_3^{2^{d-2}-1}x_4^{2^{d-2}-1}$\cr  $w_{3} = x_1^{7}x_2^{2^{d-2}-1}x_3^{2^{d-2}-5}x_4^{2^{d-2}-1}$ &$w_{4} = x_1^{7}x_2^{2^{d-2}-1}x_3^{2^{d-2}-1}x_4^{2^{d-2}-5}$\cr  $w_{5} = x_1^{15}x_2^{2^{d-2}-9}x_3^{2^{d-2}-5}x_4^{2^{d-2}-1}$ &$w_{6} = x_1^{15}x_2^{2^{d-2}-9}x_3^{2^{d-2}-1}x_4^{2^{d-2}-5}$\cr  $w_{7} = x_1^{15}x_2^{2^{d-2}-1}x_3^{2^{d-2}-9}x_4^{2^{d-2}-5}$ &$w_{8} = x_1^{31}x_2^{2^{d-2}-17}x_3^{2^{d-2}-9}x_4^{2^{d-2}-5}$\cr  $w_{9} = x_1^{2^{d-2}-1}x_2^{3}x_3^{2^{d-2}-1}x_4^{2^{d-2}-1}$ &$w_{10} = x_1^{2^{d-2}-1}x_2^{7}x_3^{2^{d-2}-5}x_4^{2^{d-2}-1}$\cr  $w_{11} = x_1^{2^{d-2}-1}x_2^{7}x_3^{2^{d-2}-1}x_4^{2^{d-2}-5}$ &$w_{12} = x_1^{2^{d-2}-1}x_2^{15}x_3^{2^{d-2}-9}x_4^{2^{d-2}-5}$\cr  $w_{13} = x_1^{2^{d-2}-1}x_2^{2^{d-2}-1}x_3^{3}x_4^{2^{d-2}-1}$ &$w_{14} = x_1^{2^{d-2}-1}x_2^{2^{d-2}-1}x_3^{7}x_4^{2^{d-2}-5}$\cr  $w_{15} = x_1^{2^{d-2}-1}x_2^{2^{d-2}-1}x_3^{2^{d-2}-1}x_4^{3}$. &\cr  
\end{tabular}}

\medskip
We have
$$B_5^0((4)|^2|(3)|^{d-4}) = \{\theta_{J+i}(x): x\in B_4((4)|^2|(3)|^{d-4}), 1\leqslant i \leqslant 5\}.$$ 
Then, we have $|B_5^0((4)|^2|(3)|^{d-4})| = 75$.

\medskip
$B_4^+((3)|^2|(2)|^{d-4}) =\{\bar w_{t}=\bar w_{d,t}: 1 \leqslant t \leqslant 77\}$, where

\medskip
\centerline{\begin{tabular}{lll}
$\tilde w_{1} = x_1x_2^{2}x_3^{2^{d-2}-1}x_4^{2^{d-2}-1}$ &$\tilde w_{2} = x_1x_2^{3}x_3^{2^{d-2}-2}x_4^{2^{d-2}-1}$\cr  $\tilde w_{3} = x_1x_2^{3}x_3^{2^{d-2}-1}x_4^{2^{d-2}-2}$ &$\tilde w_{4} = x_1x_2^{6}x_3^{2^{d-2}-5}x_4^{2^{d-2}-1}$\cr  $\tilde w_{5} = x_1x_2^{6}x_3^{2^{d-2}-1}x_4^{2^{d-2}-5}$ &$\tilde w_{6} = x_1x_2^{7}x_3^{2^{d-2}-6}x_4^{2^{d-2}-1}$\cr  $\tilde w_{7} = x_1x_2^{7}x_3^{2^{d-2}-5}x_4^{2^{d-2}-2}$ &$\tilde w_{8} = x_1x_2^{7}x_3^{2^{d-2}-2}x_4^{2^{d-2}-5}$\cr  $\tilde w_{9} = x_1x_2^{7}x_3^{2^{d-2}-1}x_4^{2^{d-2}-6}$ &$\tilde w_{10} = x_1x_2^{14}x_3^{2^{d-2}-9}x_4^{2^{d-2}-5}$\cr  $\tilde w_{11} = x_1x_2^{15}x_3^{2^{d-2}-10}x_4^{2^{d-2}-5}$ &$\tilde w_{12} = x_1x_2^{15}x_3^{2^{d-2}-9}x_4^{2^{d-2}-6}$\cr 
\end{tabular}}
\centerline{\begin{tabular}{lll} 
$\tilde w_{13} = x_1x_2^{2^{d-2}-2}x_3^{3}x_4^{2^{d-2}-1}$ &$\tilde w_{14} = x_1x_2^{2^{d-2}-2}x_3^{7}x_4^{2^{d-2}-5}$\cr  $\tilde w_{15} = x_1x_2^{2^{d-2}-2}x_3^{2^{d-2}-1}x_4^{3}$ &$\tilde w_{16} = x_1x_2^{2^{d-2}-1}x_3^{2}x_4^{2^{d-2}-1}$\cr  $\tilde w_{17} = x_1x_2^{2^{d-2}-1}x_3^{3}x_4^{2^{d-2}-2}$ &$\tilde w_{18} = x_1x_2^{2^{d-2}-1}x_3^{6}x_4^{2^{d-2}-5}$\cr  $\tilde w_{19} = x_1x_2^{2^{d-2}-1}x_3^{7}x_4^{2^{d-2}-6}$ &$\tilde w_{20} = x_1x_2^{2^{d-2}-1}x_3^{2^{d-2}-2}x_4^{3}$\cr  $\tilde w_{21} = x_1x_2^{2^{d-2}-1}x_3^{2^{d-2}-1}x_4^{2}$ &$\tilde w_{22} = x_1^{3}x_2x_3^{2^{d-2}-2}x_4^{2^{d-2}-1}$\cr  $\tilde w_{23} = x_1^{3}x_2x_3^{2^{d-2}-1}x_4^{2^{d-2}-2}$ &$\tilde w_{24} = x_1^{3}x_2^{3}x_3^{2^{d-2}-4}x_4^{2^{d-2}-1}$\cr  $\tilde w_{25} = x_1^{3}x_2^{3}x_3^{2^{d-2}-3}x_4^{2^{d-2}-2}$ &$\tilde w_{26} = x_1^{3}x_2^{3}x_3^{2^{d-2}-1}x_4^{2^{d-2}-4}$\cr  $\tilde w_{27} = x_1^{3}x_2^{5}x_3^{2^{d-2}-6}x_4^{2^{d-2}-1}$ &$\tilde w_{28} = x_1^{3}x_2^{5}x_3^{2^{d-2}-5}x_4^{2^{d-2}-2}$\cr
\end{tabular}}

\medskip 
For any $d \geqslant 8$, we set 
$$\bar A(d) = \{x_i^{2^{d-2}-1}\theta_{J_i}(w): w \in B_4^+((3)|^2|(2)|^{d-4}), 1\leqslant i \leqslant 5\}.$$
We have $|\bar A(d)| = 275$. 

Then, $B_5^+((4)|^2|(3)|^{d-4}) = \bar A(d)\cup \bar B(d)$, where $\bar B(d) = \{\bar b_{d,t} : 1 \leqslant t \leqslant 115\}$ with the monomials $\bar b_t= \bar b_{d,t},\, 1 \leqslant t \leqslant 115$, determined as follows: 

\medskip
\centerline{\begin{tabular}{lll}
$\bar b_{1} = x_1x_2^{15}x_3^{2^{d-2}-9}x_4^{2^{d-2}-5}x_5^{2^{d-2}-2}$ & &$\bar b_{2} = x_1x_2^{15}x_3^{2^{d-2}-9}x_4^{2^{d-2}-2}x_5^{2^{d-2}-5}$\cr
$\bar b_{3} = x_1x_2^{15}x_3^{2^{d-2}-2}x_4^{2^{d-2}-9}x_5^{2^{d-2}-5}$ & &$\bar b_{4} = x_1x_2^{30}x_3^{2^{d-2}-17}x_4^{2^{d-2}-9}x_5^{2^{d-2}-5}$\cr  \end{tabular}}
\centerline{\begin{tabular}{lll}
$\bar b_{5} = x_1x_2^{31}x_3^{2^{d-2}-18}x_4^{2^{d-2}-9}x_5^{2^{d-2}-5}$ & &$\bar b_{6} = x_1x_2^{31}x_3^{2^{d-2}-17}x_4^{2^{d-2}-10}x_5^{2^{d-2}-5}$\cr  
$\bar b_{7} = x_1x_2^{31}x_3^{2^{d-2}-17}x_4^{2^{d-2}-9}x_5^{2^{d-2}-6}$ & &$\bar b_{8} = x_1x_2^{2^{d-2}-2}x_3^{15}x_4^{2^{d-2}-9}x_5^{2^{d-2}-5}$\cr  
$\bar b_{9} = x_1^{3}x_2^{7}x_3^{2^{d-2}-5}x_4^{2^{d-2}-3}x_5^{2^{d-2}-2}$ & &$\bar b_{10} = x_1^{3}x_2^{7}x_3^{2^{d-2}-3}x_4^{2^{d-2}-5}x_5^{2^{d-2}-2}$\cr  
$\bar b_{11} = x_1^{3}x_2^{7}x_3^{2^{d-2}-3}x_4^{2^{d-2}-2}x_5^{2^{d-2}-5}$ & &$\bar b_{12} = x_1^{3}x_2^{13}x_3^{2^{d-2}-9}x_4^{2^{d-2}-5}x_5^{2^{d-2}-2}$\cr  
$\bar b_{13} = x_1^{3}x_2^{13}x_3^{2^{d-2}-9}x_4^{2^{d-2}-2}x_5^{2^{d-2}-5}$ & &$\bar b_{14} = x_1^{3}x_2^{13}x_3^{2^{d-2}-2}x_4^{2^{d-2}-9}x_5^{2^{d-2}-5}$\cr  
$\bar b_{15} = x_1^{3}x_2^{15}x_3^{2^{d-2}-11}x_4^{2^{d-2}-5}x_5^{2^{d-2}-2}$ & &$\bar b_{16} = x_1^{3}x_2^{15}x_3^{2^{d-2}-11}x_4^{2^{d-2}-2}x_5^{2^{d-2}-5}$\cr  
$\bar b_{17} = x_1^{3}x_2^{15}x_3^{2^{d-2}-9}x_4^{2^{d-2}-7}x_5^{2^{d-2}-2}$ & &$\bar b_{18} = x_1^{3}x_2^{15}x_3^{2^{d-2}-9}x_4^{2^{d-2}-5}x_5^{2^{d-2}-4}$\cr  
$\bar b_{19} = x_1^{3}x_2^{15}x_3^{2^{d-2}-9}x_4^{2^{d-2}-3}x_5^{2^{d-2}-6}$ & &$\bar b_{20} = x_1^{3}x_2^{15}x_3^{2^{d-2}-3}x_4^{2^{d-2}-10}x_5^{2^{d-2}-5}$\cr
\end{tabular}}
\centerline{\begin{tabular}{lll}
$\bar b_{85} = x_1^{15}x_2^{7}x_3^{2^{d-2}-13}x_4^{2^{d-2}-3}x_5^{2^{d-2}-6}$ & &$\bar b_{86} = x_1^{15}x_2^{7}x_3^{2^{d-2}-9}x_4^{2^{d-2}-8}x_5^{2^{d-2}-5}$\cr  $\bar b_{87} = x_1^{15}x_2^{7}x_3^{2^{d-2}-9}x_4^{2^{d-2}-7}x_5^{2^{d-2}-6}$ & &$\bar b_{88} = x_1^{15}x_2^{7}x_3^{2^{d-2}-5}x_4^{2^{d-2}-11}x_5^{2^{d-2}-6}$\cr  $\bar b_{89} = x_1^{15}x_2^{15}x_3^{2^{d-2}-15}x_4^{2^{d-2}-10}x_5^{2^{d-2}-5}$ & &$\bar b_{90} = x_1^{15}x_2^{15}x_3^{2^{d-2}-15}x_4^{2^{d-2}-9}x_5^{2^{d-2}-6}$\cr  $\bar b_{91} = x_1^{15}x_2^{15}x_3^{2^{d-2}-13}x_4^{2^{d-2}-11}x_5^{2^{d-2}-6}$ & &$\bar b_{92} = x_1^{15}x_2^{23}x_3^{2^{d-2}-21}x_4^{2^{d-2}-11}x_5^{2^{d-2}-6}$\cr  $\bar b_{93} = x_1^{15}x_2^{2^{d-2}-9}x_3x_4^{2^{d-2}-5}x_5^{2^{d-2}-2}$ & &$\bar b_{94} = x_1^{15}x_2^{2^{d-2}-9}x_3x_4^{2^{d-2}-2}x_5^{2^{d-2}-5}$\cr  $\bar b_{95} = x_1^{15}x_2^{2^{d-2}-9}x_3^{3}x_4^{2^{d-2}-7}x_5^{2^{d-2}-2}$ & &$\bar b_{96} = x_1^{15}x_2^{2^{d-2}-9}x_3^{3}x_4^{2^{d-2}-5}x_5^{2^{d-2}-4}$\cr  $\bar b_{97} = x_1^{15}x_2^{2^{d-2}-9}x_3^{3}x_4^{2^{d-2}-3}x_5^{2^{d-2}-6}$ & &$\bar b_{98} = x_1^{15}x_2^{2^{d-2}-9}x_3^{7}x_4^{2^{d-2}-8}x_5^{2^{d-2}-5} $\cr  
$\bar b_{99} = x_1^{15}x_2^{2^{d-2}-9}x_3^{7}x_4^{2^{d-2}-7}x_5^{2^{d-2}-6}$ & &$\bar b_{100} = x_1^{15}x_2^{2^{d-2}-9}x_3^{11}x_4^{2^{d-2}-11}x_5^{2^{d-2}-6}$\cr
\end{tabular}}

\medskip
Thus, we obtain $|B_5((4)|^2|(3)|^{d-4})| = 75 + 275 + 115 = 465$.

\subsection{The admissible monomials of weight vector $(4)|^2|(3)|^{d-4}|(1)$ in $P_5$}\label{s53}\

\medskip
We observe that for any $x \in |B_5((4)|^2|(3)|^{d-4})$ there exists unique $J_x = (j_1^x,j_2^x,j_3^x)$ such that $1 \leqslant j_1^x < j_2^x < j_3^x \leqslant 5$ and $\nu_{j_t^x}(x) > 32$ for $t = 1,\, 2,\, 3$. Then we have
$$B_5((4)|^2|(3)|^{d-4}|(1)) = \left\{x\theta_{J_x}\big(x_s^{2^{d-2}}\big): x \in B_5((4)|^2|(3)|^{d-4}),\, s = 1,\, 2,\, 3 \right\}.$$

Thus, we obtain $|B_5((4)|^2|(3)|^{d-4}|(1))| = 1395$ and 
$$|B_5(2^d)| = 124 + 465 + 1395 = 1984,$$ 
for any $d \geqslant 8$.

\section{Appendix 2}\label{sect6}
\setcounter{equation}{1}

In this section we present some computations of $(QP_5)_{2^d}$ for $5\leqslant d \leqslant 7$.

From Lemma \ref{bdm}, it implies that if $\omega$ is a weight vector of degree $n$ and $x \in P_k(\omega)$, then $p_{(i;I)}(x) \in P_{k-1}(\omega)$. Moreover, $p_{(i;I)}$ passes to a homomorphism  
\begin{align*}
	p_{(i;I)}^{(\omega)} :QP_k(\omega)\longrightarrow QP_{k-1}(\omega).
\end{align*} 
We set 
\begin{align*}\widetilde {SF}_k(\omega) &= \bigcap_{(i;I)\in \mathcal N_k}  \mbox{Ker}(p_{(i;I)}^{(\omega)})\subset \widetilde{QP}_k^+(\omega)\\ \widetilde{QP}_k(\omega) &= QP_k(\omega)/\widetilde {SF}_k(\omega),\ \widetilde{QP}_k^+(\omega) = QP_k^+(\omega)/\widetilde {SF}_k(\omega).
\end{align*}

We see that $\widetilde {SF}_k(\omega)$ is a subspace of the spike-free module $SF^{\omega}_k$ (see Walker and Wood \cite[Chap. 30]{wa2}). Then we have
$$QP_k(\omega) \cong \widetilde{QP}_k(\omega) \bigoplus \widetilde {SF}_k(\omega).$$

\subsection{A note on the structure of the space $(QP_5)_{32}$}\

\medskip
By Lemma \ref{bdd5}, we have
$$ (QP_5)_{32} \cong QP_5((2)|(1)|^4)\bigoplus QP_5((4)|(2)|^3) \bigoplus QP_5((4)|^2(3)|(1)).$$

It is well-known from Ph\'uc \cite{ph21} that 
$$\dim QP_5((2)|(1)|^4)=124,\ \dim QP_5((4)|^2|(3)|(1)) = 395$$ and $QP_5^+((4)|(2)|^3)$ is generated by 310 classes represented by the monomials $Z'_j,\, 1 \leqslant j \leqslant 310$ as listed in Ph\'uc \cite[Pages  36-38]{ph21}. Consider the homomorphism $p_{(i;I)}^{(\omega)} :QP_5(\omega)\to QP_{4}(\omega)$ with $\omega = (4)|(2)|^3$ and $(i;I) \in \mathcal N_5$. By a direct computation using the monomials $Z'_j,\, 1 \leqslant j \leqslant 310$, we obtain  
$\widetilde {SF}_5(\omega) = \bigcap_{(i;I)\in \mathcal N_5}\mbox{Ker}\big(p_{(i;I)}^{(\omega)}\big) =\langle\{[\theta_t]_\omega: 1 \leqslant t \leqslant 20\}\rangle,$
where the polynomials $\theta_t$ are determined as follows:
\begin{align*}
\theta_1 &= x_1x_2x_3^{3}x_4^{13}x_5^{14} + x_1x_2x_3^{3}x_4^{14}x_5^{13} + x_1x_2^{6}x_3^{7}x_4^{9}x_5^{9}\\ &\quad + x_1^{3}x_2x_3^{7}x_4^{9}x_5^{12} + x_1^{3}x_2x_3^{7}x_4^{12}x_5^{9} + x_1^{3}x_2^{4}x_3^{7}x_4^{9}x_5^{9},\\
\theta_2 &=  x_1x_2^{3}x_3^{5}x_4^{9}x_5^{14} + x_1x_2^{3}x_3^{5}x_4^{14}x_5^{9} + x_1x_2^{3}x_3^{14}x_4^{5}x_5^{9}\\ &\quad+ x_1x_2^{14}x_3^{3}x_4^{5}x_5^{9} + x_1^{3}x_2^{3}x_3^{12}x_4^{5}x_5^{9} + x_1^{3}x_2^{5}x_3^{3}x_4^{12}x_5^{9}\\ &\quad + x_1^{3}x_2^{5}x_3^{9}x_4^{3}x_5^{12} + x_1^{3}x_2^{12}x_3^{3}x_4^{5}x_5^{9} + x_1^{3}x_2^{5}x_3^{9}x_4^{6}x_5^{9},\\
\theta_3 &=  x_1x_2^{3}x_3x_4^{13}x_5^{14} + x_1x_2^{3}x_3x_4^{14}x_5^{13} + x_1x_2^{7}x_3^{6}x_4^{9}x_5^{9}\\ &\quad + x_1^{3}x_2^{7}x_3x_4^{9}x_5^{12} + x_1^{3}x_2^{7}x_3x_4^{12}x_5^{9} + x_1^{3}x_2^{7}x_3^{4}x_4^{9}x_5^{9},\\
\theta_4 &= x_1x_2^{3}x_3^{3}x_4^{12}x_5^{13} + x_1x_2^{3}x_3^{13}x_4x_5^{14} + x_1x_2^{3}x_3^{13}x_4^{2}x_5^{13} + x_1x_2^{3}x_3^{14}x_4x_5^{13}\\ &\quad + x_1x_2^{7}x_3^{7}x_4^{8}x_5^{9} + x_1^{3}x_2^{7}x_3^{9}x_4x_5^{12} + x_1^{3}x_2^{7}x_3^{9}x_4^{4}x_5^{9} + x_1^{3}x_2^{7}x_3^{12}x_4x_5^{9},\\
\theta_5 &= x_1x_2^{3}x_3^{13}x_4^{3}x_5^{12} + x_1x_2^{3}x_3^{13}x_4^{14}x_5 + x_1x_2^{3}x_3^{14}x_4^{13}x_5 + x_1x_2^{7}x_3^{7}x_4^{9}x_5^{8}\\ &\quad + x_1^{3}x_2^{7}x_3^{5}x_4^{9}x_5^{8} + x_1^{3}x_2^{7}x_3^{9}x_4^{5}x_5^{8} + x_1^{3}x_2^{7}x_3^{9}x_4^{12}x_5 + x_1^{3}x_2^{7}x_3^{12}x_4^{9}x_5,\\
\theta_6 &=  x_1^{3}x_2x_3x_4^{13}x_5^{14} + x_1^{3}x_2x_3x_4^{14}x_5^{13} + x_1^{7}x_2x_3^{6}x_4^{9}x_5^{9}\\ &\quad + x_1^{7}x_2^{3}x_3x_4^{9}x_5^{12} + x_1^{7}x_2^{3}x_3x_4^{12}x_5^{9} + x_1^{7}x_2^{3}x_3^{4}x_4^{9}x_5^{9}, \\
\theta_7 &= x_1^{3}x_2x_3^{3}x_4^{12}x_5^{13} + x_1^{3}x_2x_3^{13}x_4x_5^{14} + x_1^{3}x_2x_3^{13}x_4^{2}x_5^{13} + x_1^{3}x_2x_3^{14}x_4x_5^{13}\\ &\quad + x_1^{7}x_2x_3^{7}x_4^{8}x_5^{9} + x_1^{7}x_2^{3}x_3^{9}x_4x_5^{12} + x_1^{7}x_2^{3}x_3^{9}x_4^{4}x_5^{9} + x_1^{7}x_2^{3}x_3^{12}x_4x_5^{9},\\
\theta_8 &=  x_1^{3}x_2x_3^{13}x_4^{3}x_5^{12} + x_1^{3}x_2x_3^{13}x_4^{14}x_5 + x_1^{3}x_2x_3^{14}x_4^{13}x_5 + x_1^{7}x_2x_3^{7}x_4^{9}x_5^{8}\\ &\quad + x_1^{7}x_2^{3}x_3^{5}x_4^{9}x_5^{8} + x_1^{7}x_2^{3}x_3^{9}x_4^{5}x_5^{8} + x_1^{7}x_2^{3}x_3^{9}x_4^{12}x_5 + x_1^{7}x_2^{3}x_3^{12}x_4^{9}x_5,\\
\theta_9 &= x_1^{3}x_2^{3}x_3x_4^{12}x_5^{13} + x_1^{3}x_2^{3}x_3^{12}x_4x_5^{13} + x_1^{7}x_2^{7}x_3x_4^{8}x_5^{9} + x_1^{7}x_2^{7}x_3^{8}x_4x_5^{9},\\
\theta_{10} &= x_1^{3}x_2^{3}x_3x_4^{13}x_5^{12} + x_1^{3}x_2^{3}x_3^{12}x_4^{13}x_5 + x_1^{7}x_2^{7}x_3x_4^{9}x_5^{8} + x_1^{7}x_2^{7}x_3^{8}x_4^{9}x_5,\\
\theta_{11} &=  x_1^{3}x_2^{3}x_3^{13}x_4x_5^{12} + x_1^{3}x_2^{3}x_3^{13}x_4^{12}x_5 + x_1^{7}x_2^{7}x_3^{9}x_4x_5^{8} + x_1^{7}x_2^{7}x_3^{9}x_4^{8}x_5,\\
\theta_{12} &=  x_1^{3}x_2x_3^{5}x_4^{9}x_5^{14} + x_1^{3}x_2x_3^{5}x_4^{14}x_5^{9} + x_1^{3}x_2x_3^{14}x_4^{5}x_5^{9} + x_1^{3}x_2^{3}x_3^{5}x_4^{9}x_5^{12}\\ &\quad + x_1^{3}x_2^{3}x_3^{5}x_4^{12}x_5^{9} + x_1^{3}x_2^{5}x_3^{3}x_4^{9}x_5^{12} + x_1^{3}x_2^{5}x_3^{3}x_4^{12}x_5^{9} + x_1^{3}x_2^{5}x_3^{6}x_4^{9}x_5^{9}\\ &\quad + x_1^{3}x_2^{12}x_3^{3}x_4^{5}x_5^{9} + x_1^{7}x_2x_3^{3}x_4^{12}x_5^{9} + x_1^{7}x_2x_3^{10}x_4^{5}x_5^{9} + x_1^{7}x_2^{8}x_3^{3}x_4^{5}x_5^{9},\\
\theta_{13} &= x_1^{3}x_2^{3}x_3^{12}x_4^{5}x_5^{9} + x_1^{3}x_2^{5}x_3x_4^{9}x_5^{14} + x_1^{3}x_2^{5}x_3x_4^{14}x_5^{9} + x_1^{3}x_2^{5}x_3^{6}x_4^{9}x_5^{9}\\ &\quad + x_1^{3}x_2^{7}x_3x_4^{9}x_5^{12} + x_1^{3}x_2^{7}x_3x_4^{12}x_5^{9} + x_1^{3}x_2^{13}x_3^{2}x_4^{5}x_5^{9} + x_1^{7}x_2^{3}x_3x_4^{9}x_5^{12}\\ &\quad + x_1^{7}x_2^{3}x_3x_4^{12}x_5^{9} + x_1^{7}x_2^{3}x_3^{8}x_4^{5}x_5^{9} + x_1^{7}x_2^{7}x_3x_4^{8}x_5^{9} + x_1^{7}x_2^{7}x_3x_4^{9}x_5^{8}\\ &\quad + x_1^{7}x_2^{9}x_3^{2}x_4^{5}x_5^{9},\\
\theta_{14} &= x_1x_2^{3}x_3^{3}x_4^{12}x_5^{13} + x_1x_2^{7}x_3^{7}x_4^{8}x_5^{9} + x_1^{3}x_2x_3^{3}x_4^{12}x_5^{13} + x_1^{3}x_2^{3}x_3x_4^{12}x_5^{13}\\ &\quad + x_1^{3}x_2^{3}x_3^{5}x_4^{8}x_5^{13} + x_1^{3}x_2^{3}x_3^{5}x_4^{12}x_5^{9} + x_1^{3}x_2^{3}x_3^{13}x_4x_5^{12} + x_1^{3}x_2^{3}x_3^{13}x_4^{4}x_5^{9}\\ &\quad + x_1^{3}x_2^{5}x_3^{3}x_4^{8}x_5^{13} + x_1^{3}x_2^{5}x_3^{7}x_4^{8}x_5^{9} + x_1^{3}x_2^{5}x_3^{9}x_4^{2}x_5^{13} + x_1^{3}x_2^{5}x_3^{9}x_4^{6}x_5^{9}\\ &\quad + x_1^{3}x_2^{7}x_3^{5}x_4^{8}x_5^{9} + x_1^{3}x_2^{13}x_3^{3}x_4^{4}x_5^{9} + x_1^{7}x_2x_3^{7}x_4^{8}x_5^{9} + x_1^{7}x_2^{3}x_3^{5}x_4^{8}x_5^{9}\\ &\quad + x_1^{7}x_2^{3}x_3^{9}x_4^{4}x_5^{9} + x_1^{7}x_2^{7}x_3x_4^{8}x_5^{9} + x_1^{7}x_2^{7}x_3^{9}x_4x_5^{8} + x_1^{7}x_2^{9}x_3^{3}x_4^{4}x_5^{9},\\
\theta_{15} &= x_1x_2^{3}x_3^{3}x_4^{13}x_5^{12} + x_1x_2^{7}x_3^{7}x_4^{9}x_5^{8} + x_1^{3}x_2x_3^{3}x_4^{13}x_5^{12} + x_1^{3}x_2^{3}x_3x_4^{13}x_5^{12}\\ &\quad + x_1^{3}x_2^{3}x_3^{5}x_4^{9}x_5^{12} + x_1^{3}x_2^{3}x_3^{5}x_4^{13}x_5^{8} + x_1^{3}x_2^{3}x_3^{13}x_4x_5^{12} + x_1^{3}x_2^{3}x_3^{13}x_4^{5}x_5^{8}\\ &\quad + x_1^{3}x_2^{5}x_3^{3}x_4^{9}x_5^{12} + x_1^{3}x_2^{5}x_3^{7}x_4^{9}x_5^{8} + x_1^{3}x_2^{5}x_3^{9}x_4^{3}x_5^{12} + x_1^{3}x_2^{5}x_3^{9}x_4^{7}x_5^{8}\\ &\quad + x_1^{3}x_2^{7}x_3^{5}x_4^{9}x_5^{8} + x_1^{3}x_2^{13}x_3^{3}x_4^{5}x_5^{8} + x_1^{7}x_2x_3^{7}x_4^{9}x_5^{8} + x_1^{7}x_2^{3}x_3^{5}x_4^{9}x_5^{8}\\ &\quad + x_1^{7}x_2^{3}x_3^{9}x_4^{5}x_5^{8} + x_1^{7}x_2^{7}x_3x_4^{9}x_5^{8} + x_1^{7}x_2^{7}x_3^{9}x_4x_5^{8} + x_1^{7}x_2^{9}x_3^{3}x_4^{5}x_5^{8},\\
\theta_{16} &= x_1x_2x_3x_4^{14}x_5^{15} + x_1x_2x_3^{14}x_4x_5^{15} + x_1x_2^{3}x_3x_4^{12}x_5^{15} + x_1x_2^{3}x_3^{12}x_4x_5^{15}\\ &\quad + x_1^{3}x_2x_3x_4^{12}x_5^{15} + x_1^{3}x_2x_3^{12}x_4x_5^{15} + x_1^{3}x_2^{5}x_3x_4^{8}x_5^{15} + x_1^{3}x_2^{5}x_3^{8}x_4x_5^{15},\\
\theta_{17} &= x_1x_2x_3x_4^{15}x_5^{14} + x_1x_2x_3^{14}x_4^{15}x_5 + x_1x_2^{3}x_3x_4^{15}x_5^{12} + x_1x_2^{3}x_3^{12}x_4^{15}x_5\\ &\quad + x_1^{3}x_2x_3x_4^{15}x_5^{12} + x_1^{3}x_2x_3^{12}x_4^{15}x_5 + x_1^{3}x_2^{5}x_3x_4^{15}x_5^{8} + x_1^{3}x_2^{5}x_3^{8}x_4^{15}x_5,\\
\theta_{18} &= x_1x_2x_3^{15}x_4x_5^{14} + x_1x_2x_3^{15}x_4^{14}x_5 + x_1x_2^{3}x_3^{15}x_4x_5^{12} + x_1x_2^{3}x_3^{15}x_4^{12}x_5\\ &\quad + x_1^{3}x_2x_3^{15}x_4x_5^{12} + x_1^{3}x_2x_3^{15}x_4^{12}x_5 + x_1^{3}x_2^{5}x_3^{15}x_4x_5^{8} + x_1^{3}x_2^{5}x_3^{15}x_4^{8}x_5,\\
\theta_{19} &= x_1x_2^{15}x_3x_4x_5^{14} + x_1x_2^{15}x_3x_4^{14}x_5 + x_1x_2^{15}x_3^{3}x_4x_5^{12} + x_1x_2^{15}x_3^{3}x_4^{12}x_5\\ &\quad + x_1^{3}x_2^{15}x_3x_4x_5^{12} + x_1^{3}x_2^{15}x_3x_4^{12}x_5 + x_1^{3}x_2^{15}x_3^{5}x_4x_5^{8} + x_1^{3}x_2^{15}x_3^{5}x_4^{8}x_5,\\
\theta_{20} &= x_1^{15}x_2x_3x_4x_5^{14} + x_1^{15}x_2x_3x_4^{14}x_5 + x_1^{15}x_2x_3^{3}x_4x_5^{12} + x_1^{15}x_2x_3^{3}x_4^{12}x_5\\ &\quad + x_1^{15}x_2^{3}x_3x_4x_5^{12} + x_1^{15}x_2^{3}x_3x_4^{12}x_5 + x_1^{15}x_2^{3}x_3^{5}x_4x_5^{8} + x_1^{15}x_2^{3}x_3^{5}x_4^{8}x_5.
\end{align*}

This result shows that the proof in Ph\'uc \cite[Pages 22-23]{ph21} for $QP_5((4)|(2)|^3)$ is refused. If the result in \cite{ph21} for $QP_5((4)|(2)|^3)$ is true then $\dim \widetilde {SF}_5((4)|(2)|^3) =20$. However, we unable to prove this by hand computation. We think that this can be checked by using a computer calculation.

It is well-known that the monomials $x_1^3x_2^5x_3^8x_4$, $x_1^7x_2^7x_3^8x_4^9$, $x_1^7x_2^7x_3^9x_4^8$ are admissible. So, by using Proposition \ref{mdmo}, the leading monomials of the polynomials $\theta_9$, $\theta_{10}$, $\theta_{11}$ and $\theta_t, \, 16\leqslant t\leqslant 20$ are admissible. Hence, from the above equalities and a simple calculation, we obtain the following.

\begin{props} We have
	
\medskip
{\rm i)} $\dim \widetilde{QP}_k((4)|(2)|^3) = 465$.

{\rm ii)} $\widetilde {SF}_5((4)|(2)|^3)$ is an $GL_5$-submodule of $GL_5$-module ${QP}_k((4)|(2)|^3)$ and $$8 \leqslant \dim \widetilde {SF}_5((4)|(2)|^3) \leqslant 20.$$
\end{props}

\bigskip
\subsection{Computation of $(QP_5)_{64}$}\

\medskip
By Lemma \ref{bdd5}, we have
$$ (QP_5)_{64} \cong QP_5((2)|(1)|^5)\bigoplus QP_5((4)|(2)|^4) \bigoplus QP_5((4)|^2(3)|^2|(1)).$$
By Theorems \ref{mdd51} and \ref{mdd52}, we have
$$\dim QP_5((2)|(1)|^4)=124,\ \dim QP_5((4)|(2)|^4) = 465.$$
Hence, we need to compute $QP_5((4)|^2(3)|^2|(1))$.

\subsubsection{Computation of $QP_5((4)|^2(3)|^2|(1))$}\

\medskip
By using a result in \cite{su2}, we have $B_4((4)|^2(3)|^2|(1)) = B_4^+((4)|^2(3)|^2|(1))$ and $|B_4^+((4)|^2(3)|^2|(1))| = 35$. This implies $\dim QP_5^0((4)|^2(3)|^2|(1)) = 5 \times 35 = 175.$ Hence, we need only to compute $QP_5^+((4)|^2(3)|^2|(1))$.
We will prove the following 

\begin{props}\label{mdda61}  $B_5^+((4)|^2|(3)|^{2}|(1))$ is the set contained $915$ admissible monomials $g_j,\, 1 \leqslant j \leqslant 915$, which are determined as in Subsubsection $\ref{sss52}$. Consequently, 
	$$\dim QP_5^+((4)|^2|(3)|^{2}|(1)) \geqslant 915,\ \dim QP_5((4)|^2|(3)|^{2}|(1)) \geqslant 1090.$$
\end{props}

\begin{lems}\label{bdaa61} The following monomials are strictly inadmissible:
	
\medskip
\centerline{\begin{tabular}{llll}
$x_1^{7}x_2^{9}x_3^{3}x_4^{14}x_5^{31}$& $x_1^{7}x_2^{9}x_3^{3}x_4^{31}x_5^{14}$& $x_1^{7}x_2^{9}x_3^{31}x_4^{3}x_5^{14}$& $x_1^{7}x_2^{31}x_3^{9}x_4^{3}x_5^{14} $\cr  $x_1^{31}x_2^{7}x_3^{9}x_4^{3}x_5^{14}$& $x_1^{3}x_2^{15}x_3^{15}x_4^{16}x_5^{15}$& $x_1^{7}x_2^{7}x_3^{11}x_4^{24}x_5^{15}$& $x_1^{7}x_2^{7}x_3^{27}x_4^{8}x_5^{15} $\cr  $x_1^{7}x_2^{7}x_3^{27}x_4^{15}x_5^{8}$& $x_1^{7}x_2^{9}x_3^{3}x_4^{15}x_5^{30}$& $x_1^{7}x_2^{9}x_3^{3}x_4^{30}x_5^{15}$& $x_1^{7}x_2^{9}x_3^{15}x_4^{3}x_5^{30} $\cr  $x_1^{7}x_2^{11}x_3^{15}x_4^{16}x_5^{15}$& $x_1^{7}x_2^{15}x_3^{9}x_4^{3}x_5^{30}$& $x_1^{7}x_2^{15}x_3^{11}x_4^{16}x_5^{15}$& $x_1^{7}x_2^{15}x_3^{16}x_4^{11}x_5^{15} $\cr  $x_1^{7}x_2^{15}x_3^{16}x_4^{15}x_5^{11}$& $x_1^{7}x_2^{15}x_3^{17}x_4^{10}x_5^{15}$& $x_1^{7}x_2^{15}x_3^{17}x_4^{14}x_5^{11}$& $x_1^{7}x_2^{15}x_3^{17}x_4^{15}x_5^{10} $\cr  $x_1^{7}x_2^{15}x_3^{25}x_4^{3}x_5^{14}$& $x_1^{15}x_2^{3}x_3^{15}x_4^{16}x_5^{15}$& $x_1^{15}x_2^{7}x_3^{9}x_4^{3}x_5^{30}$& $x_1^{15}x_2^{7}x_3^{11}x_4^{16}x_5^{15} $\cr  $x_1^{15}x_2^{7}x_3^{16}x_4^{11}x_5^{15}$& $x_1^{15}x_2^{7}x_3^{16}x_4^{15}x_5^{11}$& $x_1^{15}x_2^{7}x_3^{17}x_4^{10}x_5^{15}$& $x_1^{15}x_2^{7}x_3^{17}x_4^{14}x_5^{11} $\cr  $x_1^{15}x_2^{7}x_3^{17}x_4^{15}x_5^{10}$& $x_1^{15}x_2^{7}x_3^{25}x_4^{3}x_5^{14}$& $x_1^{15}x_2^{15}x_3^{3}x_4^{16}x_5^{15}$& $x_1^{15}x_2^{15}x_3^{16}x_4^{3}x_5^{15} $\cr  $x_1^{15}x_2^{15}x_3^{16}x_4^{15}x_5^{3}$& $x_1^{15}x_2^{15}x_3^{17}x_4^{2}x_5^{15}$& $x_1^{15}x_2^{15}x_3^{17}x_4^{14}x_5^{3}$& $x_1^{15}x_2^{15}x_3^{17}x_4^{15}x_5^{2} $\cr  $x_1^{15}x_2^{19}x_3^{5}x_4^{10}x_5^{15}$& $x_1^{15}x_2^{19}x_3^{5}x_4^{14}x_5^{11}$& $x_1^{15}x_2^{19}x_3^{5}x_4^{15}x_5^{10}$& $x_1^{15}x_2^{19}x_3^{15}x_4^{5}x_5^{10} $\cr  $x_1^{15}x_2^{23}x_3^{9}x_4^{3}x_5^{14}$& $x_1^{15}x_2^{23}x_3^{9}x_4^{3}x_5^{14}$& & \cr   
\end{tabular}}
\end{lems}
\begin{proof} We prove the lemma for $w_1=x_1^{7}x_2^{9}x_3^{3}x_4^{14}x_5^{31}$, $w_2 = x_1^{7}x_2^{7}x_3^{11}x_4^{24}x_5^{15}$, $w_3 = x_1^{7}x_2^{15}x_3^{17}x_4^{10}x_5^{15}$ and $w_4 = x_1^{15}x_2^{19}x_3^{5}x_4^{10}x_5^{15}$. The others can be proved by a similar computation. We have
\begin{align*}
w_1 &= x_1^{4}x_2^{7}x_3^{11}x_4^{11}x_5^{31} + x_1^{4}x_2^{11}x_3^{11}x_4^{7}x_5^{31} + x_1^{5}x_2^{11}x_3^{3}x_4^{14}x_5^{31} + x_1^{5}x_2^{11}x_3^{6}x_4^{11}x_5^{31}\\
&\quad + x_1^{5}x_2^{11}x_3^{10}x_4^{7}x_5^{31} + x_1^{7}x_2^{5}x_3^{10}x_4^{11}x_5^{31} + x_1^{7}x_2^{7}x_3^{8}x_4^{11}x_5^{31} + x_1^{7}x_2^{7}x_3^{10}x_4^{9}x_5^{31}\\
&\quad + x_1^{7}x_2^{7}x_3^{11}x_4^{8}x_5^{31} + x_1^{7}x_2^{8}x_3^{11}x_4^{7}x_5^{31} +  Sq^1\big(x_1^{7}x_2^{7}x_3^{5}x_4^{13}x_5^{31} + x_1^{7}x_2^{7}x_3^{11}x_4^{7}x_5^{31}\big)\\
&\quad +  Sq^2\big(x_1^{7}x_2^{7}x_3^{3}x_4^{14}x_5^{31} + x_1^{7}x_2^{7}x_3^{6}x_4^{11}x_5^{31} + x_1^{7}x_2^{7}x_3^{10}x_4^{7}x_5^{31}\big)\\
&\quad +  Sq^4\big(x_1^{4}x_2^{7}x_3^{11}x_4^{7}x_5^{31} + x_1^{5}x_2^{7}x_3^{3}x_4^{14}x_5^{31} + x_1^{5}x_2^{7}x_3^{6}x_4^{11}x_5^{31} + x_1^{5}x_2^{7}x_3^{10}x_4^{7}x_5^{31}\\
&\quad + x_1^{11}x_2^{5}x_3^{6}x_4^{7}x_5^{31}\big) +  Sq^8\big(x_1^{7}x_2^{5}x_3^{6}x_4^{7}x_5^{31}\big)   \ \mbox{mod}\big(P_5^-((4)|^2|(3)|^{2}|(1)\big). 
\end{align*}
Hence, the monomial $w_1$ is strictly inadmissible.
\begin{align*}
w_2 &= x_1^{4}x_2^{7}x_3^{11}x_4^{15}x_5^{27} + x_1^{4}x_2^{7}x_3^{11}x_4^{27}x_5^{15} + x_1^{4}x_2^{11}x_3^{7}x_4^{15}x_5^{27} + x_1^{4}x_2^{11}x_3^{7}x_4^{27}x_5^{15}\\ 
&\quad + x_1^{4}x_2^{11}x_3^{19}x_4^{15}x_5^{15} + x_1^{5}x_2^{3}x_3^{11}x_4^{15}x_5^{30} + x_1^{5}x_2^{3}x_3^{11}x_4^{30}x_5^{15} + x_1^{5}x_2^{3}x_3^{14}x_4^{15}x_5^{27}\\ 
&\quad + x_1^{5}x_2^{3}x_3^{14}x_4^{27}x_5^{15} + x_1^{5}x_2^{6}x_3^{11}x_4^{15}x_5^{27} + x_1^{5}x_2^{6}x_3^{11}x_4^{27}x_5^{15} + x_1^{5}x_2^{10}x_3^{7}x_4^{15}x_5^{27}\\ 
&\quad + x_1^{5}x_2^{10}x_3^{7}x_4^{27}x_5^{15} + x_1^{5}x_2^{10}x_3^{19}x_4^{15}x_5^{15} + x_1^{5}x_2^{11}x_3^{7}x_4^{15}x_5^{26} + x_1^{5}x_2^{11}x_3^{7}x_4^{26}x_5^{15}\\ 
&\quad + x_1^{5}x_2^{11}x_3^{11}x_4^{15}x_5^{22} + x_1^{5}x_2^{11}x_3^{11}x_4^{22}x_5^{15} + x_1^{5}x_2^{11}x_3^{14}x_4^{15}x_5^{19} + x_1^{5}x_2^{11}x_3^{14}x_4^{19}x_5^{15}\\ 
&\quad + x_1^{5}x_2^{11}x_3^{18}x_4^{15}x_5^{15} + x_1^{7}x_2^{3}x_3^{9}x_4^{15}x_5^{30} + x_1^{7}x_2^{3}x_3^{9}x_4^{30}x_5^{15} + x_1^{7}x_2^{3}x_3^{14}x_4^{15}x_5^{25}\\ 
&\quad + x_1^{7}x_2^{3}x_3^{14}x_4^{25}x_5^{15} + x_1^{7}x_2^{5}x_3^{11}x_4^{15}x_5^{26} + x_1^{7}x_2^{5}x_3^{11}x_4^{26}x_5^{15} + x_1^{7}x_2^{6}x_3^{9}x_4^{15}x_5^{27}\\ 
&\quad + x_1^{7}x_2^{6}x_3^{9}x_4^{27}x_5^{15} + x_1^{7}x_2^{7}x_3^{8}x_4^{15}x_5^{27} + x_1^{7}x_2^{7}x_3^{8}x_4^{27}x_5^{15} + x_1^{7}x_2^{7}x_3^{9}x_4^{15}x_5^{26}\\ 
&\quad + x_1^{7}x_2^{7}x_3^{9}x_4^{26}x_5^{15} + x_1^{7}x_2^{7}x_3^{11}x_4^{15}x_5^{24} +  Sq^1\big(x_1^{7}x_2^{5}x_3^{7}x_4^{15}x_5^{29} + x_1^{7}x_2^{5}x_3^{7}x_4^{29}x_5^{15}\\ 
&\quad + x_1^{7}x_2^{5}x_3^{21}x_4^{15}x_5^{15} + x_1^{7}x_2^{7}x_3^{7}x_4^{15}x_5^{27} + x_1^{7}x_2^{7}x_3^{7}x_4^{27}x_5^{15} + x_1^{7}x_2^{7}x_3^{13}x_4^{15}x_5^{21}\\ 
&\quad + x_1^{7}x_2^{7}x_3^{13}x_4^{21}x_5^{15} + x_1^{7}x_2^{7}x_3^{19}x_4^{15}x_5^{15}\big) +  Sq^2\big(x_1^{7}x_2^{3}x_3^{7}x_4^{15}x_5^{30} + x_1^{7}x_2^{3}x_3^{7}x_4^{22}x_5^{23}\\ 
&\quad + x_1^{7}x_2^{3}x_3^{7}x_4^{23}x_5^{22} + x_1^{7}x_2^{3}x_3^{7}x_4^{30}x_5^{15} + x_1^{7}x_2^{3}x_3^{14}x_4^{15}x_5^{23} + x_1^{7}x_2^{3}x_3^{14}x_4^{23}x_5^{15}\\ 
&\quad + x_1^{7}x_2^{6}x_3^{7}x_4^{15}x_5^{27} + x_1^{7}x_2^{6}x_3^{7}x_4^{27}x_5^{15} + x_1^{7}x_2^{6}x_3^{19}x_4^{15}x_5^{15} + x_1^{7}x_2^{7}x_3^{7}x_4^{15}x_5^{26}\\ 
&\quad + x_1^{7}x_2^{7}x_3^{7}x_4^{26}x_5^{15} + x_1^{7}x_2^{7}x_3^{11}x_4^{15}x_5^{22} + x_1^{7}x_2^{7}x_3^{11}x_4^{24}x_5^{15} + x_1^{7}x_2^{7}x_3^{14}x_4^{15}x_5^{19}\\ 
&\quad + x_1^{7}x_2^{7}x_3^{14}x_4^{19}x_5^{15} + x_1^{7}x_2^{7}x_3^{18}x_4^{15}x_5^{15}\big) + Sq^4\big(x_1^{4}x_2^{7}x_3^{7}x_4^{15}x_5^{27} + x_1^{4}x_2^{7}x_3^{7}x_4^{27}x_5^{15}\\ 
&\quad + x_1^{4}x_2^{7}x_3^{19}x_4^{15}x_5^{15} + x_1^{5}x_2^{3}x_3^{7}x_4^{15}x_5^{30} + x_1^{5}x_2^{3}x_3^{7}x_4^{30}x_5^{15} + x_1^{5}x_2^{3}x_3^{14}x_4^{15}x_5^{23}\\ 
&\quad + x_1^{5}x_2^{3}x_3^{14}x_4^{23}x_5^{15} + x_1^{5}x_2^{6}x_3^{7}x_4^{15}x_5^{27} + x_1^{5}x_2^{6}x_3^{7}x_4^{27}x_5^{15} + x_1^{5}x_2^{6}x_3^{19}x_4^{15}x_5^{15}\\ 
&\quad + x_1^{5}x_2^{7}x_3^{7}x_4^{15}x_5^{26} + x_1^{5}x_2^{7}x_3^{7}x_4^{26}x_5^{15} + x_1^{5}x_2^{7}x_3^{11}x_4^{15}x_5^{22} + x_1^{5}x_2^{7}x_3^{11}x_4^{22}x_5^{15}\\ 
&\quad + x_1^{5}x_2^{7}x_3^{14}x_4^{15}x_5^{19} + x_1^{5}x_2^{7}x_3^{14}x_4^{19}x_5^{15} + x_1^{5}x_2^{7}x_3^{18}x_4^{15}x_5^{15} + x_1^{11}x_2^{5}x_3^{7}x_4^{15}x_5^{22}\\ 
&\quad + x_1^{11}x_2^{5}x_3^{7}x_4^{22}x_5^{15} + x_1^{11}x_2^{5}x_3^{14}x_4^{15}x_5^{15}\big) +  Sq^8\big(x_1^{7}x_2^{5}x_3^{7}x_4^{15}x_5^{22} + x_1^{7}x_2^{5}x_3^{7}x_4^{22}x_5^{15}\\ 
&\quad + x_1^{7}x_2^{5}x_3^{14}x_4^{15}x_5^{15}\big)  \ \mbox{mod}\big(P_5^-((4)|^2|(3)|^{2}|(1)\big). 
\end{align*}
From this we see that the monomial $w_2$ is strictly inadmissible.
\begin{align*}
w_3 & = x_1^{4}x_2^{15}x_3^{15}x_4^{11}x_5^{19} + x_1^{4}x_2^{15}x_3^{19}x_4^{11}x_5^{15} + x_1^{4}x_2^{19}x_3^{15}x_4^{11}x_5^{15}\\ 
&\quad + x_1^{5}x_2^{15}x_3^{15}x_4^{10}x_5^{19} + x_1^{5}x_2^{15}x_3^{19}x_4^{10}x_5^{15} + x_1^{5}x_2^{19}x_3^{15}x_4^{10}x_5^{15}\\ 
&\quad  + x_1^{7}x_2^{8}x_3^{15}x_4^{11}x_5^{23} + x_1^{7}x_2^{8}x_3^{15}x_4^{19}x_5^{15}+ x_1^{7}x_2^{8}x_3^{23}x_4^{11}x_5^{15} + x_1^{7}x_2^{9}x_3^{15}x_4^{10}x_5^{23}\\ 
&\quad  + x_1^{7}x_2^{9}x_3^{15}x_4^{18}x_5^{15} + x_1^{7}x_2^{9}x_3^{23}x_4^{10}x_5^{15} + x_1^{7}x_2^{15}x_3^{15}x_4^{10}x_5^{17}\\ 
&\quad  + x_1^{7}x_2^{15}x_3^{15}x_4^{11}x_5^{16} + x_1^{7}x_2^{15}x_3^{16}x_4^{11}x_5^{15} +  Sq^1\big(x_1^{7}x_2^{15}x_3^{15}x_4^{11}x_5^{15}\big)\\ 
&\quad +  Sq^2\big(x_1^{7}x_2^{15}x_3^{15}x_4^{10}x_5^{15}\big)+  Sq^4\big(x_1^{4}x_2^{15}x_3^{15}x_4^{11}x_5^{15} + x_1^{5}x_2^{15}x_3^{15}x_4^{10}x_5^{15}\big)\\ 
&\quad + Sq^8\big(x_1^{7}x_2^{8}x_3^{15}x_4^{11}x_5^{15}+ x_1^{7}x_2^{9}x_3^{15}x_4^{10}x_5^{15}\big)\ \mbox{mod}\big(P_5^-((4)|^2|(3)|^{2}|(1)\big). 
\end{align*}
Hence, the monomial $w_3$ is strictly inadmissible.
\begin{align*}
w_4 &= x_1^{8}x_2^{15}x_3^{7}x_4^{11}x_5^{23} + x_1^{8}x_2^{15}x_3^{7}x_4^{19}x_5^{15} + x_1^{8}x_2^{23}x_3^{7}x_4^{11}x_5^{15} + x_1^{9}x_2^{15}x_3^{7}x_4^{10}x_5^{23}\\ 
&\quad + x_1^{9}x_2^{15}x_3^{7}x_4^{18}x_5^{15} + x_1^{9}x_2^{23}x_3^{7}x_4^{10}x_5^{15} + x_1^{11}x_2^{15}x_3^{4}x_4^{11}x_5^{23} + x_1^{11}x_2^{15}x_3^{4}x_4^{19}x_5^{15}\\ 
&\quad + x_1^{11}x_2^{15}x_3^{5}x_4^{10}x_5^{23} + x_1^{11}x_2^{15}x_3^{5}x_4^{18}x_5^{15} + x_1^{11}x_2^{23}x_3^{4}x_4^{11}x_5^{15} + x_1^{11}x_2^{23}x_3^{5}x_4^{10}x_5^{15}\\ 
&\quad + x_1^{15}x_2^{15}x_3^{4}x_4^{11}x_5^{19} + x_1^{15}x_2^{15}x_3^{5}x_4^{10}x_5^{19} + x_1^{15}x_2^{15}x_3^{7}x_4^{10}x_5^{17} + x_1^{15}x_2^{15}x_3^{7}x_4^{11}x_5^{16}\\ 
&\quad + x_1^{15}x_2^{16}x_3^{7}x_4^{11}x_5^{15} + x_1^{15}x_2^{17}x_3^{7}x_4^{10}x_5^{15} + x_1^{15}x_2^{19}x_3^{4}x_4^{11}x_5^{15}\\ 
&\quad + Sq^1\big(x_1^{15}x_2^{15}x_3^{7}x_4^{11}x_5^{15}\big) + Sq^2\big(x_1^{15}x_2^{15}x_3^{7}x_4^{10}x_5^{15}\big) +  Sq^4\big(x_1^{15}x_2^{15}x_3^{4}x_4^{11}x_5^{15}\\ 
&\quad + x_1^{15}x_2^{15}x_3^{5}x_4^{10}x_5^{15}\big) +  Sq^8\big(x_1^{8}x_2^{15}x_3^{7}x_4^{11}x_5^{15} + x_1^{9}x_2^{15}x_3^{7}x_4^{10}x_5^{15}\\ 
&\quad + x_1^{11}x_2^{15}x_3^{4}x_4^{11}x_5^{15} + x_1^{11}x_2^{15}x_3^{5}x_4^{10}x_5^{15}\big)\ \mbox{mod}\big(P_5^-((4)|^2|(3)|^{2}|(1)\big). 
\end{align*}
This equality implies the monomial $w_4$ is strictly inadmissible.
\end{proof}
\begin{proof}[Proof of Proposition $\ref{mdda61}$] Let $x \in P_5^+(\omega)$ be an admissible monomial with $\omega := (4)|^2|(3)|^{2}|(1)$. Then $x = X_ry^2$ with $1 \leqslant r \leqslant 5$ and $y$ an admissible monomial of weight vector $(4)|(3)|^{2}|(1)$. By a direct computation we see that if $x \ne g_j,\, 1 \leqslant j \leqslant 930$, then either $x$ is one of monomials as given in Lemma \ref{bdaa61} or there is a monomial $w$ as given in one of Lemmas \ref{bda72}, \ref{bda73} such that $x = wz_1^{2^u}$ with $u$ nonnegative integers, $2\leqslant u \leqslant 4$, and $z_1$ a monomial of weight vector $(3)|^{4-u}|(1)$. By Theorem \ref{dlcb1}, $x$ is inadmissible. This contradicts the fact that $x$ is admissible. Hence, $x=g_j$ for suitable $1 \leqslant j \leqslant 930$.
	
Now we prove that the set $\{[g_j]_{\omega}:1 \leqslant j \leqslant 915\}$ is linearly independent in $QP_5(\omega)$.
Suppose there is a linear relation
\begin{equation}\label{ctd51}
\mathcal S:= \sum_{1\leqslant j \leqslant 930}\gamma_jg_j \equiv_{\omega} 0,
\end{equation}
where $\gamma_j \in \mathbb F_2$. 
	
Let $u_t,\, 1\leqslant t \leqslant 35$, be as in Subsubsection \ref{sss52} and the homomorphism $p_{(i;I)}:P_5\to P_4$ which is defined by \eqref{ct23} for $k=5$. From Lemma \ref{bdm}, we see that $p_{(i;I)}$ passes to a homomorphism from $QP_5(\omega)$ to $QP_4(\omega)$. For each $(i;I) \in \mathcal N_5$, we compute $p_{(i;I)}(\mathcal S)$ in terms of $w_u,\, 1\leqslant u \leqslant 35$ (mod$P_4^-(\omega)+\mathcal A^+P_4$). By a direct computation we see that $p_{(i;I)}(\mathcal S)\equiv_\omega 0$ for all $(i;I) \in \mathcal N_5$ if and only if
\begin{equation}\label{ctd52}
\mathcal S = \sum_{331\leqslant j \leqslant 335}\gamma_j\tilde \theta_j + \sum_{916\leqslant j \leqslant 930}\gamma_j\tilde \theta_j \equiv_{\omega} 0,
\end{equation}
where the polynomials $\tilde \theta_j$ are determined as in Subsubsection \ref{sss53}. Hence, we get $\dim \widetilde{QP}_5(\omega) = 1085$ and $$\widetilde{SF}_5(\omega) = \langle \{[\tilde \theta_j]_\omega : 331\leqslant j \leqslant 335 \mbox{ or } 916\leqslant j \leqslant 930\}\rangle.$$ 

From \cite{su2} we see that $x_1^7x_2^7x_3^{11}x_4^8$ is admissible. So, by using Proposition \ref{mdmo}, the leading monomials of the polynomials $\tilde\theta_j$, $331\leqslant j\leqslant 335$ are admissible. 
 Hence, $5 \leqslant \dim \widetilde{SF}_5(\omega) \leqslant 20$. The proposition is proved. 
\end{proof}
Thus, we have the following result.
\begin{thms} There exist $1679$	
admissible monomials of degree $64$ in $P_5$. Consequently $1679 \leqslant \dim (QP_5)_{64} \leqslant 1694.$   
\end{thms}

\subsubsection{The admissible monomials of weight vector $(4)|^2|(3)|^{2}|(1)$ in $P_5$}\label{sss52}\

\bigskip
$B_4((4)|^2|(3)|^{2}|(1)) = B_4^+((4)|^2|(3)|^{2}|(1))$ is the set of 35 monomials $u_t,\, 1 \leqslant t \leqslant 35$, which are determined as follows:

\bigskip
\centerline{\begin{tabular}{llll}  
$u_{1} =  x_1^{3}x_2^{15}x_3^{15}x_4^{31}$& $u_{2} =  x_1^{3}x_2^{15}x_3^{31}x_4^{15}$& $u_{3} =  x_1^{3}x_2^{31}x_3^{15}x_4^{15} $\cr  $u_{4} =  x_1^{7}x_2^{11}x_3^{15}x_4^{31}$& $u_{5} =  x_1^{7}x_2^{11}x_3^{31}x_4^{15}$& $u_{6} =  x_1^{7}x_2^{15}x_3^{11}x_4^{31} $\cr $u_{7} =  x_1^{7}x_2^{15}x_3^{15}x_4^{27}$& $u_{8} =  x_1^{7}x_2^{15}x_3^{27}x_4^{15}$& $u_{9} =  x_1^{7}x_2^{15}x_3^{31}x_4^{11} $\cr  $u_{10} =  x_1^{7}x_2^{27}x_3^{15}x_4^{15}$& $u_{11} =  x_1^{7}x_2^{31}x_3^{11}x_4^{15}$& $u_{12} =  x_1^{7}x_2^{31}x_3^{15}x_4^{11} $\cr $u_{13} =  x_1^{15}x_2^{3}x_3^{15}x_4^{31}$& $u_{14} =  x_1^{15}x_2^{3}x_3^{31}x_4^{15}$& $u_{15} =  x_1^{15}x_2^{7}x_3^{11}x_4^{31} $\cr  $u_{16} =  x_1^{15}x_2^{7}x_3^{15}x_4^{27}$& $u_{17} =  x_1^{15}x_2^{7}x_3^{27}x_4^{15}$& $u_{18} =  x_1^{15}x_2^{7}x_3^{31}x_4^{11} $\cr $u_{19} =  x_1^{15}x_2^{15}x_3^{3}x_4^{31}$& $u_{20} =  x_1^{15}x_2^{15}x_3^{7}x_4^{27}$& $u_{21} =  x_1^{15}x_2^{15}x_3^{15}x_4^{19} $\cr  $u_{22} =  x_1^{15}x_2^{15}x_3^{19}x_4^{15}$& $u_{23} =  x_1^{15}x_2^{15}x_3^{23}x_4^{11}$& $u_{24} =  x_1^{15}x_2^{15}x_3^{31}x_4^{3} $\cr  
$u_{25} =  x_1^{15}x_2^{23}x_3^{11}x_4^{15}$& $u_{26} =  x_1^{15}x_2^{23}x_3^{15}x_4^{11}$& $u_{27} =  x_1^{15}x_2^{31}x_3^{3}x_4^{15} $\cr  
\end{tabular}}
\centerline{\begin{tabular}{lll}  
$u_{28} =  x_1^{15}x_2^{31}x_3^{7}x_4^{11}$& $u_{29} =  x_1^{15}x_2^{31}x_3^{15}x_4^{3}$& $u_{30} =  x_1^{31}x_2^{3}x_3^{15}x_4^{15} $\cr  $u_{31} =  x_1^{31}x_2^{7}x_3^{11}x_4^{15}$& $u_{32} =  x_1^{31}x_2^{7}x_3^{15}x_4^{11}$& $u_{33} =  x_1^{31}x_2^{15}x_3^{3}x_4^{15} $\cr  $u_{34} =  x_1^{31}x_2^{15}x_3^{7}x_4^{11}$& $u_{35} =  x_1^{31}x_2^{15}x_3^{15}x_4^{3}$& \cr 
\end{tabular}}

$$B_5^0((4)|^2|(3)|^2|(1)) =  \{\theta_{J_i}(y): y\in B_4((4)|^2|(3)|^2|(1)),\, 1\leqslant i \leqslant 5\}.$$ 

For a weight vector $\omega$ and a positive integer $d$, set $$ B(d;\omega) = \{x_i^{2^d-1}\theta_{J_i}(u): u\in B_4(\omega),\, 1\leqslant i \leqslant 5\}.$$ 

It is easy to see that 
$$B(5;(3)|^2|(2)|^2)\cup B(4;(3)|^2|(2)|^2|(1)) \cup B(2;(3)|^4|(1))\subset B_5^+((4)|^2|(3)|^2|(1)).$$  From \cite{su2} we have  $|B_4((3)|^2|(2)|^2)| = 67$ and $B(5;(3)|^2|(2)|^2)$ is the set of 335 monomials $g_j,\, 1\leqslant j \leqslant 335$, which are determined as follows:

\medskip
\centerline{
\begin{tabular}{lll}  
1.\, $x_1x_2^{2}x_3^{15}x_4^{15}x_5^{31}$& 2.\, $x_1x_2^{2}x_3^{15}x_4^{31}x_5^{15}$& 3.\, $x_1x_2^{2}x_3^{31}x_4^{15}x_5^{15} $\cr  4.\, $x_1x_2^{3}x_3^{14}x_4^{15}x_5^{31}$& 5.\, $x_1x_2^{3}x_3^{14}x_4^{31}x_5^{15}$& 6.\, $x_1x_2^{3}x_3^{15}x_4^{14}x_5^{31} $\cr  7.\, $x_1x_2^{3}x_3^{15}x_4^{31}x_5^{14}$& 8.\, $x_1x_2^{3}x_3^{31}x_4^{14}x_5^{15}$& 9.\, $x_1x_2^{3}x_3^{31}x_4^{15}x_5^{14} $\cr  10.\, $x_1x_2^{6}x_3^{11}x_4^{15}x_5^{31}$& 11.\, $x_1x_2^{6}x_3^{11}x_4^{31}x_5^{15}$& 12.\, $x_1x_2^{6}x_3^{15}x_4^{11}x_5^{31} $\cr  13.\, $x_1x_2^{6}x_3^{15}x_4^{31}x_5^{11}$& 14.\, $x_1x_2^{6}x_3^{31}x_4^{11}x_5^{15}$& 15.\, $x_1x_2^{6}x_3^{31}x_4^{15}x_5^{11} $\cr  16.\, $x_1x_2^{7}x_3^{10}x_4^{15}x_5^{31}$& 17.\, $x_1x_2^{7}x_3^{10}x_4^{31}x_5^{15}$& 18.\, $x_1x_2^{7}x_3^{11}x_4^{14}x_5^{31} $\cr  19.\, $x_1x_2^{7}x_3^{11}x_4^{31}x_5^{14}$& 20.\, $x_1x_2^{7}x_3^{14}x_4^{11}x_5^{31}$& 21.\, $x_1x_2^{7}x_3^{14}x_4^{31}x_5^{11} $\cr  22.\, $x_1x_2^{7}x_3^{15}x_4^{10}x_5^{31}$& 23.\, $x_1x_2^{7}x_3^{15}x_4^{31}x_5^{10}$& 24.\, $x_1x_2^{7}x_3^{31}x_4^{10}x_5^{15} $\cr
\end{tabular}}
\centerline{\begin{tabular}{lll}  
109.\, $x_1^{3}x_2^{13}x_3^{2}x_4^{15}x_5^{31}$& 110.\, $x_1^{3}x_2^{13}x_3^{2}x_4^{31}x_5^{15}$& 111.\, $x_1^{3}x_2^{13}x_3^{3}x_4^{14}x_5^{31} $\cr  112.\, $x_1^{3}x_2^{13}x_3^{3}x_4^{31}x_5^{14}$& 113.\, $x_1^{3}x_2^{13}x_3^{6}x_4^{11}x_5^{31}$& 114.\, $x_1^{3}x_2^{13}x_3^{6}x_4^{31}x_5^{11} $\cr  115.\, $x_1^{3}x_2^{13}x_3^{7}x_4^{10}x_5^{31}$& 116.\, $x_1^{3}x_2^{13}x_3^{7}x_4^{31}x_5^{10}$& 117.\, $x_1^{3}x_2^{13}x_3^{14}x_4^{3}x_5^{31} $\cr  118.\, $x_1^{3}x_2^{13}x_3^{14}x_4^{31}x_5^{3}$& 119.\, $x_1^{3}x_2^{13}x_3^{15}x_4^{2}x_5^{31}$& 120.\, $x_1^{3}x_2^{13}x_3^{15}x_4^{31}x_5^{2} $\cr  121.\, $x_1^{3}x_2^{13}x_3^{31}x_4^{2}x_5^{15}$& 122.\, $x_1^{3}x_2^{13}x_3^{31}x_4^{3}x_5^{14}$& 123.\, $x_1^{3}x_2^{13}x_3^{31}x_4^{6}x_5^{11} $\cr  124.\, $x_1^{3}x_2^{13}x_3^{31}x_4^{7}x_5^{10}$& 125.\, $x_1^{3}x_2^{13}x_3^{31}x_4^{14}x_5^{3}$& 126.\, $x_1^{3}x_2^{13}x_3^{31}x_4^{15}x_5^{2} $\cr  127.\, $x_1^{3}x_2^{15}x_3x_4^{14}x_5^{31}$& 128.\, $x_1^{3}x_2^{15}x_3x_4^{31}x_5^{14}$& 129.\, $x_1^{3}x_2^{15}x_3^{3}x_4^{12}x_5^{31} $\cr  130.\, $x_1^{3}x_2^{15}x_3^{3}x_4^{31}x_5^{12}$& 131.\, $x_1^{3}x_2^{15}x_3^{5}x_4^{10}x_5^{31}$& 132.\, $x_1^{3}x_2^{15}x_3^{5}x_4^{31}x_5^{10} $\cr
\end{tabular}}
\centerline{\begin{tabular}{lll}  
253.\, $x_1^{15}x_2^{31}x_3x_4^{2}x_5^{15}$& 254.\, $x_1^{15}x_2^{31}x_3x_4^{3}x_5^{14}$& 255.\, $x_1^{15}x_2^{31}x_3x_4^{6}x_5^{11} $\cr  256.\, $x_1^{15}x_2^{31}x_3x_4^{7}x_5^{10}$& 257.\, $x_1^{15}x_2^{31}x_3x_4^{14}x_5^{3}$& 258.\, $x_1^{15}x_2^{31}x_3x_4^{15}x_5^{2} $\cr  259.\, $x_1^{15}x_2^{31}x_3^{3}x_4x_5^{14}$& 260.\, $x_1^{15}x_2^{31}x_3^{3}x_4^{3}x_5^{12}$& 261.\, $x_1^{15}x_2^{31}x_3^{3}x_4^{5}x_5^{10} $\cr  262.\, $x_1^{15}x_2^{31}x_3^{3}x_4^{13}x_5^{2}$& 263.\, $x_1^{15}x_2^{31}x_3^{7}x_4x_5^{10}$& 264.\, $x_1^{15}x_2^{31}x_3^{15}x_4x_5^{2} $\cr  265.\, $x_1^{31}x_2x_3^{2}x_4^{15}x_5^{15}$& 266.\, $x_1^{31}x_2x_3^{3}x_4^{14}x_5^{15}$& 267.\, $x_1^{31}x_2x_3^{3}x_4^{15}x_5^{14} $\cr  268.\, $x_1^{31}x_2x_3^{6}x_4^{11}x_5^{15}$& 269.\, $x_1^{31}x_2x_3^{6}x_4^{15}x_5^{11}$& 270.\, $x_1^{31}x_2x_3^{7}x_4^{10}x_5^{15} $\cr  271.\, $x_1^{31}x_2x_3^{7}x_4^{11}x_5^{14}$& 272.\, $x_1^{31}x_2x_3^{7}x_4^{14}x_5^{11}$& 273.\, $x_1^{31}x_2x_3^{7}x_4^{15}x_5^{10} $\cr  274.\, $x_1^{31}x_2x_3^{14}x_4^{3}x_5^{15}$& 275.\, $x_1^{31}x_2x_3^{14}x_4^{7}x_5^{11}$& 276.\, $x_1^{31}x_2x_3^{14}x_4^{15}x_5^{3} $\cr
\end{tabular}}

\medskip
$B(4;(3)|^2|(2)|^2|(1))\cup B(2;(3)|^4|(1))\setminus B(5,(3)|^2|(2)|^2)$ is the set of 425 monomials $g_j,\, 336\leqslant j \leqslant 760$, which are determined as follows:

\medskip
\centerline{\begin{tabular}{lll}
336.\, $x_1x_2^{3}x_3^{15}x_4^{15}x_5^{30}$& 337.\, $x_1x_2^{3}x_3^{15}x_4^{30}x_5^{15}$& 338.\, $x_1x_2^{3}x_3^{30}x_4^{15}x_5^{15} $\cr  339.\, $x_1x_2^{6}x_3^{15}x_4^{15}x_5^{27}$& 340.\, $x_1x_2^{6}x_3^{15}x_4^{27}x_5^{15}$& 341.\, $x_1x_2^{6}x_3^{27}x_4^{15}x_5^{15} $\cr  342.\, $x_1x_2^{7}x_3^{11}x_4^{15}x_5^{30}$& 343.\, $x_1x_2^{7}x_3^{11}x_4^{30}x_5^{15}$& 344.\, $x_1x_2^{7}x_3^{14}x_4^{15}x_5^{27} $\cr  345.\, $x_1x_2^{7}x_3^{14}x_4^{27}x_5^{15}$& 346.\, $x_1x_2^{7}x_3^{15}x_4^{11}x_5^{30}$& 347.\, $x_1x_2^{7}x_3^{15}x_4^{14}x_5^{27} $\cr  348.\, $x_1x_2^{7}x_3^{15}x_4^{15}x_5^{26}$& 349.\, $x_1x_2^{7}x_3^{15}x_4^{26}x_5^{15}$& 350.\, $x_1x_2^{7}x_3^{15}x_4^{27}x_5^{14} $\cr  351.\, $x_1x_2^{7}x_3^{15}x_4^{30}x_5^{11}$& 352.\, $x_1x_2^{7}x_3^{26}x_4^{15}x_5^{15}$& 353.\, $x_1x_2^{7}x_3^{27}x_4^{14}x_5^{15} $\cr  354.\, $x_1x_2^{7}x_3^{27}x_4^{15}x_5^{14}$& 355.\, $x_1x_2^{7}x_3^{30}x_4^{11}x_5^{15}$& 356.\, $x_1x_2^{7}x_3^{30}x_4^{15}x_5^{11} $\cr  357.\, $x_1x_2^{14}x_3^{7}x_4^{15}x_5^{27}$& 358.\, $x_1x_2^{14}x_3^{7}x_4^{27}x_5^{15}$& 359.\, $x_1x_2^{14}x_3^{15}x_4^{7}x_5^{27} $\cr
\end{tabular}}
\centerline{\begin{tabular}{lll}  
390.\, $x_1x_2^{15}x_3^{23}x_4^{10}x_5^{15}$& 391.\, $x_1x_2^{15}x_3^{23}x_4^{15}x_5^{10}$& 392.\, $x_1x_2^{15}x_3^{30}x_4^{3}x_5^{15} $\cr  393.\, $x_1x_2^{15}x_3^{30}x_4^{7}x_5^{11}$& 394.\, $x_1x_2^{15}x_3^{30}x_4^{15}x_5^{3}$& 395.\, $x_1x_2^{30}x_3^{3}x_4^{15}x_5^{15} $\cr  396.\, $x_1x_2^{30}x_3^{7}x_4^{11}x_5^{15}$& 397.\, $x_1x_2^{30}x_3^{7}x_4^{15}x_5^{11}$& 398.\, $x_1x_2^{30}x_3^{15}x_4^{3}x_5^{15} $\cr  399.\, $x_1x_2^{30}x_3^{15}x_4^{7}x_5^{11}$& 400.\, $x_1x_2^{30}x_3^{15}x_4^{15}x_5^{3}$& 401.\, $x_1^{3}x_2x_3^{15}x_4^{15}x_5^{30} $\cr  402.\, $x_1^{3}x_2x_3^{15}x_4^{30}x_5^{15}$& 403.\, $x_1^{3}x_2x_3^{30}x_4^{15}x_5^{15}$& 404.\, $x_1^{3}x_2^{3}x_3^{13}x_4^{15}x_5^{30} $\cr  405.\, $x_1^{3}x_2^{3}x_3^{13}x_4^{30}x_5^{15}$& 406.\, $x_1^{3}x_2^{3}x_3^{15}x_4^{13}x_5^{30}$& 407.\, $x_1^{3}x_2^{3}x_3^{15}x_4^{15}x_5^{28} $\cr  408.\, $x_1^{3}x_2^{3}x_3^{15}x_4^{28}x_5^{15}$& 409.\, $x_1^{3}x_2^{3}x_3^{15}x_4^{29}x_5^{14}$& 410.\, $x_1^{3}x_2^{3}x_3^{28}x_4^{15}x_5^{15} $\cr  411.\, $x_1^{3}x_2^{3}x_3^{29}x_4^{14}x_5^{15}$& 412.\, $x_1^{3}x_2^{3}x_3^{29}x_4^{15}x_5^{14}$& 413.\, $x_1^{3}x_2^{5}x_3^{11}x_4^{15}x_5^{30} $\cr
\end{tabular}}
\centerline{\begin{tabular}{lll}  
534.\, $x_1^{7}x_2x_3^{15}x_4^{15}x_5^{26}$& 535.\, $x_1^{7}x_2x_3^{15}x_4^{26}x_5^{15}$& 536.\, $x_1^{7}x_2x_3^{15}x_4^{27}x_5^{14} $\cr  537.\, $x_1^{7}x_2x_3^{15}x_4^{30}x_5^{11}$& 538.\, $x_1^{7}x_2x_3^{26}x_4^{15}x_5^{15}$& 539.\, $x_1^{7}x_2x_3^{27}x_4^{14}x_5^{15} $\cr  540.\, $x_1^{7}x_2x_3^{27}x_4^{15}x_5^{14}$& 541.\, $x_1^{7}x_2x_3^{30}x_4^{11}x_5^{15}$& 542.\, $x_1^{7}x_2x_3^{30}x_4^{15}x_5^{11} $\cr  543.\, $x_1^{7}x_2^{3}x_3^{9}x_4^{15}x_5^{30}$& 544.\, $x_1^{7}x_2^{3}x_3^{9}x_4^{30}x_5^{15}$& 545.\, $x_1^{7}x_2^{3}x_3^{11}x_4^{13}x_5^{30} $\cr  546.\, $x_1^{7}x_2^{3}x_3^{11}x_4^{15}x_5^{28}$& 547.\, $x_1^{7}x_2^{3}x_3^{11}x_4^{28}x_5^{15}$& 548.\, $x_1^{7}x_2^{3}x_3^{11}x_4^{29}x_5^{14} $\cr  549.\, $x_1^{7}x_2^{3}x_3^{13}x_4^{15}x_5^{26}$& 550.\, $x_1^{7}x_2^{3}x_3^{13}x_4^{26}x_5^{15}$& 551.\, $x_1^{7}x_2^{3}x_3^{15}x_4^{9}x_5^{30} $\cr  552.\, $x_1^{7}x_2^{3}x_3^{15}x_4^{11}x_5^{28}$& 553.\, $x_1^{7}x_2^{3}x_3^{15}x_4^{13}x_5^{26}$& 554.\, $x_1^{7}x_2^{3}x_3^{15}x_4^{25}x_5^{14} $\cr  555.\, $x_1^{7}x_2^{3}x_3^{15}x_4^{27}x_5^{12}$& 556.\, $x_1^{7}x_2^{3}x_3^{15}x_4^{29}x_5^{10}$& 557.\, $x_1^{7}x_2^{3}x_3^{25}x_4^{14}x_5^{15} $\cr
\end{tabular}}
\centerline{\begin{tabular}{lll}  
678.\, $x_1^{15}x_2^{3}x_3^{5}x_4^{11}x_5^{30}$& 679.\, $x_1^{15}x_2^{3}x_3^{5}x_4^{14}x_5^{27}$& 680.\, $x_1^{15}x_2^{3}x_3^{5}x_4^{15}x_5^{26} $\cr  681.\, $x_1^{15}x_2^{3}x_3^{5}x_4^{26}x_5^{15}$& 682.\, $x_1^{15}x_2^{3}x_3^{5}x_4^{27}x_5^{14}$& 683.\, $x_1^{15}x_2^{3}x_3^{5}x_4^{30}x_5^{11} $\cr  684.\, $x_1^{15}x_2^{3}x_3^{7}x_4^{9}x_5^{30}$& 685.\, $x_1^{15}x_2^{3}x_3^{7}x_4^{11}x_5^{28}$& 686.\, $x_1^{15}x_2^{3}x_3^{7}x_4^{13}x_5^{26} $\cr  687.\, $x_1^{15}x_2^{3}x_3^{7}x_4^{25}x_5^{14}$& 688.\, $x_1^{15}x_2^{3}x_3^{7}x_4^{27}x_5^{12}$& 689.\, $x_1^{15}x_2^{3}x_3^{7}x_4^{29}x_5^{10} $\cr  690.\, $x_1^{15}x_2^{3}x_3^{13}x_4^{3}x_5^{30}$& 691.\, $x_1^{15}x_2^{3}x_3^{13}x_4^{6}x_5^{27}$& 692.\, $x_1^{15}x_2^{3}x_3^{13}x_4^{7}x_5^{26} $\cr  693.\, $x_1^{15}x_2^{3}x_3^{13}x_4^{14}x_5^{19}$& 694.\, $x_1^{15}x_2^{3}x_3^{13}x_4^{15}x_5^{18}$& 695.\, $x_1^{15}x_2^{3}x_3^{13}x_4^{18}x_5^{15} $\cr  696.\, $x_1^{15}x_2^{3}x_3^{13}x_4^{22}x_5^{11}$& 697.\, $x_1^{15}x_2^{3}x_3^{13}x_4^{23}x_5^{10}$& 698.\, $x_1^{15}x_2^{3}x_3^{13}x_4^{30}x_5^{3} $\cr  699.\, $x_1^{15}x_2^{3}x_3^{15}x_4x_5^{30}$& 700.\, $x_1^{15}x_2^{3}x_3^{15}x_4^{3}x_5^{28}$& 701.\, $x_1^{15}x_2^{3}x_3^{15}x_4^{5}x_5^{26} $\cr
\end{tabular}}

\medskip
We have 
$$B_5((4)|^2|(3)|^2|(1) =  B(5,(3)|^2|(2)|^2)\cup B(4;(3)|^2|(2)|^2|(1))\cup B(2;(3)|^4|(1))\cup \mathcal C,$$
where $\mathcal C$ is the set of 170 monomials $g_j,\, 761\leqslant j \leqslant 930$, which are determined as follows:

\medskip
\centerline{\begin{tabular}{lll}
761.\, $x_1x_2^{15}x_3^{15}x_4^{19}x_5^{14}$& 762.\, $x_1x_2^{15}x_3^{23}x_4^{11}x_5^{14}$& 763.\, $x_1x_2^{15}x_3^{23}x_4^{14}x_5^{11} $\cr  764.\, $x_1^{3}x_2^{7}x_3^{13}x_4^{11}x_5^{30}$& 765.\, $x_1^{3}x_2^{7}x_3^{13}x_4^{14}x_5^{27}$& 766.\, $x_1^{3}x_2^{7}x_3^{13}x_4^{27}x_5^{14} $\cr  767.\, $x_1^{3}x_2^{7}x_3^{13}x_4^{30}x_5^{11}$& 768.\, $x_1^{3}x_2^{7}x_3^{29}x_4^{11}x_5^{14}$& 769.\, $x_1^{3}x_2^{7}x_3^{29}x_4^{14}x_5^{11} $\cr  770.\, $x_1^{3}x_2^{13}x_3^{7}x_4^{11}x_5^{30}$& 771.\, $x_1^{3}x_2^{13}x_3^{7}x_4^{14}x_5^{27}$& 772.\, $x_1^{3}x_2^{13}x_3^{7}x_4^{27}x_5^{14} $\cr  773.\, $x_1^{3}x_2^{13}x_3^{7}x_4^{30}x_5^{11}$& 774.\, $x_1^{3}x_2^{13}x_3^{14}x_4^{7}x_5^{27}$& 775.\, $x_1^{3}x_2^{13}x_3^{14}x_4^{23}x_5^{11} $\cr  776.\, $x_1^{3}x_2^{13}x_3^{15}x_4^{19}x_5^{14}$& 777.\, $x_1^{3}x_2^{13}x_3^{23}x_4^{11}x_5^{14}$& 778.\, $x_1^{3}x_2^{13}x_3^{23}x_4^{14}x_5^{11} $\cr  779.\, $x_1^{3}x_2^{13}x_3^{30}x_4^{7}x_5^{11}$& 780.\, $x_1^{3}x_2^{15}x_3^{13}x_4^{19}x_5^{14}$& 781.\, $x_1^{3}x_2^{15}x_3^{15}x_4^{15}x_5^{16} $\cr  782.\, $x_1^{3}x_2^{15}x_3^{15}x_4^{17}x_5^{14}$& 783.\, $x_1^{3}x_2^{15}x_3^{15}x_4^{19}x_5^{12}$& 784.\, $x_1^{3}x_2^{15}x_3^{21}x_4^{11}x_5^{14} $\cr
\end{tabular}}
\centerline{\begin{tabular}{lll}  
806.\, $x_1^{7}x_2^{7}x_3^{11}x_4^{27}x_5^{12}$& 807.\, $x_1^{7}x_2^{7}x_3^{11}x_4^{29}x_5^{10}$& 808.\, $x_1^{7}x_2^{7}x_3^{25}x_4^{11}x_5^{14} $\cr  809.\, $x_1^{7}x_2^{7}x_3^{25}x_4^{14}x_5^{11}$& 810.\, $x_1^{7}x_2^{7}x_3^{27}x_4^{9}x_5^{14}$& 811.\, $x_1^{7}x_2^{7}x_3^{27}x_4^{11}x_5^{12} $\cr  812.\, $x_1^{7}x_2^{7}x_3^{27}x_4^{13}x_5^{10}$& 813.\, $x_1^{7}x_2^{9}x_3^{7}x_4^{11}x_5^{30}$& 814.\, $x_1^{7}x_2^{9}x_3^{7}x_4^{27}x_5^{14} $\cr  815.\, $x_1^{7}x_2^{9}x_3^{23}x_4^{11}x_5^{14}$& 816.\, $x_1^{7}x_2^{11}x_3^{5}x_4^{11}x_5^{30}$& 817.\, $x_1^{7}x_2^{11}x_3^{5}x_4^{14}x_5^{27} $\cr  818.\, $x_1^{7}x_2^{11}x_3^{5}x_4^{27}x_5^{14}$& 819.\, $x_1^{7}x_2^{11}x_3^{5}x_4^{30}x_5^{11}$& 820.\, $x_1^{7}x_2^{11}x_3^{7}x_4^{9}x_5^{30} $\cr  821.\, $x_1^{7}x_2^{11}x_3^{7}x_4^{11}x_5^{28}$& 822.\, $x_1^{7}x_2^{11}x_3^{7}x_4^{13}x_5^{26}$& 823.\, $x_1^{7}x_2^{11}x_3^{7}x_4^{25}x_5^{14} $\cr  824.\, $x_1^{7}x_2^{11}x_3^{7}x_4^{27}x_5^{12}$& 825.\, $x_1^{7}x_2^{11}x_3^{7}x_4^{29}x_5^{10}$& 826.\, $x_1^{7}x_2^{11}x_3^{13}x_4^{6}x_5^{27} $\cr  827.\, $x_1^{7}x_2^{11}x_3^{13}x_4^{7}x_5^{26}$& 828.\, $x_1^{7}x_2^{11}x_3^{13}x_4^{14}x_5^{19}$& 829.\, $x_1^{7}x_2^{11}x_3^{13}x_4^{19}x_5^{14} $\cr
\end{tabular}}
\centerline{\begin{tabular}{lll}
\end{tabular}}
\subsubsection{Generators of $\widetilde {SF}_5((4)|^2|(3)|^2(1))$}\label{sss53}
\begin{align*}
\tilde\theta_{331} &= x_1x_2^{7}x_3^{11}x_4^{14}x_5^{31} + x_1x_2^{7}x_3^{14}x_4^{11}x_5^{31} + x_1^{3}x_2^{5}x_3^{11}x_4^{14}x_5^{31}\\ &\quad + x_1^{3}x_2^{5}x_3^{14}x_4^{11}x_5^{31} + x_1^{7}x_2x_3^{11}x_4^{14}x_5^{31} + x_1^{7}x_2x_3^{14}x_4^{11}x_5^{31}\\ &\quad + x_1^{7}x_2^{7}x_3^{8}x_4^{11}x_5^{31} + x_1^{7}x_2^{7}x_3^{9}x_4^{10}x_5^{31} + x_1^{7}x_2^{7}x_3^{11}x_4^{8}x_5^{31},\\  
\tilde\theta_{332} &=  x_1x_2^{7}x_3^{11}x_4^{31}x_5^{14} + x_1x_2^{7}x_3^{14}x_4^{31}x_5^{11} + x_1^{3}x_2^{5}x_3^{11}x_4^{31}x_5^{14}\\ &\quad + x_1^{3}x_2^{5}x_3^{14}x_4^{31}x_5^{11} + x_1^{7}x_2x_3^{11}x_4^{31}x_5^{14} + x_1^{7}x_2x_3^{14}x_4^{31}x_5^{11} \\ &\quad+ x_1^{7}x_2^{7}x_3^{8}x_4^{31}x_5^{11} + x_1^{7}x_2^{7}x_3^{9}x_4^{31}x_5^{10} + x_1^{7}x_2^{7}x_3^{11}x_4^{31}x_5^{8}, \\  
\tilde\theta_{333} &= 
x_1x_2^{7}x_3^{31}x_4^{11}x_5^{14} + x_1x_2^{7}x_3^{31}x_4^{14}x_5^{11} + x_1^{3}x_2^{5}x_3^{31}x_4^{11}x_5^{14}\\ &\quad + x_1^{3}x_2^{5}x_3^{31}x_4^{14}x_5^{11} + x_1^{7}x_2x_3^{31}x_4^{11}x_5^{14} + x_1^{7}x_2x_3^{31}x_4^{14}x_5^{11}\\ &\quad + x_1^{7}x_2^{7}x_3^{31}x_4^{8}x_5^{11} + x_1^{7}x_2^{7}x_3^{31}x_4^{9}x_5^{10} + x_1^{7}x_2^{7}x_3^{31}x_4^{11}x_5^{8}, \\  
\tilde\theta_{334} &= x_1x_2^{31}x_3^{7}x_4^{11}x_5^{14} + x_1x_2^{31}x_3^{7}x_4^{14}x_5^{11} + x_1^{3}x_2^{31}x_3^{5}x_4^{11}x_5^{14}\\ &\quad + x_1^{3}x_2^{31}x_3^{5}x_4^{14}x_5^{11} + x_1^{7}x_2^{31}x_3x_4^{11}x_5^{14} + x_1^{7}x_2^{31}x_3x_4^{14}x_5^{11}\\ &\quad + x_1^{7}x_2^{31}x_3^{7}x_4^{8}x_5^{11} + x_1^{7}x_2^{31}x_3^{7}x_4^{9}x_5^{10} + x_1^{7}x_2^{31}x_3^{7}x_4^{11}x_5^{8}, \\  
\tilde\theta_{335} &= x_1^{31}x_2x_3^{7}x_4^{11}x_5^{14} + x_1^{31}x_2x_3^{7}x_4^{14}x_5^{11} + x_1^{31}x_2^{3}x_3^{5}x_4^{11}x_5^{14}\\ &\quad + x_1^{31}x_2^{3}x_3^{5}x_4^{14}x_5^{11} + x_1^{31}x_2^{7}x_3x_4^{11}x_5^{14} + x_1^{31}x_2^{7}x_3x_4^{14}x_5^{11}\\ &\quad + x_1^{31}x_2^{7}x_3^{7}x_4^{8}x_5^{11} + x_1^{31}x_2^{7}x_3^{7}x_4^{9}x_5^{10} + x_1^{31}x_2^{7}x_3^{7}x_4^{11}x_5^{8},\\  
\tilde\theta_{916} &= x_1x_2^{7}x_3^{11}x_4^{15}x_5^{30} + x_1x_2^{7}x_3^{11}x_4^{15}x_5^{30} + x_1x_2^{7}x_3^{14}x_4^{15}x_5^{27}\\ &\quad + x_1^{3}x_2^{5}x_3^{11}x_4^{15}x_5^{30} + x_1^{3}x_2^{5}x_3^{14}x_4^{15}x_5^{27} + x_1^{7}x_2x_3^{11}x_4^{15}x_5^{30}\\ &\quad + x_1^{7}x_2x_3^{14}x_4^{15}x_5^{27} + x_1^{7}x_2^{7}x_3^{8}x_4^{15}x_5^{27}\\ &\quad + x_1^{7}x_2^{7}x_3^{9}x_4^{15}x_5^{26} + x_1^{7}x_2^{7}x_3^{11}x_4^{15}x_5^{24},\\  
\tilde\theta_{917} &= x_1x_2^{7}x_3^{15}x_4^{11}x_5^{30} + x_1x_2^{7}x_3^{15}x_4^{14}x_5^{27} + x_1^{3}x_2^{5}x_3^{15}x_4^{11}x_5^{30}\\ &\quad + x_1^{3}x_2^{5}x_3^{15}x_4^{14}x_5^{27} + x_1^{7}x_2x_3^{15}x_4^{11}x_5^{30} + x_1^{7}x_2x_3^{15}x_4^{14}x_5^{27}\\ &\quad + x_1^{7}x_2^{7}x_3^{15}x_4^{8}x_5^{27} + x_1^{7}x_2^{7}x_3^{15}x_4^{9}x_5^{26} + x_1^{7}x_2^{7}x_3^{15}x_4^{11}x_5^{24},\\  
\tilde\theta_{918} &= x_1x_2^{7}x_3^{15}x_4^{27}x_5^{14} + x_1x_2^{7}x_3^{15}x_4^{30}x_5^{11} + x_1^{3}x_2^{5}x_3^{15}x_4^{27}x_5^{14}\\ &\quad + x_1^{3}x_2^{5}x_3^{15}x_4^{30}x_5^{11} + x_1^{7}x_2x_3^{15}x_4^{27}x_5^{14} + x_1^{7}x_2x_3^{15}x_4^{30}x_5^{11}\\ &\quad + x_1^{7}x_2^{7}x_3^{15}x_4^{24}x_5^{11} + x_1^{7}x_2^{7}x_3^{15}x_4^{25}x_5^{10} + x_1^{7}x_2^{7}x_3^{15}x_4^{27}x_5^{8},\\  
\tilde\theta_{919} &= x_1^{3}x_2^{3}x_3^{15}x_4^{29}x_5^{14} + x_1^{3}x_2^{13}x_3^{15}x_4^{19}x_5^{14} + x_1^{7}x_2^{3}x_3^{15}x_4^{25}x_5^{14} + x_1^{7}x_2^{9}x_3^{15}x_4^{19}x_5^{14},\\  
\tilde\theta_{920} &= x_1x_2^{15}x_3^{7}x_4^{11}x_5^{30} + x_1x_2^{15}x_3^{7}x_4^{14}x_5^{27} + x_1^{3}x_2^{15}x_3^{5}x_4^{11}x_5^{30}\\ &\quad + x_1^{3}x_2^{15}x_3^{5}x_4^{14}x_5^{27} + x_1^{7}x_2^{15}x_3x_4^{11}x_5^{30} + x_1^{7}x_2^{15}x_3x_4^{14}x_5^{27}\\ &\quad + x_1^{7}x_2^{15}x_3^{7}x_4^{8}x_5^{27} + x_1^{7}x_2^{15}x_3^{7}x_4^{9}x_5^{26} + x_1^{7}x_2^{15}x_3^{7}x_4^{11}x_5^{24},\\  
\tilde\theta_{921} &= x_1x_2^{15}x_3^{7}x_4^{27}x_5^{14} + x_1x_2^{15}x_3^{7}x_4^{30}x_5^{11} + x_1^{3}x_2^{15}x_3^{5}x_4^{27}x_5^{14}\\ &\quad + x_1^{3}x_2^{15}x_3^{5}x_4^{30}x_5^{11} + x_1^{7}x_2^{15}x_3x_4^{27}x_5^{14} + x_1^{7}x_2^{15}x_3x_4^{30}x_5^{11}\\ &\quad + x_1^{7}x_2^{15}x_3^{7}x_4^{24}x_5^{11} + x_1^{7}x_2^{15}x_3^{7}x_4^{25}x_5^{10} + x_1^{7}x_2^{15}x_3^{7}x_4^{27}x_5^{8},\\  
\tilde\theta_{922} &= x_1^{3}x_2^{15}x_3^{3}x_4^{29}x_5^{14} + x_1^{3}x_2^{15}x_3^{13}x_4^{19}x_5^{14} + x_1^{7}x_2^{15}x_3^{3}x_4^{25}x_5^{14} + x_1^{7}x_2^{15}x_3^{9}x_4^{19}x_5^{14},\\  
\tilde\theta_{923} &= x_1x_2^{15}x_3^{15}x_4^{3}x_5^{30} + x_1x_2^{15}x_3^{15}x_4^{7}x_5^{26} + x_1x_2^{15}x_3^{15}x_4^{19}x_5^{14} + x_1x_2^{15}x_3^{15}x_4^{22}x_5^{11}\\ & + x_1x_2^{15}x_3^{23}x_4^{11}x_5^{14} + x_1x_2^{15}x_3^{23}x_4^{14}x_5^{11} + x_1^{3}x_2^{5}x_3^{15}x_4^{27}x_5^{14} + x_1^{3}x_2^{5}x_3^{15}x_4^{30}x_5^{11}\\ & + x_1^{3}x_2^{7}x_3^{15}x_4^{11}x_5^{28} + x_1^{3}x_2^{7}x_3^{15}x_4^{13}x_5^{26} + x_1^{3}x_2^{7}x_3^{15}x_4^{25}x_5^{14} + x_1^{3}x_2^{7}x_3^{15}x_4^{27}x_5^{12}\\ & + x_1^{3}x_2^{7}x_3^{29}x_4^{11}x_5^{14} + x_1^{3}x_2^{7}x_3^{29}x_4^{14}x_5^{11} + x_1^{3}x_2^{15}x_3^{15}x_4^{3}x_5^{28} + x_1^{3}x_2^{15}x_3^{15}x_4^{13}x_5^{18}\\ & + x_1^{3}x_2^{15}x_3^{15}x_4^{17}x_5^{14} + x_1^{3}x_2^{15}x_3^{15}x_4^{19}x_5^{12} + x_1^{7}x_2x_3^{15}x_4^{27}x_5^{14} + x_1^{7}x_2x_3^{15}x_4^{30}x_5^{11}\\ & + x_1^{7}x_2^{7}x_3^{15}x_4^{11}x_5^{24} + x_1^{7}x_2^{7}x_3^{15}x_4^{25}x_5^{10} + x_1^{7}x_2^{7}x_3^{15}x_4^{27}x_5^{8} + x_1^{7}x_2^{7}x_3^{25}x_4^{11}x_5^{14}\\ & + x_1^{7}x_2^{7}x_3^{25}x_4^{14}x_5^{11} + x_1^{7}x_2^{15}x_3^{15}x_4^{11}x_5^{16} + x_1^{7}x_2^{15}x_3^{15}x_4^{16}x_5^{11} + x_1^{7}x_2^{15}x_3^{23}x_4^{8}x_5^{11}\\ & + x_1^{7}x_2^{15}x_3^{23}x_4^{9}x_5^{10} + x_1^{7}x_2^{15}x_3^{23}x_4^{11}x_5^{8},\\  
\tilde\theta_{924} &= x_1^{15}x_2x_3^{7}x_4^{11}x_5^{30} + x_1^{15}x_2x_3^{7}x_4^{14}x_5^{27} + x_1^{15}x_2^{3}x_3^{5}x_4^{11}x_5^{30}\\ &\quad + x_1^{15}x_2^{3}x_3^{5}x_4^{14}x_5^{27} + x_1^{15}x_2^{7}x_3x_4^{11}x_5^{30} + x_1^{15}x_2^{7}x_3x_4^{14}x_5^{27}\\ &\quad + x_1^{15}x_2^{7}x_3^{7}x_4^{8}x_5^{27} + x_1^{15}x_2^{7}x_3^{7}x_4^{9}x_5^{26} + x_1^{15}x_2^{7}x_3^{7}x_4^{11}x_5^{24},\\  
\tilde\theta_{925} &= x_1^{15}x_2x_3^{7}x_4^{27}x_5^{14} + x_1^{15}x_2x_3^{7}x_4^{30}x_5^{11} + x_1^{15}x_2^{3}x_3^{5}x_4^{27}x_5^{14}\\ &\quad + x_1^{15}x_2^{3}x_3^{5}x_4^{30}x_5^{11} + x_1^{15}x_2^{7}x_3x_4^{27}x_5^{14} + x_1^{15}x_2^{7}x_3x_4^{30}x_5^{11}\\ &\quad + x_1^{15}x_2^{7}x_3^{7}x_4^{24}x_5^{11} + x_1^{15}x_2^{7}x_3^{7}x_4^{25}x_5^{10} + x_1^{15}x_2^{7}x_3^{7}x_4^{27}x_5^{8},\\  
\tilde\theta_{926} &= x_1^{15}x_2^{3}x_3^{3}x_4^{29}x_5^{14} + x_1^{15}x_2^{3}x_3^{13}x_4^{19}x_5^{14} + x_1^{15}x_2^{7}x_3^{3}x_4^{25}x_5^{14} + x_1^{15}x_2^{7}x_3^{9}x_4^{19}x_5^{14},\\  
\tilde\theta_{927} &= x_1^{7}x_2x_3^{15}x_4^{27}x_5^{14} + x_1^{7}x_2x_3^{15}x_4^{30}x_5^{11} + x_1^{7}x_2^{3}x_3^{15}x_4^{11}x_5^{28} + x_1^{7}x_2^{3}x_3^{15}x_4^{13}x_5^{26}\\ & + x_1^{7}x_2^{3}x_3^{15}x_4^{25}x_5^{14} + x_1^{7}x_2^{3}x_3^{15}x_4^{27}x_5^{12} + x_1^{7}x_2^{3}x_3^{29}x_4^{11}x_5^{14} + x_1^{7}x_2^{3}x_3^{29}x_4^{14}x_5^{11}\\ & + x_1^{7}x_2^{7}x_3^{15}x_4^{11}x_5^{24} + x_1^{7}x_2^{7}x_3^{15}x_4^{24}x_5^{11} + x_1^{7}x_2^{7}x_3^{25}x_4^{11}x_5^{14} + x_1^{7}x_2^{7}x_3^{25}x_4^{14}x_5^{11}\\ & + x_1^{15}x_2x_3^{15}x_4^{3}x_5^{30} + x_1^{15}x_2x_3^{15}x_4^{7}x_5^{26} + x_1^{15}x_2x_3^{15}x_4^{19}x_5^{14} + x_1^{15}x_2x_3^{15}x_4^{22}x_5^{11}\\ & + x_1^{15}x_2x_3^{23}x_4^{11}x_5^{14} + x_1^{15}x_2x_3^{23}x_4^{14}x_5^{11} + x_1^{15}x_2^{3}x_3^{15}x_4^{3}x_5^{28} + x_1^{15}x_2^{3}x_3^{15}x_4^{13}x_5^{18}\\ & + x_1^{15}x_2^{3}x_3^{15}x_4^{17}x_5^{14} + x_1^{15}x_2^{3}x_3^{15}x_4^{19}x_5^{12} + x_1^{15}x_2^{7}x_3^{15}x_4^{11}x_5^{16} + x_1^{15}x_2^{7}x_3^{15}x_4^{16}x_5^{11}\\ & + x_1^{15}x_2^{7}x_3^{23}x_4^{8}x_5^{11} + x_1^{15}x_2^{7}x_3^{23}x_4^{9}x_5^{10} + x_1^{15}x_2^{7}x_3^{23}x_4^{11}x_5^{8},\\  
\tilde\theta_{928} &=  x_1x_2^{15}x_3^{7}x_4^{27}x_5^{14} + x_1x_2^{15}x_3^{7}x_4^{30}x_5^{11} + x_1x_2^{15}x_3^{15}x_4^{19}x_5^{14} + x_1x_2^{15}x_3^{15}x_4^{30}x_5^{3}\\ & + x_1x_2^{15}x_3^{23}x_4^{11}x_5^{14} + x_1x_2^{15}x_3^{30}x_4^{7}x_5^{11} + x_1^{3}x_2^{7}x_3^{11}x_4^{13}x_5^{30} + x_1^{3}x_2^{7}x_3^{11}x_4^{29}x_5^{14}\\ & + x_1^{3}x_2^{7}x_3^{13}x_4^{11}x_5^{30} + x_1^{3}x_2^{7}x_3^{13}x_4^{30}x_5^{11} + x_1^{3}x_2^{7}x_3^{29}x_4^{11}x_5^{14} + x_1^{3}x_2^{7}x_3^{29}x_4^{14}x_5^{11}\\ & + x_1^{3}x_2^{13}x_3^{7}x_4^{11}x_5^{30} + x_1^{3}x_2^{13}x_3^{7}x_4^{30}x_5^{11} + x_1^{3}x_2^{13}x_3^{15}x_4^{19}x_5^{14} + x_1^{3}x_2^{13}x_3^{15}x_4^{30}x_5^{3}\\ & + x_1^{3}x_2^{13}x_3^{30}x_4^{7}x_5^{11} + x_1^{3}x_2^{15}x_3^{3}x_4^{13}x_5^{30} + x_1^{3}x_2^{15}x_3^{3}x_4^{29}x_5^{14} + x_1^{3}x_2^{15}x_3^{7}x_4^{11}x_5^{28}\\ & + x_1^{3}x_2^{15}x_3^{7}x_4^{13}x_5^{26} + x_1^{3}x_2^{15}x_3^{7}x_4^{25}x_5^{14} + x_1^{3}x_2^{15}x_3^{7}x_4^{27}x_5^{12} + x_1^{3}x_2^{15}x_3^{13}x_4^{3}x_5^{30}\\ & + x_1^{3}x_2^{15}x_3^{13}x_4^{7}x_5^{26} + x_1^{3}x_2^{15}x_3^{13}x_4^{22}x_5^{11} + x_1^{3}x_2^{15}x_3^{13}x_4^{30}x_5^{3} + x_1^{3}x_2^{15}x_3^{15}x_4^{3}x_5^{28}\\ & + x_1^{3}x_2^{15}x_3^{15}x_4^{13}x_5^{18} + x_1^{3}x_2^{15}x_3^{15}x_4^{17}x_5^{14} + x_1^{3}x_2^{15}x_3^{15}x_4^{19}x_5^{12} + x_1^{3}x_2^{15}x_3^{21}x_4^{11}x_5^{14}\\ & + x_1^{3}x_2^{15}x_3^{21}x_4^{14}x_5^{11} + x_1^{3}x_2^{15}x_3^{23}x_4^{11}x_5^{12} + x_1^{7}x_2^{3}x_3^{11}x_4^{13}x_5^{30} + x_1^{7}x_2^{3}x_3^{11}x_4^{29}x_5^{14}\\ & + x_1^{7}x_2^{3}x_3^{27}x_4^{13}x_5^{14} + x_1^{7}x_2^{7}x_3^{9}x_4^{11}x_5^{30} + x_1^{7}x_2^{7}x_3^{9}x_4^{30}x_5^{11} + x_1^{7}x_2^{7}x_3^{11}x_4^{9}x_5^{30}\\ & + x_1^{7}x_2^{7}x_3^{11}x_4^{11}x_5^{28} + x_1^{7}x_2^{7}x_3^{11}x_4^{13}x_5^{26} + x_1^{7}x_2^{7}x_3^{11}x_4^{29}x_5^{10} + x_1^{7}x_2^{7}x_3^{15}x_4^{9}x_5^{26}\\ & + x_1^{7}x_2^{7}x_3^{25}x_4^{14}x_5^{11} + x_1^{7}x_2^{7}x_3^{27}x_4^{11}x_5^{12} + x_1^{7}x_2^{9}x_3^{7}x_4^{11}x_5^{30} + x_1^{7}x_2^{9}x_3^{7}x_4^{27}x_5^{14}\\ & + x_1^{7}x_2^{9}x_3^{23}x_4^{11}x_5^{14} + x_1^{7}x_2^{11}x_3^{3}x_4^{13}x_5^{30} + x_1^{7}x_2^{11}x_3^{3}x_4^{29}x_5^{14} + x_1^{7}x_2^{11}x_3^{7}x_4^{11}x_5^{28}\\ & + x_1^{7}x_2^{11}x_3^{7}x_4^{13}x_5^{26} + x_1^{7}x_2^{11}x_3^{7}x_4^{25}x_5^{14} + x_1^{7}x_2^{11}x_3^{7}x_4^{27}x_5^{12} + x_1^{7}x_2^{11}x_3^{13}x_4^{3}x_5^{30}\\ & + x_1^{7}x_2^{11}x_3^{13}x_4^{7}x_5^{26} + x_1^{7}x_2^{11}x_3^{13}x_4^{22}x_5^{11} + x_1^{7}x_2^{11}x_3^{13}x_4^{30}x_5^{3} + x_1^{7}x_2^{11}x_3^{15}x_4^{3}x_5^{28}\\ & + x_1^{7}x_2^{11}x_3^{15}x_4^{13}x_5^{18} + x_1^{7}x_2^{11}x_3^{15}x_4^{17}x_5^{14} + x_1^{7}x_2^{11}x_3^{15}x_4^{19}x_5^{12} + x_1^{7}x_2^{11}x_3^{21}x_4^{11}x_5^{14}\\ & + x_1^{7}x_2^{11}x_3^{21}x_4^{14}x_5^{11} + x_1^{7}x_2^{11}x_3^{23}x_4^{11}x_5^{12} + x_1^{15}x_2x_3^{7}x_4^{27}x_5^{14} + x_1^{15}x_2x_3^{7}x_4^{30}x_5^{11}\\ & + x_1^{15}x_2x_3^{15}x_4^{19}x_5^{14} + x_1^{15}x_2x_3^{15}x_4^{30}x_5^{3} + x_1^{15}x_2x_3^{23}x_4^{11}x_5^{14} + x_1^{15}x_2x_3^{30}x_4^{7}x_5^{11}\\ & + x_1^{15}x_2^{3}x_3^{3}x_4^{13}x_5^{30} + x_1^{15}x_2^{3}x_3^{3}x_4^{29}x_5^{14} + x_1^{15}x_2^{3}x_3^{7}x_4^{11}x_5^{28} + x_1^{15}x_2^{3}x_3^{7}x_4^{13}x_5^{26}\\ & + x_1^{15}x_2^{3}x_3^{7}x_4^{25}x_5^{14} + x_1^{15}x_2^{3}x_3^{7}x_4^{27}x_5^{12} + x_1^{15}x_2^{3}x_3^{13}x_4^{3}x_5^{30} + x_1^{15}x_2^{3}x_3^{13}x_4^{7}x_5^{26}\\ & + x_1^{15}x_2^{3}x_3^{13}x_4^{22}x_5^{11} + x_1^{15}x_2^{3}x_3^{13}x_4^{30}x_5^{3} + x_1^{15}x_2^{3}x_3^{15}x_4^{3}x_5^{28} + x_1^{15}x_2^{3}x_3^{15}x_4^{13}x_5^{18}\\ & + x_1^{15}x_2^{3}x_3^{15}x_4^{17}x_5^{14} + x_1^{15}x_2^{3}x_3^{15}x_4^{19}x_5^{12} + x_1^{15}x_2^{3}x_3^{21}x_4^{11}x_5^{14} + x_1^{15}x_2^{3}x_3^{21}x_4^{14}x_5^{11}\\ & + x_1^{15}x_2^{3}x_3^{23}x_4^{11}x_5^{12} + x_1^{15}x_2^{15}x_3x_4^{19}x_5^{14} + x_1^{15}x_2^{15}x_3x_4^{30}x_5^{3} + x_1^{15}x_2^{15}x_3^{3}x_4^{3}x_5^{28}\\ & + x_1^{15}x_2^{15}x_3^{3}x_4^{19}x_5^{12} + x_1^{15}x_2^{15}x_3^{7}x_4^{11}x_5^{16} + x_1^{15}x_2^{15}x_3^{7}x_4^{16}x_5^{11} + x_1^{15}x_2^{15}x_3^{15}x_4^{3}x_5^{16}\\ & + x_1^{15}x_2^{15}x_3^{15}x_4^{16}x_5^{3} + x_1^{15}x_2^{15}x_3^{16}x_4^{7}x_5^{11},\\  
\tilde\theta_{929} &= x_1x_2^{15}x_3^{15}x_4^{19}x_5^{14} + x_1x_2^{15}x_3^{15}x_4^{30}x_5^{3} + x_1^{3}x_2^{3}x_3^{15}x_4^{13}x_5^{30} + x_1^{3}x_2^{7}x_3^{11}x_4^{13}x_5^{30}\\ & + x_1^{3}x_2^{7}x_3^{11}x_4^{29}x_5^{14} + x_1^{3}x_2^{7}x_3^{13}x_4^{11}x_5^{30} + x_1^{3}x_2^{7}x_3^{13}x_4^{27}x_5^{14} + x_1^{3}x_2^{13}x_3^{7}x_4^{11}x_5^{30}\\ & + x_1^{3}x_2^{13}x_3^{7}x_4^{27}x_5^{14} + x_1^{3}x_2^{13}x_3^{15}x_4^{30}x_5^{3} + x_1^{3}x_2^{15}x_3^{3}x_4^{13}x_5^{30} + x_1^{3}x_2^{15}x_3^{13}x_4^{30}x_5^{3}\\ & + x_1^{3}x_2^{15}x_3^{15}x_4^{3}x_5^{28} + x_1^{3}x_2^{15}x_3^{15}x_4^{19}x_5^{12} + x_1^{7}x_2^{3}x_3^{11}x_4^{13}x_5^{30} + x_1^{7}x_2^{3}x_3^{11}x_4^{29}x_5^{14}\\ & + x_1^{7}x_2^{3}x_3^{15}x_4^{25}x_5^{14} + x_1^{7}x_2^{7}x_3^{9}x_4^{11}x_5^{30} + x_1^{7}x_2^{7}x_3^{9}x_4^{27}x_5^{14} + x_1^{7}x_2^{7}x_3^{11}x_4^{9}x_5^{30}\\ & + x_1^{7}x_2^{7}x_3^{11}x_4^{13}x_5^{26} + x_1^{7}x_2^{7}x_3^{11}x_4^{25}x_5^{14} + x_1^{7}x_2^{7}x_3^{11}x_4^{29}x_5^{10} + x_1^{7}x_2^{7}x_3^{15}x_4^{9}x_5^{26}\\ & + x_1^{7}x_2^{7}x_3^{27}x_4^{11}x_5^{12} + x_1^{7}x_2^{9}x_3^{7}x_4^{11}x_5^{30} + x_1^{7}x_2^{9}x_3^{7}x_4^{27}x_5^{14} + x_1^{7}x_2^{9}x_3^{15}x_4^{19}x_5^{14}\\ & + x_1^{7}x_2^{11}x_3^{3}x_4^{13}x_5^{30} + x_1^{7}x_2^{11}x_3^{3}x_4^{29}x_5^{14} + x_1^{7}x_2^{11}x_3^{7}x_4^{11}x_5^{28} + x_1^{7}x_2^{11}x_3^{7}x_4^{13}x_5^{26}\\ & + x_1^{7}x_2^{11}x_3^{13}x_4^{3}x_5^{30} + x_1^{7}x_2^{11}x_3^{13}x_4^{7}x_5^{26} + x_1^{7}x_2^{11}x_3^{13}x_4^{19}x_5^{14} + x_1^{7}x_2^{11}x_3^{13}x_4^{30}x_5^{3}\\ & + x_1^{7}x_2^{11}x_3^{15}x_4^{3}x_5^{28} + x_1^{7}x_2^{11}x_3^{15}x_4^{19}x_5^{12} + x_1^{7}x_2^{15}x_3^{3}x_4^{25}x_5^{14} + x_1^{7}x_2^{15}x_3^{7}x_4^{9}x_5^{26}\\ & + x_1^{7}x_2^{15}x_3^{9}x_4^{19}x_5^{14} + x_1^{7}x_2^{15}x_3^{11}x_4^{3}x_5^{28} + x_1^{7}x_2^{15}x_3^{11}x_4^{19}x_5^{12} + x_1^{7}x_2^{27}x_3^{7}x_4^{11}x_5^{12}\\ & + x_1^{15}x_2x_3^{7}x_4^{11}x_5^{30} + x_1^{15}x_2x_3^{7}x_4^{14}x_5^{27} + x_1^{15}x_2x_3^{15}x_4^{19}x_5^{14} + x_1^{15}x_2x_3^{15}x_4^{30}x_5^{3}\\ & + x_1^{15}x_2^{3}x_3^{3}x_4^{13}x_5^{30} + x_1^{15}x_2^{3}x_3^{5}x_4^{11}x_5^{30} + x_1^{15}x_2^{3}x_3^{5}x_4^{14}x_5^{27} + x_1^{15}x_2^{3}x_3^{13}x_4^{30}x_5^{3}\\ & + x_1^{15}x_2^{3}x_3^{15}x_4^{3}x_5^{28} + x_1^{15}x_2^{3}x_3^{15}x_4^{19}x_5^{12} + x_1^{15}x_2^{7}x_3x_4^{11}x_5^{30} + x_1^{15}x_2^{7}x_3x_4^{14}x_5^{27}\\ & + x_1^{15}x_2^{7}x_3^{3}x_4^{25}x_5^{14} + x_1^{15}x_2^{7}x_3^{7}x_4^{8}x_5^{27} + x_1^{15}x_2^{7}x_3^{9}x_4^{19}x_5^{14} + x_1^{15}x_2^{7}x_3^{11}x_4^{3}x_5^{28}\\ & + x_1^{15}x_2^{7}x_3^{11}x_4^{19}x_5^{12} + x_1^{15}x_2^{7}x_3^{19}x_4^{11}x_5^{12} + x_1^{15}x_2^{15}x_3x_4^{19}x_5^{14} + x_1^{15}x_2^{15}x_3x_4^{30}x_5^{3}\\ & + x_1^{15}x_2^{15}x_3^{3}x_4^{3}x_5^{28} + x_1^{15}x_2^{15}x_3^{3}x_4^{19}x_5^{12} + x_1^{15}x_2^{15}x_3^{15}x_4^{3}x_5^{16} + x_1^{15}x_2^{15}x_3^{15}x_4^{16}x_5^{3}\\ & + x_1^{15}x_2^{19}x_3^{7}x_4^{11}x_5^{12},\\  
\tilde\theta_{930} &= x_1^{3}x_2^{15}x_3^{7}x_4^{11}x_5^{28} + x_1^{3}x_2^{15}x_3^{7}x_4^{13}x_5^{26} + x_1^{3}x_2^{15}x_3^{7}x_4^{25}x_5^{14} + x_1^{3}x_2^{15}x_3^{7}x_4^{27}x_5^{12}\\ & + x_1^{3}x_2^{29}x_3^{7}x_4^{11}x_5^{14} + x_1^{3}x_2^{29}x_3^{7}x_4^{14}x_5^{11} + x_1^{7}x_2^{15}x_3^{3}x_4^{11}x_5^{28} + x_1^{7}x_2^{15}x_3^{3}x_4^{13}x_5^{26}\\ & + x_1^{7}x_2^{15}x_3^{3}x_4^{25}x_5^{14} + x_1^{7}x_2^{15}x_3^{3}x_4^{27}x_5^{12} + x_1^{7}x_2^{15}x_3^{7}x_4^{24}x_5^{11} + x_1^{7}x_2^{15}x_3^{7}x_4^{25}x_5^{10}\\ & + x_1^{7}x_2^{15}x_3^{7}x_4^{27}x_5^{8} + x_1^{7}x_2^{15}x_3^{23}x_4^{8}x_5^{11} + x_1^{7}x_2^{15}x_3^{23}x_4^{9}x_5^{10} + x_1^{7}x_2^{15}x_3^{23}x_4^{11}x_5^{8}\\ & + x_1^{7}x_2^{27}x_3^{5}x_4^{11}x_5^{14} + x_1^{7}x_2^{27}x_3^{5}x_4^{14}x_5^{11} + x_1^{15}x_2^{15}x_3x_4^{3}x_5^{30} + x_1^{15}x_2^{15}x_3x_4^{7}x_5^{26}\\ & + x_1^{15}x_2^{15}x_3x_4^{19}x_5^{14} + x_1^{15}x_2^{15}x_3x_4^{22}x_5^{11} + x_1^{15}x_2^{15}x_3^{3}x_4^{3}x_5^{28} + x_1^{15}x_2^{15}x_3^{3}x_4^{13}x_5^{18}\\ & + x_1^{15}x_2^{15}x_3^{3}x_4^{17}x_5^{14} + x_1^{15}x_2^{15}x_3^{3}x_4^{19}x_5^{12} + x_1^{15}x_2^{15}x_3^{7}x_4^{11}x_5^{16} + x_1^{15}x_2^{15}x_3^{7}x_4^{16}x_5^{11}\\ & + x_1^{15}x_2^{23}x_3x_4^{11}x_5^{14} + x_1^{15}x_2^{23}x_3x_4^{14}x_5^{11} + x_1^{15}x_2^{23}x_3^{7}x_4^{8}x_5^{11} + x_1^{15}x_2^{23}x_3^{7}x_4^{9}x_5^{10}\\ & + x_1^{15}x_2^{23}x_3^{7}x_4^{11}x_5^{8}.
\end{align*}

\bigskip
\subsection{Computation of $(QP_5)_{128}$}\

\medskip
By Lemma \ref{bdd5}, we have
$$ (QP_5)_{128} \cong QP_5((2)|(1)|^6)\bigoplus QP_5((4)|(2)|^5) \bigoplus QP_5((4)|^2(3)|^3|(1)).$$
By Theorems \ref{mdd51} and \ref{mdd52}, we have
$$\dim QP_5((2)|(1)|^4)=124,\ \dim QP_5((4)|(2)|^4) = 465.$$
Hence, we need to compute $QP_5((4)|^2(3)|^3|(1))$.

\subsubsection{Computation of $QP_5((4)|^2(3)|^3|(1))$}\

\medskip
By using a result in \cite{su2}, we have $B_4((4)|^2(3)|^3|(1)) = B_4^+((4)|^2(3)|^3|(1))$ and $|B_4^+((4)|^2(3)|^3|(1))| = 45$. This implies $\dim QP_5^0((4)|^2(3)|^2|(1)) = 5 \times 45 = 225.$ Hence, we need only to compute $QP_5^+((4)|^2(3)|^2|(1))$.
We prove the following result.

\begin{props}\label{mddb61}\
	
\medskip
{\rm i)} There exist exactly $1395$	classes in $\widetilde{QP}_5((4)|^2(3)|^3|(1))$ which are presented by 
admissible monomials of degree $128$ in $P_5$. 

{\rm ii)} $\dim \widetilde{SF}_5((4)|^2|(3)|^3|(1))\leqslant 6$.    Consequently, 
	$$1395 \leqslant \dim\widetilde{QP}_5((4)|^2|(3)|^{3}|(1)) \leqslant 1401.$$
\end{props}

\begin{lems}\label{bdab61} The following monomials are strictly inadmissible:
	
\medskip
\centerline{
\begin{tabular}{llll}
$x_1^{3}x_2^{31}x_3^{31}x_4^{31}x_5^{32}$& $x_1^{7}x_2^{15}x_3^{23}x_4^{27}x_5^{56}$& $x_1^{7}x_2^{15}x_3^{23}x_4^{59}x_5^{24}$& $x_1^{7}x_2^{15}x_3^{55}x_4^{27}x_5^{24} $\cr  $x_1^{7}x_2^{27}x_3^{31}x_4^{31}x_5^{32}$& $x_1^{7}x_2^{31}x_3^{27}x_4^{31}x_5^{32}$& $x_1^{7}x_2^{31}x_3^{31}x_4^{27}x_5^{32}$& $x_1^{7}x_2^{31}x_3^{31}x_4^{32}x_5^{27} $\cr  $x_1^{7}x_2^{31}x_3^{31}x_4^{33}x_5^{26}$& $x_1^{15}x_2^{7}x_3^{23}x_4^{27}x_5^{56}$& $x_1^{15}x_2^{7}x_3^{23}x_4^{59}x_5^{24}$& $x_1^{15}x_2^{7}x_3^{55}x_4^{27}x_5^{24} $\cr  $x_1^{15}x_2^{15}x_3^{19}x_4^{19}x_5^{60}$& $x_1^{15}x_2^{15}x_3^{23}x_4^{16}x_5^{59}$& $x_1^{15}x_2^{15}x_3^{23}x_4^{17}x_5^{58}$& $x_1^{15}x_2^{15}x_3^{23}x_4^{27}x_5^{48} $\cr  $x_1^{15}x_2^{15}x_3^{23}x_4^{48}x_5^{27}$& $x_1^{15}x_2^{15}x_3^{23}x_4^{49}x_5^{26}$& $x_1^{15}x_2^{15}x_3^{23}x_4^{59}x_5^{16}$& $x_1^{15}x_2^{15}x_3^{48}x_4^{23}x_5^{27} $\cr  $x_1^{15}x_2^{15}x_3^{55}x_4^{16}x_5^{27}$& $x_1^{15}x_2^{15}x_3^{55}x_4^{17}x_5^{26}$& $x_1^{15}x_2^{15}x_3^{55}x_4^{27}x_5^{16}$& $x_1^{15}x_2^{23}x_3^{7}x_4^{27}x_5^{56} $\cr  $x_1^{15}x_2^{23}x_3^{7}x_4^{59}x_5^{24}$& $x_1^{15}x_2^{23}x_3^{15}x_4^{16}x_5^{59}$& $x_1^{15}x_2^{23}x_3^{15}x_4^{17}x_5^{58}$& $x_1^{15}x_2^{23}x_3^{15}x_4^{48}x_5^{27} $\cr  $x_1^{15}x_2^{23}x_3^{15}x_4^{49}x_5^{26}$& $x_1^{15}x_2^{23}x_3^{27}x_4^{31}x_5^{32}$& $x_1^{15}x_2^{23}x_3^{31}x_4^{27}x_5^{32}$& $x_1^{15}x_2^{23}x_3^{31}x_4^{32}x_5^{27} $\cr
\end{tabular}}
\end{lems}
 \begin{proof}
 We prove the lemma for 
 
\medskip
\centerline{\begin{tabular}{lll}
$\tilde w_{1} =  x_1^{3}x_2^{31}x_3^{31}x_4^{31}x_5^{32}$& $\tilde w_{2} =  x_1^{7}x_2^{15}x_3^{23}x_4^{59}x_5^{24}$& $\tilde w_{3} =  x_1^{7}x_2^{31}x_3^{31}x_4^{33}x_5^{26} $\cr  $\tilde w_{4} =  x_1^{15}x_2^{15}x_3^{19}x_4^{19}x_5^{60}$& $\tilde w_{5} =  x_1^{15}x_2^{15}x_3^{23}x_4^{17}x_5^{58}$& $\tilde w_{6} =  x_1^{15}x_2^{15}x_3^{23}x_4^{27}x_5^{48} $\cr  $\tilde w_{7} =  x_1^{15}x_2^{23}x_3^{7}x_4^{27}x_5^{56}$& $\tilde w_{8} =  x_1^{31}x_2^{39}x_3^{11}x_4^{21}x_5^{26}.$&  \cr      
\end{tabular}}

\medskip
The others can be proved by a similar computation. We have
%
Hence, the monomials $\tilde w_t,\, 1 \leqslant t \leqslant 8$, are strictly inadmissible.
\end{proof}
\begin{proof}[Proof of Proposition $\ref{mddb61}$] Let $x \in P_5^+(\tilde \omega)$ be an admissible monomial with $\tilde\omega := (4)|^2|(3)|^{3}|(1)$. Then $x = X_ry^2$ with $1 \leqslant r \leqslant 5$ and $y$ an admissible monomial of weight vector $(4)|(3)|^{3}|(1)$. By a direct computation we see that if $x \ne \tilde g_j,\, 1\leqslant j\leqslant  1176$, then either $x$ is one of monomials as given in Lemma \ref{bdab61} or there is a monomial $w$ as given in one of Lemmas \ref{bda72}, \ref{bda73} and \ref{bda74} such that $x = wz_1^{2^u}$ with $u$ nonnegative integers, $2\leqslant u \leqslant 5$, and $z_1$ a monomial of weight vector $(3)|^{6-u}$. By Theorem \ref{dlcb1}, $x$ is inadmissible. This contradicts the fact that $x$ is admissible. Hence, $x = \tilde g_j$ for some $j,\ 1\leqslant j\leqslant  1176$.
	
Now we prove that the set $\{[\tilde g_j]_{\tilde\omega}:1\leqslant j\leqslant  1176\}$ is linearly independent in $QP_5(\tilde\omega)$. 
	
Suppose there is a linear relation
\begin{equation}\label{ctda61}
\tilde{\mathcal S}:= \sum_{1\leqslant j \leqslant 1176}\gamma_j\tilde g_j \equiv_{\tilde\omega} 0,
\end{equation}
where $\gamma_j \in \mathbb F_2$.
Let $\tilde u_i,\, 1\leqslant i \leqslant 45$, be as in Subsubsection \ref{sss72} and the homomorphism $p_{(i;I)}:P_5\to P_4$ which is defined by \eqref{ct23} for $k=5$. From Lemma \ref{bdm}, we see that $p_{(i;I)}$ passes to a homomorphism from $QP_5(\tilde\omega)$ to $QP_4(\tilde\omega)$. By a routine computation we see that $p_{(i;I)}(\tilde{\mathcal S}) \equiv_{\tilde \omega}0$ for all $(i;I)\in \mathcal N_5$ if and only if 
\begin{equation*}
\tilde{\mathcal S}= \sum_{1171\leqslant j \leqslant 1176}\gamma_j\eta_{j} \equiv_{\tilde\omega} 0.
\end{equation*}
The leading monomial of $\eta_j$ is $\tilde g_j$ with $1171\leqslant j \leqslant 1176$. 

The above facts imply that the monomials $\tilde g_j$ with $1\leqslant j \leqslant 1170$ are admissible, $\dim \widetilde{QP}_5((4)|^2|(3)|^{3}|(1)) = 1395$ and  $\dim \widetilde{SF}_5((4)|^2|(3)|^3|(1))\leqslant 6$. The proposition is proved.
\end{proof}

Combining the above results we obtain the following.

\begin{corls} We have $$\dim\widetilde{QP}_5((4)|^2(3)|^3|(1)) = 1395 \mbox{ and }  1984 \leqslant \dim (QP_5)_{128} \leqslant 1990.$$
\end{corls}

\subsubsection{The admissible monomials of weight vector $(4)|^2|(3)|^{3}|(1)$ in $P_5$}\label{sss72}\
 
\medskip
$B_4((4)|^2|(3)|^{3}|(1)) = B_4^+((4)|^2|(3)|^{3}|(1))$ is the set of 45 monomials $\tilde u_t,\, 1 \leqslant t \leqslant 45$, which are determined as follows:
 
\medskip
\centerline{
\begin{tabular}{lll}
$\tilde u_{1} =  x_1^{3}x_2^{31}x_3^{31}x_4^{63}$& $\tilde u_{2} =  x_1^{3}x_2^{31}x_3^{63}x_4^{31}$& $\tilde u_{3} =  x_1^{3}x_2^{63}x_3^{31}x_4^{31} $\cr  $\tilde u_{4} =  x_1^{7}x_2^{27}x_3^{31}x_4^{63}$& $\tilde u_{5} =  x_1^{7}x_2^{27}x_3^{63}x_4^{31}$& $\tilde u_{6} =  x_1^{7}x_2^{31}x_3^{27}x_4^{63} $\cr  $\tilde u_{7} =  x_1^{7}x_2^{31}x_3^{31}x_4^{59}$& $\tilde u_{8} =  x_1^{7}x_2^{31}x_3^{59}x_4^{31}$& $\tilde u_{9} =  x_1^{7}x_2^{31}x_3^{63}x_4^{27} $\cr  $\tilde u_{10} =  x_1^{7}x_2^{59}x_3^{31}x_4^{31}$& $\tilde u_{11} =  x_1^{7}x_2^{63}x_3^{27}x_4^{31}$& $\tilde u_{12} =  x_1^{7}x_2^{63}x_3^{31}x_4^{27} $\cr  $\tilde u_{13} =  x_1^{15}x_2^{23}x_3^{27}x_4^{63}$& $\tilde u_{14} =  x_1^{15}x_2^{23}x_3^{31}x_4^{59}$& $\tilde u_{15} =  x_1^{15}x_2^{23}x_3^{59}x_4^{31} $\cr  $\tilde u_{16} =  x_1^{15}x_2^{23}x_3^{63}x_4^{27}$& $\tilde u_{17} =  x_1^{15}x_2^{31}x_3^{23}x_4^{59}$& $\tilde u_{18} =  x_1^{15}x_2^{31}x_3^{55}x_4^{27} $\cr  $\tilde u_{19} =  x_1^{15}x_2^{55}x_3^{27}x_4^{31}$& $\tilde u_{20} =  x_1^{15}x_2^{55}x_3^{31}x_4^{27}$& $\tilde u_{21} =  x_1^{15}x_2^{63}x_3^{23}x_4^{27} $\cr  $\tilde u_{22} =  x_1^{31}x_2^{3}x_3^{31}x_4^{63}$& $\tilde u_{23} =  x_1^{31}x_2^{3}x_3^{63}x_4^{31}$& $\tilde u_{24} =  x_1^{31}x_2^{7}x_3^{27}x_4^{63} $\cr
\end{tabular}}

\medskip
It is easy to see that 
$$B(6;(3)|^2|(2)|^3)\cup B(5;(3)|^2|(2)|^3|(1)) \cup B(2;(3)|^5|(1))\subset B_5^+((4)|^2|(3)|^3|(1)).$$  From \cite{su2} we have  $|B_4((3)|^2|(2)|^2)| = 77$ and $B(5;(3)|^2|(2)|^3)$ is the set of 385 monomials $\tilde g_j,\, 1\leqslant j \leqslant 385$, which are determined as follows:

\medskip
\centerline{\begin{tabular}{lll}
1.\, $x_1x_2^{2}x_3^{31}x_4^{31}x_5^{63}$& 2.\, $x_1x_2^{2}x_3^{31}x_4^{63}x_5^{31}$& 3.\, $x_1x_2^{2}x_3^{63}x_4^{31}x_5^{31} $\cr  4.\, $x_1x_2^{3}x_3^{30}x_4^{31}x_5^{63}$& 5.\, $x_1x_2^{3}x_3^{30}x_4^{63}x_5^{31}$& 6.\, $x_1x_2^{3}x_3^{31}x_4^{30}x_5^{63} $\cr  7.\, $x_1x_2^{3}x_3^{31}x_4^{63}x_5^{30}$& 8.\, $x_1x_2^{3}x_3^{63}x_4^{30}x_5^{31}$& 9.\, $x_1x_2^{3}x_3^{63}x_4^{31}x_5^{30} $\cr  10.\, $x_1x_2^{6}x_3^{27}x_4^{31}x_5^{63}$& 11.\, $x_1x_2^{6}x_3^{27}x_4^{63}x_5^{31}$& 12.\, $x_1x_2^{6}x_3^{31}x_4^{27}x_5^{63} $\cr  13.\, $x_1x_2^{6}x_3^{31}x_4^{63}x_5^{27}$& 14.\, $x_1x_2^{6}x_3^{63}x_4^{27}x_5^{31}$& 15.\, $x_1x_2^{6}x_3^{63}x_4^{31}x_5^{27} $\cr  16.\, $x_1x_2^{7}x_3^{26}x_4^{31}x_5^{63}$& 17.\, $x_1x_2^{7}x_3^{26}x_4^{63}x_5^{31}$& 18.\, $x_1x_2^{7}x_3^{27}x_4^{30}x_5^{63} $\cr  19.\, $x_1x_2^{7}x_3^{27}x_4^{63}x_5^{30}$& 20.\, $x_1x_2^{7}x_3^{30}x_4^{27}x_5^{63}$& 21.\, $x_1x_2^{7}x_3^{30}x_4^{63}x_5^{27} $\cr  22.\, $x_1x_2^{7}x_3^{31}x_4^{26}x_5^{63}$& 23.\, $x_1x_2^{7}x_3^{31}x_4^{63}x_5^{26}$& 24.\, $x_1x_2^{7}x_3^{63}x_4^{26}x_5^{31} $\cr
\end{tabular}}
\centerline{\begin{tabular}{lll}  
145.\, $x_1^{3}x_2^{29}x_3^{63}x_4^{7}x_5^{26}$& 146.\, $x_1^{3}x_2^{29}x_3^{63}x_4^{30}x_5^{3}$& 147.\, $x_1^{3}x_2^{29}x_3^{63}x_4^{31}x_5^{2} $\cr  148.\, $x_1^{3}x_2^{31}x_3x_4^{30}x_5^{63}$& 149.\, $x_1^{3}x_2^{31}x_3x_4^{63}x_5^{30}$& 150.\, $x_1^{3}x_2^{31}x_3^{3}x_4^{28}x_5^{63} $\cr  151.\, $x_1^{3}x_2^{31}x_3^{3}x_4^{63}x_5^{28}$& 152.\, $x_1^{3}x_2^{31}x_3^{5}x_4^{26}x_5^{63}$& 153.\, $x_1^{3}x_2^{31}x_3^{5}x_4^{63}x_5^{26} $\cr  154.\, $x_1^{3}x_2^{31}x_3^{29}x_4^{2}x_5^{63}$& 155.\, $x_1^{3}x_2^{31}x_3^{29}x_4^{63}x_5^{2}$& 156.\, $x_1^{3}x_2^{31}x_3^{63}x_4x_5^{30} $\cr  157.\, $x_1^{3}x_2^{31}x_3^{63}x_4^{3}x_5^{28}$& 158.\, $x_1^{3}x_2^{31}x_3^{63}x_4^{5}x_5^{26}$& 159.\, $x_1^{3}x_2^{31}x_3^{63}x_4^{29}x_5^{2} $\cr  160.\, $x_1^{3}x_2^{63}x_3x_4^{30}x_5^{31}$& 161.\, $x_1^{3}x_2^{63}x_3x_4^{31}x_5^{30}$& 162.\, $x_1^{3}x_2^{63}x_3^{3}x_4^{28}x_5^{31} $\cr  163.\, $x_1^{3}x_2^{63}x_3^{3}x_4^{29}x_5^{30}$& 164.\, $x_1^{3}x_2^{63}x_3^{3}x_4^{31}x_5^{28}$& 165.\, $x_1^{3}x_2^{63}x_3^{5}x_4^{26}x_5^{31} $\cr  166.\, $x_1^{3}x_2^{63}x_3^{5}x_4^{27}x_5^{30}$& 167.\, $x_1^{3}x_2^{63}x_3^{5}x_4^{30}x_5^{27}$& 168.\, $x_1^{3}x_2^{63}x_3^{5}x_4^{31}x_5^{26} $\cr
\end{tabular}}
\centerline{\begin{tabular}{lll}
289.\, $x_1^{31}x_2^{3}x_3^{63}x_4^{5}x_5^{26}$& 290.\, $x_1^{31}x_2^{3}x_3^{63}x_4^{29}x_5^{2}$& 291.\, $x_1^{31}x_2^{7}x_3x_4^{26}x_5^{63} $\cr  292.\, $x_1^{31}x_2^{7}x_3x_4^{63}x_5^{26}$& 293.\, $x_1^{31}x_2^{7}x_3^{63}x_4x_5^{26}$& 294.\, $x_1^{31}x_2^{31}x_3x_4^{2}x_5^{63} $\cr  295.\, $x_1^{31}x_2^{31}x_3x_4^{63}x_5^{2}$& 296.\, $x_1^{31}x_2^{31}x_3^{63}x_4x_5^{2}$& 297.\, $x_1^{31}x_2^{63}x_3x_4^{2}x_5^{31} $\cr  298.\, $x_1^{31}x_2^{63}x_3x_4^{3}x_5^{30}$& 299.\, $x_1^{31}x_2^{63}x_3x_4^{6}x_5^{27}$& 300.\, $x_1^{31}x_2^{63}x_3x_4^{7}x_5^{26} $\cr  301.\, $x_1^{31}x_2^{63}x_3x_4^{30}x_5^{3}$& 302.\, $x_1^{31}x_2^{63}x_3x_4^{31}x_5^{2}$& 303.\, $x_1^{31}x_2^{63}x_3^{3}x_4x_5^{30} $\cr  304.\, $x_1^{31}x_2^{63}x_3^{3}x_4^{3}x_5^{28}$& 305.\, $x_1^{31}x_2^{63}x_3^{3}x_4^{5}x_5^{26}$& 306.\, $x_1^{31}x_2^{63}x_3^{3}x_4^{29}x_5^{2} $\cr  307.\, $x_1^{31}x_2^{63}x_3^{7}x_4x_5^{26}$& 308.\, $x_1^{31}x_2^{63}x_3^{31}x_4x_5^{2}$& 309.\, $x_1^{63}x_2x_3^{2}x_4^{31}x_5^{31} $\cr  310.\, $x_1^{63}x_2x_3^{3}x_4^{30}x_5^{31}$& 311.\, $x_1^{63}x_2x_3^{3}x_4^{31}x_5^{30}$& 312.\, $x_1^{63}x_2x_3^{6}x_4^{27}x_5^{31} $\cr
\end{tabular}}

\medskip
$B(5;(3)|^2|(2)|^3|(1))\cup B(2;(3)|^5|(1))\setminus B(6,(3)|^2|(2)|^3)$ is the set of 455 monomials $\tilde g_j,\, 386\leqslant j \leqslant 840$, which are determined as follows:

\medskip
\centerline{\begin{tabular}{lll}
386.\, $x_1x_2^{3}x_3^{31}x_4^{31}x_5^{62}$& 387.\, $x_1x_2^{3}x_3^{31}x_4^{62}x_5^{31}$& 388.\, $x_1x_2^{3}x_3^{62}x_4^{31}x_5^{31} $\cr  389.\, $x_1x_2^{6}x_3^{31}x_4^{31}x_5^{59}$& 390.\, $x_1x_2^{6}x_3^{31}x_4^{59}x_5^{31}$& 391.\, $x_1x_2^{6}x_3^{59}x_4^{31}x_5^{31} $\cr  392.\, $x_1x_2^{7}x_3^{27}x_4^{31}x_5^{62}$& 393.\, $x_1x_2^{7}x_3^{27}x_4^{62}x_5^{31}$& 394.\, $x_1x_2^{7}x_3^{30}x_4^{31}x_5^{59} $\cr  395.\, $x_1x_2^{7}x_3^{30}x_4^{59}x_5^{31}$& 396.\, $x_1x_2^{7}x_3^{31}x_4^{27}x_5^{62}$& 397.\, $x_1x_2^{7}x_3^{31}x_4^{30}x_5^{59} $\cr  398.\, $x_1x_2^{7}x_3^{31}x_4^{31}x_5^{58}$& 399.\, $x_1x_2^{7}x_3^{31}x_4^{58}x_5^{31}$& 400.\, $x_1x_2^{7}x_3^{31}x_4^{59}x_5^{30} $\cr  401.\, $x_1x_2^{7}x_3^{31}x_4^{62}x_5^{27}$& 402.\, $x_1x_2^{7}x_3^{58}x_4^{31}x_5^{31}$& 403.\, $x_1x_2^{7}x_3^{59}x_4^{30}x_5^{31} $\cr  404.\, $x_1x_2^{7}x_3^{59}x_4^{31}x_5^{30}$& 405.\, $x_1x_2^{7}x_3^{62}x_4^{27}x_5^{31}$& 406.\, $x_1x_2^{7}x_3^{62}x_4^{31}x_5^{27} $\cr  407.\, $x_1x_2^{14}x_3^{23}x_4^{31}x_5^{59}$& 408.\, $x_1x_2^{14}x_3^{23}x_4^{59}x_5^{31}$& 409.\, $x_1x_2^{14}x_3^{31}x_4^{23}x_5^{59} $\cr  410.\, $x_1x_2^{14}x_3^{31}x_4^{55}x_5^{27}$& 411.\, $x_1x_2^{14}x_3^{55}x_4^{27}x_5^{31}$& 412.\, $x_1x_2^{14}x_3^{55}x_4^{31}x_5^{27} $\cr  413.\, $x_1x_2^{15}x_3^{22}x_4^{31}x_5^{59}$& 414.\, $x_1x_2^{15}x_3^{22}x_4^{59}x_5^{31}$& 415.\, $x_1x_2^{15}x_3^{23}x_4^{31}x_5^{58} $\cr  416.\, $x_1x_2^{15}x_3^{23}x_4^{58}x_5^{31}$& 417.\, $x_1x_2^{15}x_3^{31}x_4^{22}x_5^{59}$& 418.\, $x_1x_2^{15}x_3^{31}x_4^{23}x_5^{58} $\cr  419.\, $x_1x_2^{15}x_3^{31}x_4^{54}x_5^{27}$& 420.\, $x_1x_2^{15}x_3^{31}x_4^{55}x_5^{26}$& 421.\, $x_1x_2^{15}x_3^{54}x_4^{27}x_5^{31} $\cr  422.\, $x_1x_2^{15}x_3^{54}x_4^{31}x_5^{27}$& 423.\, $x_1x_2^{15}x_3^{55}x_4^{26}x_5^{31}$& 424.\, $x_1x_2^{15}x_3^{55}x_4^{31}x_5^{26} $\cr  
\end{tabular}}
\centerline{\begin{tabular}{lll}
425.\, $x_1x_2^{30}x_3^{7}x_4^{31}x_5^{59}$& 426.\, $x_1x_2^{30}x_3^{7}x_4^{59}x_5^{31}$& 427.\, $x_1x_2^{30}x_3^{31}x_4^{7}x_5^{59} $\cr  428.\, $x_1x_2^{31}x_3^{3}x_4^{31}x_5^{62}$& 429.\, $x_1x_2^{31}x_3^{3}x_4^{62}x_5^{31}$& 430.\, $x_1x_2^{31}x_3^{6}x_4^{31}x_5^{59} $\cr  431.\, $x_1x_2^{31}x_3^{6}x_4^{59}x_5^{31}$& 432.\, $x_1x_2^{31}x_3^{7}x_4^{27}x_5^{62}$& 433.\, $x_1x_2^{31}x_3^{7}x_4^{30}x_5^{59} $\cr  434.\, $x_1x_2^{31}x_3^{7}x_4^{31}x_5^{58}$& 435.\, $x_1x_2^{31}x_3^{7}x_4^{58}x_5^{31}$& 436.\, $x_1x_2^{31}x_3^{7}x_4^{59}x_5^{30} $\cr  437.\, $x_1x_2^{31}x_3^{7}x_4^{62}x_5^{27}$& 438.\, $x_1x_2^{31}x_3^{14}x_4^{23}x_5^{59}$& 439.\, $x_1x_2^{31}x_3^{14}x_4^{55}x_5^{27} $\cr  440.\, $x_1x_2^{31}x_3^{15}x_4^{22}x_5^{59}$& 441.\, $x_1x_2^{31}x_3^{15}x_4^{23}x_5^{58}$& 442.\, $x_1x_2^{31}x_3^{15}x_4^{54}x_5^{27} $\cr  443.\, $x_1x_2^{31}x_3^{15}x_4^{55}x_5^{26}$& 444.\, $x_1x_2^{31}x_3^{30}x_4^{7}x_5^{59}$& 445.\, $x_1x_2^{31}x_3^{31}x_4^{3}x_5^{62} $\cr  446.\, $x_1x_2^{31}x_3^{31}x_4^{6}x_5^{59}$& 447.\, $x_1x_2^{31}x_3^{31}x_4^{7}x_5^{58}$& 448.\, $x_1x_2^{31}x_3^{31}x_4^{62}x_5^{3} $\cr
\end{tabular}}
\centerline{\begin{tabular}{lll}
569.\, $x_1^{3}x_2^{31}x_3^{61}x_4^{3}x_5^{30}$& 570.\, $x_1^{3}x_2^{31}x_3^{61}x_4^{6}x_5^{27}$& 571.\, $x_1^{3}x_2^{31}x_3^{61}x_4^{7}x_5^{26} $\cr  572.\, $x_1^{3}x_2^{31}x_3^{61}x_4^{30}x_5^{3}$& 573.\, $x_1^{3}x_2^{31}x_3^{61}x_4^{31}x_5^{2}$& 574.\, $x_1^{3}x_2^{61}x_3^{2}x_4^{31}x_5^{31} $\cr  575.\, $x_1^{3}x_2^{61}x_3^{3}x_4^{30}x_5^{31}$& 576.\, $x_1^{3}x_2^{61}x_3^{3}x_4^{31}x_5^{30}$& 577.\, $x_1^{3}x_2^{61}x_3^{6}x_4^{27}x_5^{31} $\cr  578.\, $x_1^{3}x_2^{61}x_3^{6}x_4^{31}x_5^{27}$& 579.\, $x_1^{3}x_2^{61}x_3^{7}x_4^{26}x_5^{31}$& 580.\, $x_1^{3}x_2^{61}x_3^{7}x_4^{31}x_5^{26} $\cr  581.\, $x_1^{3}x_2^{61}x_3^{30}x_4^{3}x_5^{31}$& 582.\, $x_1^{3}x_2^{61}x_3^{30}x_4^{31}x_5^{3}$& 583.\, $x_1^{3}x_2^{61}x_3^{31}x_4^{2}x_5^{31} $\cr  584.\, $x_1^{3}x_2^{61}x_3^{31}x_4^{3}x_5^{30}$& 585.\, $x_1^{3}x_2^{61}x_3^{31}x_4^{6}x_5^{27}$& 586.\, $x_1^{3}x_2^{61}x_3^{31}x_4^{7}x_5^{26} $\cr  587.\, $x_1^{3}x_2^{61}x_3^{31}x_4^{30}x_5^{3}$& 588.\, $x_1^{3}x_2^{61}x_3^{31}x_4^{31}x_5^{2}$& 589.\, $x_1^{7}x_2x_3^{27}x_4^{31}x_5^{62} $\cr  590.\, $x_1^{7}x_2x_3^{27}x_4^{62}x_5^{31}$& 591.\, $x_1^{7}x_2x_3^{30}x_4^{31}x_5^{59}$& 592.\, $x_1^{7}x_2x_3^{30}x_4^{59}x_5^{31} $\cr
\end{tabular}}
\centerline{\begin{tabular}{lll}
710.\, $x_1^{15}x_2x_3^{31}x_4^{55}x_5^{26}$& 711.\, $x_1^{15}x_2x_3^{54}x_4^{27}x_5^{31}$& 712.\, $x_1^{15}x_2x_3^{54}x_4^{31}x_5^{27} $\cr  713.\, $x_1^{15}x_2x_3^{55}x_4^{26}x_5^{31}$& 714.\, $x_1^{15}x_2x_3^{55}x_4^{31}x_5^{26}$& 715.\, $x_1^{15}x_2^{3}x_3^{21}x_4^{31}x_5^{58} $\cr  716.\, $x_1^{15}x_2^{3}x_3^{21}x_4^{58}x_5^{31}$& 717.\, $x_1^{15}x_2^{3}x_3^{31}x_4^{21}x_5^{58}$& 718.\, $x_1^{15}x_2^{3}x_3^{31}x_4^{53}x_5^{26} $\cr  719.\, $x_1^{15}x_2^{3}x_3^{53}x_4^{26}x_5^{31}$& 720.\, $x_1^{15}x_2^{3}x_3^{53}x_4^{31}x_5^{26}$& 721.\, $x_1^{15}x_2^{23}x_3x_4^{31}x_5^{58} $\cr  722.\, $x_1^{15}x_2^{23}x_3x_4^{58}x_5^{31}$& 723.\, $x_1^{15}x_2^{23}x_3^{31}x_4x_5^{58}$& 724.\, $x_1^{15}x_2^{31}x_3x_4^{22}x_5^{59} $\cr  725.\, $x_1^{15}x_2^{31}x_3x_4^{23}x_5^{58}$& 726.\, $x_1^{15}x_2^{31}x_3x_4^{54}x_5^{27}$& 727.\, $x_1^{15}x_2^{31}x_3x_4^{55}x_5^{26} $\cr  728.\, $x_1^{15}x_2^{31}x_3^{3}x_4^{21}x_5^{58}$& 729.\, $x_1^{15}x_2^{31}x_3^{3}x_4^{53}x_5^{26}$& 730.\, $x_1^{15}x_2^{31}x_3^{23}x_4x_5^{58} $\cr  731.\, $x_1^{15}x_2^{31}x_3^{55}x_4x_5^{26}$& 732.\, $x_1^{15}x_2^{55}x_3x_4^{26}x_5^{31}$& 733.\, $x_1^{15}x_2^{55}x_3x_4^{31}x_5^{26} $\cr
\end{tabular}}

\medskip

We have 
$$B_5((4)|^2|(3)|^3|(1) \subset  B(6,(3)|^2|(2)|^3)\cup B(5;(3)|^2|(2)|^3|(1))\cup B(2;(3)|^5|(1))\cup\widetilde{\mathcal C},$$
where $\widetilde{\mathcal C}$ is the set of 335 monomials $\tilde g_j,\, 841\leqslant j \leqslant 1175$, which are determined as follows:

\medskip
\centerline{\begin{tabular}{lll}
841.\, $x_1x_2^{15}x_3^{23}x_4^{27}x_5^{62}$& 842.\, $x_1x_2^{15}x_3^{23}x_4^{30}x_5^{59}$& 843.\, $x_1x_2^{15}x_3^{23}x_4^{59}x_5^{30} $\cr  844.\, $x_1x_2^{15}x_3^{23}x_4^{62}x_5^{27}$& 845.\, $x_1x_2^{15}x_3^{30}x_4^{23}x_5^{59}$& 846.\, $x_1x_2^{15}x_3^{30}x_4^{55}x_5^{27} $\cr  847.\, $x_1x_2^{15}x_3^{55}x_4^{27}x_5^{30}$& 848.\, $x_1x_2^{15}x_3^{55}x_4^{30}x_5^{27}$& 849.\, $x_1x_2^{15}x_3^{62}x_4^{23}x_5^{27} $\cr  850.\, $x_1x_2^{30}x_3^{15}x_4^{23}x_5^{59}$& 851.\, $x_1x_2^{30}x_3^{15}x_4^{55}x_5^{27}$& 852.\, $x_1x_2^{30}x_3^{31}x_4^{31}x_5^{35} $\cr  853.\, $x_1x_2^{30}x_3^{31}x_4^{39}x_5^{27}$& 854.\, $x_1x_2^{30}x_3^{47}x_4^{23}x_5^{27}$& 855.\, $x_1x_2^{31}x_3^{30}x_4^{31}x_5^{35} $\cr  856.\, $x_1x_2^{31}x_3^{30}x_4^{39}x_5^{27}$& 857.\, $x_1x_2^{31}x_3^{31}x_4^{30}x_5^{35}$& 858.\, $x_1x_2^{31}x_3^{31}x_4^{31}x_5^{34} $\cr  859.\, $x_1x_2^{31}x_3^{31}x_4^{38}x_5^{27}$& 860.\, $x_1x_2^{31}x_3^{31}x_4^{39}x_5^{26}$& 861.\, $x_1x_2^{31}x_3^{46}x_4^{23}x_5^{27} $\cr  862.\, $x_1x_2^{31}x_3^{47}x_4^{22}x_5^{27}$& 863.\, $x_1x_2^{31}x_3^{47}x_4^{23}x_5^{26}$& 864.\, $x_1x_2^{62}x_3^{15}x_4^{23}x_5^{27} $\cr
\end{tabular}}
\centerline{\begin{tabular}{lll}
979.\, $x_1^{7}x_2^{11}x_3^{55}x_4^{29}x_5^{26}$& 980.\, $x_1^{7}x_2^{11}x_3^{61}x_4^{22}x_5^{27}$& 981.\, $x_1^{7}x_2^{11}x_3^{61}x_4^{23}x_5^{26} $\cr  982.\, $x_1^{7}x_2^{15}x_3^{19}x_4^{25}x_5^{62}$& 983.\, $x_1^{7}x_2^{15}x_3^{19}x_4^{27}x_5^{60}$& 984.\, $x_1^{7}x_2^{15}x_3^{19}x_4^{29}x_5^{58} $\cr  985.\, $x_1^{7}x_2^{15}x_3^{19}x_4^{57}x_5^{30}$& 986.\, $x_1^{7}x_2^{15}x_3^{19}x_4^{59}x_5^{28}$& 987.\, $x_1^{7}x_2^{15}x_3^{19}x_4^{61}x_5^{26} $\cr  988.\, $x_1^{7}x_2^{15}x_3^{23}x_4^{24}x_5^{59}$& 989.\, $x_1^{7}x_2^{15}x_3^{23}x_4^{25}x_5^{58}$& 990.\, $x_1^{7}x_2^{15}x_3^{23}x_4^{56}x_5^{27} $\cr  991.\, $x_1^{7}x_2^{15}x_3^{23}x_4^{57}x_5^{26}$& 992.\, $x_1^{7}x_2^{15}x_3^{27}x_4^{21}x_5^{58}$& 993.\, $x_1^{7}x_2^{15}x_3^{27}x_4^{53}x_5^{26} $\cr  994.\, $x_1^{7}x_2^{15}x_3^{51}x_4^{25}x_5^{30}$& 995.\, $x_1^{7}x_2^{15}x_3^{51}x_4^{27}x_5^{28}$& 996.\, $x_1^{7}x_2^{15}x_3^{51}x_4^{29}x_5^{26} $\cr  997.\, $x_1^{7}x_2^{15}x_3^{55}x_4^{24}x_5^{27}$& 998.\, $x_1^{7}x_2^{15}x_3^{55}x_4^{25}x_5^{26}$& 999.\, $x_1^{7}x_2^{15}x_3^{59}x_4^{21}x_5^{26} $\cr  1000.\, $x_1^{7}x_2^{27}x_3^{5}x_4^{27}x_5^{62}$& 1001.\, $x_1^{7}x_2^{27}x_3^{5}x_4^{30}x_5^{59}$& 1002.\, $x_1^{7}x_2^{27}x_3^{5}x_4^{59}x_5^{30} $\cr
\end{tabular}}
\centerline{\begin{tabular}{lll}
1126.\, $x_1^{15}x_2^{55}x_3x_4^{27}x_5^{30}$& 1127.\, $x_1^{15}x_2^{55}x_3x_4^{30}x_5^{27}$& 1128.\, $x_1^{15}x_2^{55}x_3^{3}x_4^{25}x_5^{30} $\cr  1129.\, $x_1^{15}x_2^{55}x_3^{3}x_4^{27}x_5^{28}$& 1130.\, $x_1^{15}x_2^{55}x_3^{3}x_4^{29}x_5^{26}$& 1131.\, $x_1^{15}x_2^{55}x_3^{7}x_4^{24}x_5^{27} $\cr  1132.\, $x_1^{15}x_2^{55}x_3^{7}x_4^{25}x_5^{26}$& 1133.\, $x_1^{15}x_2^{55}x_3^{11}x_4^{21}x_5^{26}$& 1134.\, $x_1^{15}x_2^{55}x_3^{27}x_4x_5^{30} $\cr  1135.\, $x_1^{15}x_2^{55}x_3^{27}x_4^{3}x_5^{28}$& 1136.\, $x_1^{15}x_2^{55}x_3^{27}x_4^{5}x_5^{26}$& 1137.\, $x_1^{15}x_2^{55}x_3^{27}x_4^{29}x_5^{2} $\cr  1138.\, $x_1^{31}x_2x_3^{30}x_4^{31}x_5^{35}$& 1139.\, $x_1^{31}x_2x_3^{30}x_4^{39}x_5^{27}$& 1140.\, $x_1^{31}x_2x_3^{31}x_4^{30}x_5^{35} $\cr  1141.\, $x_1^{31}x_2x_3^{31}x_4^{31}x_5^{34}$& 1142.\, $x_1^{31}x_2x_3^{31}x_4^{38}x_5^{27}$& 1143.\, $x_1^{31}x_2x_3^{31}x_4^{39}x_5^{26} $\cr  1144.\, $x_1^{31}x_2x_3^{46}x_4^{23}x_5^{27}$& 1145.\, $x_1^{31}x_2x_3^{47}x_4^{22}x_5^{27}$& 1146.\, $x_1^{31}x_2x_3^{47}x_4^{23}x_5^{26} $\cr  1147.\, $x_1^{31}x_2^{3}x_3^{29}x_4^{30}x_5^{35}$& 1148.\, $x_1^{31}x_2^{3}x_3^{29}x_4^{31}x_5^{34}$& 1149.\, $x_1^{31}x_2^{3}x_3^{29}x_4^{38}x_5^{27} $\cr
\end{tabular}}
\centerline{\begin{tabular}{lll}
\end{tabular}}

 \subsubsection{Generators of $\widetilde {SF}_5((4)|^2|(3)|^3|(1))$}\label{sss73}
 
\begin{align*}
\eta_{1171} &= x_1x_2^{15}x_3^{23}x_4^{27}x_5^{62} + x_1x_2^{15}x_3^{30}x_4^{23}x_5^{59} + x_1^{3}x_2^{7}x_3^{29}x_4^{27}x_5^{62} + x_1^{3}x_2^{7}x_3^{29}x_4^{30}x_5^{59}\\
&\quad + x_1^{3}x_2^{13}x_3^{30}x_4^{23}x_5^{59} + x_1^{3}x_2^{15}x_3^{21}x_4^{27}x_5^{62} + x_1^{3}x_2^{15}x_3^{21}x_4^{30}x_5^{59} + x_1^{3}x_2^{15}x_3^{23}x_4^{27}x_5^{60}\\
&\quad + x_1^{7}x_2^{3}x_3^{27}x_4^{29}x_5^{62} + x_1^{7}x_2^{7}x_3^{25}x_4^{30}x_5^{59} + x_1^{7}x_2^{7}x_3^{27}x_4^{27}x_5^{60} + x_1^{7}x_2^{9}x_3^{23}x_4^{27}x_5^{62}\\
&\quad + x_1^{7}x_2^{11}x_3^{21}x_4^{27}x_5^{62} + x_1^{7}x_2^{11}x_3^{21}x_4^{30}x_5^{59} + x_1^{7}x_2^{11}x_3^{23}x_4^{27}x_5^{60} + x_1^{15}x_2x_3^{23}x_4^{27}x_5^{62}\\
&\quad + x_1^{15}x_2x_3^{30}x_4^{23}x_5^{59} + x_1^{15}x_2^{3}x_3^{21}x_4^{27}x_5^{62} + x_1^{15}x_2^{3}x_3^{21}x_4^{30}x_5^{59} + x_1^{15}x_2^{3}x_3^{23}x_4^{27}x_5^{60}\\
&\quad + x_1^{15}x_2^{15}x_3^{16}x_4^{23}x_5^{59},\\
\eta_{1172} &= x_1x_2^{15}x_3^{23}x_4^{59}x_5^{30} + x_1x_2^{15}x_3^{30}x_4^{55}x_5^{27} + x_1^{3}x_2^{7}x_3^{29}x_4^{59}x_5^{30} + x_1^{3}x_2^{7}x_3^{29}x_4^{62}x_5^{27}\\
&\quad + x_1^{3}x_2^{13}x_3^{30}x_4^{55}x_5^{27} + x_1^{3}x_2^{15}x_3^{21}x_4^{59}x_5^{30} + x_1^{3}x_2^{15}x_3^{21}x_4^{62}x_5^{27} + x_1^{3}x_2^{15}x_3^{23}x_4^{59}x_5^{28}\\
&\quad + x_1^{7}x_2^{3}x_3^{27}x_4^{61}x_5^{30} + x_1^{7}x_2^{7}x_3^{25}x_4^{62}x_5^{27} + x_1^{7}x_2^{7}x_3^{27}x_4^{59}x_5^{28} + x_1^{7}x_2^{9}x_3^{23}x_4^{59}x_5^{30}\\
&\quad + x_1^{7}x_2^{11}x_3^{21}x_4^{59}x_5^{30} + x_1^{7}x_2^{11}x_3^{21}x_4^{62}x_5^{27} + x_1^{7}x_2^{11}x_3^{23}x_4^{59}x_5^{28} + x_1^{15}x_2x_3^{23}x_4^{59}x_5^{30}\\
&\quad + x_1^{15}x_2x_3^{30}x_4^{55}x_5^{27} + x_1^{15}x_2^{3}x_3^{21}x_4^{59}x_5^{30} + x_1^{15}x_2^{3}x_3^{21}x_4^{62}x_5^{27} + x_1^{15}x_2^{3}x_3^{23}x_4^{59}x_5^{28}\\
&\quad + x_1^{15}x_2^{15}x_3^{16}x_4^{55}x_5^{27},\\
\eta_{1173} &= x_1^{3}x_2^{7}x_3^{29}x_4^{27}x_5^{62} + x_1^{3}x_2^{15}x_3^{21}x_4^{27}x_5^{62} + x_1^{7}x_2^{7}x_3^{25}x_4^{27}x_5^{62} + x_1^{7}x_2^{15}x_3^{17}x_4^{27}x_5^{62},\\
\eta_{1174} &= x_1^{7}x_2^{3}x_3^{29}x_4^{27}x_5^{62} + x_1^{7}x_2^{7}x_3^{25}x_4^{27}x_5^{62} + x_1^{15}x_2^{3}x_3^{21}x_4^{27}x_5^{62} + x_1^{15}x_2^{7}x_3^{17}x_4^{27}x_5^{62},\\
\eta_{1175} &= x_1^{7}x_2^{7}x_3^{25}x_4^{27}x_5^{62} + x_1^{15}x_2^{15}x_3^{17}x_4^{19}x_5^{62},\\
\eta_{1176} &= x_1^{7}x_2^{7}x_3^{27}x_4^{25}x_5^{62} + x_1^{15}x_2^{15}x_3^{19}x_4^{17}x_5^{62}.
\end{align*}

 {}

\end{document}